\documentclass[11pt]{amsart}
\usepackage{a4wide}
\usepackage[T1]{fontenc}
\usepackage{amssymb,amsmath,amsthm,latexsym}
\usepackage{mathrsfs}
\usepackage[usenames,dvipsnames]{color}
\usepackage{euscript}
\usepackage{graphicx}
\usepackage{mdwlist}
\usepackage{enumerate}
\usepackage{mathtools,dsfont,wasysym}

\usepackage{bbm}
\usepackage{stmaryrd}
\usepackage{centernot}
\usepackage{todonotes}
\usepackage{hyperref}
\numberwithin{equation}{section}
\usepackage[toc]{appendix}

\makeatletter
\@namedef{subjclassname@2020}{\textup{2020} Mathematics Subject Classification}

\renewcommand\subsubsection{\@startsection{subsubsection}{3}%
  \z@{.5\linespacing\@plus.7\linespacing}{-.5em}%
  {\normalfont\bfseries}}

\makeatother

\usepackage{indentfirst}
\usepackage{float}

\theoremstyle{plain}
\newtheorem{thm}{Theorem}[section]
\newtheorem{cor}[thm]{Corollary}
\newtheorem{lem}[thm]{Lemma}
\newtheorem{prop}[thm]{Proposition}

\newtheorem{rem}[thm]{Remark}
\newtheorem{defn}[thm]{Definition}
\newtheorem{exm}[thm]{Example}

\theoremstyle{remark}
\newtheorem{rmk}[thm]{Remark}
\numberwithin{equation}{section}

\theoremstyle{plain}
\newtheorem{assumpt}{Assumption}
\renewcommand{\theassumpt}{$\mathbf{A_{\arabic{assumpt}}}$}

\newtheorem{cond}{Condition}
\renewcommand{\thecond}{$\mathbf{B}$}

\begin{document}
\title[]{Fractional porous medium equation on manifolds with nonnegative Ricci curvature: existence of solutions and smoothing effects via potential methods}

\author[D. von Criegern]{Dorothea-Enrica von Criegern}
\address[D. von Criegern]{Politecnico di Milano, Dipartimento di Matematica, Piazza Leonardo da Vinci 32, 20133 Milano, Italy}
\email{dorotheaenrica.von@polimi.it}

\author[G.~Grillo]{Gabriele Grillo}
\address[G.~Grillo]{Politecnico di Milano, Dipartimento di Matematica, Piazza Leonardo da Vinci 32, 20133 Milano, Italy}
\email{gabriele.grillo@polimi.it}

\author[D.D.~Monticelli]{Dario D. Monticelli}
\address[D.D.~Monticelli]{Politecnico di Milano, Dipartimento di Matematica, Piazza Leonardo da Vinci 32, 20133 Milano, Italy}
\email{dario.monticelli@polimi.it}



\begin{abstract}
We study the fractional porous medium equation on complete noncompact Riemannian manifolds with nonnegative Ricci curvature for $s\in(0,1]$ and $n>2s$. Assuming that the manifold is $s$-nonparabolic, so that the fractional Laplacian admits a suitable positive minimal Green function, we use this Green function to introduce a natural weighted space of initial data, strictly larger than $L^1$. This leads to a weak dual, or potential, formulation of the equation, for which we prove existence for nonnegative initial data in the weighted space. We then establish quantitative local smoothing estimates for initial data
in either $L^1$ or the Green-weighted space. Under additional
noncollapsing and uniform volume-growth assumptions, we obtain global
smoothing estimates. We also show that, even in $\mathbb{R}^n$, data in the Green-weighted space need not generate bounded solutions unless a suitable uniform weighted integrability condition is imposed. The short- and long-time
behaviors predicted by our estimates are shown to be optimal in
appropriate senses. When $s\in(0,1/2)$, the results require only $\operatorname{Ric}\geq0$, together with a noncollapsing assumption where needed. Our estimates also cover the case $s=1$, several of them being new in that setting. Finally, we extend the approach to more general filtration equations.
\end{abstract}
\maketitle
\section{Introduction}

Let $(M,g)$ be a complete, noncompact, $n$-dimensional Riemannian manifold
without boundary. We study solutions to the \textit{Fractional Porous Medium
Equation} (FPME)
\begin{equation}\label{maineq}
\begin{cases}
\partial_t u + (-\Delta)^s(u^m)=0,
    & (t,x)\in(0,\infty)\times M,\\
u(0,x)=u_0(x),
    & x\in M,
\end{cases}
\end{equation}
where $m>1$, $s\in(0,1]$, $n>2s$ and $(-\Delta)^s$ denotes the spectral $s$-power of the
Laplace--Beltrami operator $-\Delta$ associated with $g$; see
Section~\ref{S2} for the precise definition. Our main interest is the
fractional range $s\in(0,1)$, although the same methods also cover $s=1$,
with some results being new even in the non-fractional case.

The passage from $s=1$ to $0<s<1$ introduces nonlocal diffusion, accounting
for long-range interactions and anomalous diffusion. On a complete Riemannian
manifold, the operator $(-\Delta)^s$ also admits a probabilistic interpretation
through subordination of Brownian motion by an independent $s$-stable
subordinator; see \cite{Jacob2001}.



Throughout, we assume $\operatorname{Ric}\geq0$ and that $M$ is
$s$-nonparabolic, i.e., that the minimal fractional Green function $G_M^s$
exists. Some of our results additionally require a noncollapsing condition
(Assumption~\ref{A1}) and a uniform volume-growth condition
(Assumption~\ref{A2}). We emphasize that, when $0<s<1/2$,
$s$-nonparabolicity is automatic and Assumption~\ref{A2} also holds, with
an explicit choice of the function $f$.

We consider nonnegative initial data in the Green-weighted space
$L^1_{G_M^s,o}(M)$, where $o\in M$ is a fixed reference point; see
Section~\ref{weakdual}. This space is tailored to the potential
methods used throughout the paper and is strictly larger than $L^1(M)$,
defined with respect to the Riemannian measure.

The FPME in the Euclidean space $\mathbb R^N$ was introduced and studied in the fundamental works \cite{DePabloQuirosRodriguezVazquez2011, DePabloQuirosRodriguezVazquez2012},
where the extension method for the fractional Laplacian introduced in
\cite{CS} plays an important role. When dealing with other kinds of nonlocal
operators, in particular with the various versions of the fractional
Laplacian on bounded Euclidean domains, different techniques are necessary. 
In that setting, the spectral, restricted, and censored fractional Laplacian
can all be associated with Dirichlet-type boundary conditions, while defining
genuinely different operators; see, e.g.,
\cite{BonforteFigalliVazquez2018,BonforteSireVazquez2014,
BonforteVazquez2014,BonforteVazquez2015,BonforteVazquez2016}.
Their associated Green functions, in particular their local and boundary
behavior, provide a useful way to distinguish these realizations. This illustrates the central role played by the Green function
in the analysis of nonlocal diffusion equations, a viewpoint that is also
fundamental to our approach on manifolds.


Fractional nonlinear evolution equations on Riemannian manifolds have been considered only in a few recent works. The case of hyperbolic space was studied in
\cite{BerchioBonforteGangulyGrillo2020}, where the almost explicit form of
the heat kernel and the corresponding fractional Green function, obtained
for instance by subordination, play a central role. More recently,
\cite{BerchioBonforteGrilloMuratori2024} considered general
Cartan--Hadamard manifolds. In that setting, comparison results from
Riemannian geometry yield suitable upper bounds for the heat kernel and,
consequently, for the fractional Green function. Such estimates allow one to
establish existence results for the analogue of \eqref{maineq}, as well as
quantitative smoothing effects, showing in particular that suitable initial
data, such as data in $L^1(M)$, generate essentially bounded solutions for
every $t>0$; see Theorem~\ref{thm:smoothing} below for results of this type in
our setting. The same approach also yields information on the long-time
behavior of solutions.


Smoothing effects and decay estimates for linear and nonlinear evolution
equations have a long history and are closely related to suitable functional
inequalities, such as Sobolev, Nash, Gagliardo--Nirenberg, and Log-Sobolev
inequalities; see, in particular, \cite{Davies} for the linear theory.
For nonlinear evolutions, including the porous medium equation and its
fast diffusion counterpart for $0<m<1$, the first results of this kind, to the best of our knowledge, appear to be due to B\'enilan \cite{Benilan1978} and V\'eron
\cite{Veron1979}, who used Moser iteration to establish smoothing effects.
We also refer to \cite{Vazquez2006} for a general account.
This approach has since been used extensively. Without attempting an
exhaustive list, we mention \cite{DePabloQuirosRodriguezVazquez2012} for
fractional nonlinear diffusion and \cite{GrigorSuerig2025, GrilloMuratori2016, GrilloMuratoriPunzo2018bis, GrilloMuratoriPunzo2018, GrilloMuratoriVazquez2017, GrilloMuratoriVazquez2019,RS2, RS, RS3, Vazquez2015Fundamental} for nonlinear diffusion equations on various classes of noncompact
Riemannian manifolds, in the
non-fractional case $s=1$.


A particularly interesting class in the study of linear evolution equations is that of manifolds with nonnegative Ricci curvature, which we denote by Ric\,$\geq 0$. The fundamental work of Li--Yau \cite{liyau} provides, among
several other results, sharp upper and lower bounds for solutions to the heat
equation on such manifolds, in which the volume growth of geodesic balls
plays a crucial role. These bounds readily imply that no global smoothing effect
for the heat equation can hold on collapsing manifolds, that is, if
$\inf_{x\in M}\mathrm{Vol}(B_1(x))=0$, where $\mathrm{Vol}$ denotes the
Riemannian volume and $B_r(x)$ the geodesic ball of radius $r$ centered at
$x$. Results for nonlinear evolutions under the condition Ric\,$\ge0$ are still
few. We refer in particular to
\cite{GrigorSuerig2025,GrigorSunSuerig}, especially for long-time upper
bounds for solutions corresponding to $L^1\cap L^\infty$-data, in a more
general geometric framework and for doubly nonlinear evolution equations,
and to \cite{FPMD} for the porous medium equation. All these results
concern the non-fractional case $s=1$.

In this article, we improve the results of \cite{FPMD} in two fundamental aspects. First, we treat the fractional case and prove existence without imposing any additional geometric assumptions beyond $s$-nonparabolicity; see Theorem~\ref{thm:existence}. The result applies to nonnegative initial data
in $L^1_{G_M^s,o}(M)$, which may be neither bounded nor integrable, and is
new in this generality even for $s=1$. Second, in
Theorems~\ref{thm:smoothing} and~\ref{thm:smoothing2} we establish
quantitative local smoothing estimates for nonnegative data in $L^1(M)$
and $L^1_{G_M^s,o}(M)$, respectively. Under the additional geometric Assumptions~\ref{A1} and~\ref{A2}, namely
that $M$ is noncollapsing and satisfies a suitable uniform volume-growth
condition, these estimates become global; see
Corollaries~\ref{cor:global-smoothing-L1} and
\ref{cor:global-smoothing-weighted}. Their short- and long-time behaviors
are sharp in appropriate senses, as discussed in
Section~\ref{sec:optimality}. In the range $0<s<1/2$,
$s$-nonparabolicity is automatic and Assumption~\ref{A2} holds with
$f(r)=r^{2-2s}$, considerably simplifying
the geometric assumptions; see Corollaries~\ref{cor:subcritical-existence},
\ref{cor:subcritical-local-smoothing}, and
\ref{cor:subcritical-global-smoothing}.


Our approach is based on the so-called ``weak dual formulation'', obtained
formally by applying $(-\Delta)^{-s}$ to the equation, yielding the system $\partial_t U = -u^m$ and $U = (-\Delta)^{-s} u$. This method goes back to Pierre's work for the local case $s=1$ \cite{Pierre1982} and was later extended to the fractional setting on $\mathbb{R}^N$, $0<s<1$, in \cite{Vazquez2014Barenblatt}. This weak dual formulation was subsequently combined with Green function
methods in \cite{BonforteSireVazquez2014,BonforteVazquez2015,
BonforteVazquez2016}, an approach that is also central here. Making this formal procedure rigorous is delicate. We first construct mild solutions by time discretization and then prove that they are in fact weak dual solutions.

The paper is organized as follows. Section~\ref{S2} collects the
preliminaries on the geometric setting, the fractional Laplacian and Green
function, the relevant initial-data spaces, and weak dual solutions.
Section~\ref{S3} states the main results for $s\in(0,1]$, with
Section~\ref{sec:optimality} devoted to their optimality.
Section~\ref{subsec:subcritical-range} treats the case $0<s<1/2$, where
$s$-nonparabolicity is immediate and some of the geometric assumptions used
in the general theory are automatic. In Section~\ref{sec:apriori} we
establish preliminary Green function and integral estimates, while
Section~\ref{S6} provides examples satisfying the assumptions required for
global smoothing. Section~\ref{S7} develops comparison estimates for the
Green function and Green potentials that are central to the proofs of our
main results. In Section~\ref{sec:frommstowds} we construct mild solutions
and identify them as weak dual solutions, while Section~\ref{sec: propswds}
establishes their key properties. These are then used in
Section~\ref{sec:proofsofmainres} to prove the main results.
Section~\ref{S11} extends the approach to more general filtration
nonlinearities. Finally, Appendix~\ref{proofs} contains the proofs of some
technical results.

\section{Preliminaries and Assumptions}\label{S2}

\medskip

\noindent \bf Notation\rm. For a function $v \colon [0,\infty) \times M \to \mathbb{R}$, we denote by $v(t)$ the function $v(t, \cdot) \colon M \to \mathbb{R}$ for each fixed $t \geq 0$. Unless otherwise specified, all Riemannian manifolds in this article are assumed to be connected, without boundary, and of dimension $n > 2s$, with $s\in(0,1]$. We shall denote by $\mu$ the Riemannian measure on $M$ and by $d$ geodesic distance on $M$. $\operatorname{Vol}$ will denote the Riemannian volume and $B_r(x)$ the Riemannian ball of radius $r$ centered at $x$.

\subsection{The Fractional Laplacian and the Fractional Green Function}

For $s=1$, $(-\Delta)^1=-\Delta$ denotes the \textit{Laplace--Beltrami operator} on the manifold $M$. When $s\in(0,1)$ the \textit{fractional Laplace--Beltrami operator} $(-\Delta)^s$ is defined as the spectral $s$-th power of the Laplace--Beltrami operator $-\Delta$. 
On a complete Riemannian manifold, $-\Delta$  coincides with its closure, initially defined on $C_c^\infty(M)$, since $-\Delta$ is essentially self-adjoint \cite[Theorem 2.4]{strichartz1983}. For $s\in (0,1)$, on any stochastically complete manifold, a property which always holds on manifolds with $\operatorname{Ric}\ge0$, the operator admits an explicit integral representation on suitable classes of functions; in particular, for $v\in C_c^\infty(M)$, and more generally whenever the following integral is well defined, one has
\begin{equation*}
(-\Delta)^s v(x) = \frac{1}{\Gamma(-s)} \int_0^\infty \left( \int_M p_M(t,x,y)(v(y)-v(x))\,d\mu(y) \right) \frac{dt}{t^{1+s}},
\end{equation*}
where $p_M(t,x,y)$ denotes the heat kernel associated with $-\Delta$.


The \textit{fractional Green function} $G_M^s(x,y)$ is defined for $s\in(0,1]$ for distinct points $x,y \in M$ by
\begin{equation}\label{fracGreen'sheat}
    G_M^s(x,y) \coloneqq \frac{1}{\Gamma(s)} \int_0^\infty \frac{p_M(t,x,y)}{t^{1-s}} \, dt.
\end{equation}
\begin{defn}
We say that $M$ is \emph{$s$-nonparabolic} if the fractional Green function
\eqref{fracGreen'sheat} is finite and positive for all $x\neq y$.
\end{defn}
For $s=1$, this is the usual notion of nonparabolicity. We will always be dealing with \textit{$s$-nonparabolic} manifolds. Then $G_M^s$ is positive and satisfies
\begin{equation*}
    (-\Delta)^s G_M^s(x,\cdot) = \delta_x,
\end{equation*}
in the distributional sense, where $\delta_x$ is the Dirac measure centered at $x$.

For a function $\psi\in L^1(M)\cap L^{\infty}(M)$, we define
\begin{equation}\label{defform}
    (-\Delta)^{-s}\psi(x)\coloneqq\int_M \psi(y)G_M^s(x,y)\,d\mu(y).
\end{equation}
By Lemma~\ref{lem_leftinv}, within the geometric framework established in Section~\ref{sec:fracgreenest}, this operator is well-defined and constitutes the left-inverse of $(-\Delta)^s$, mapping $L^1(M) \cap L^\infty(M)$ continuously into $L^{\infty}_{\mathrm{loc}}(M)$.

\subsection{Geometric Setting, Assumptions and Properties}
Let $(M,g)$ be a complete, noncompact $n$-dimensional Riemannian manifold, $n>2s$, with nonnegative Ricci curvature, $\mathrm{Ric}\geq 0$. We denote by $d\mu$ the Riemannian volume measure, by $\Delta$ the Laplace--Beltrami operator, and by $B_r(x)$ the geodesic ball of radius $r>0$ centered at $x \in M$. The geodesic distance between points $x, y \in M$ is written as $d(x,y)$.

In this setting, the \textit{Bishop--Gromov Volume Comparison Theorem} (see e.g. \cite{petersen2016}) applies, yielding for any $0 < R_1 < R_2$:
\begin{equation}\label{bishopgromov}
    \frac{\mathrm{Vol}(B_{R_2}(x))}{\mathrm{Vol}(B_{R_1}(x))}\leq \frac{R_2^n}{R_1^n}.
\end{equation}
One obtains the following uniform volume growth estimate: For any $R>0$, with $\omega_n$ denoting the volume of the unit Euclidean ball in $\mathbb{R}^n$,
\begin{equation}\label{bishopgromovD}
\mathrm{Vol}(B_R(x)) \leq \omega_n R^n.
\end{equation}

We also recall that on any Riemannian manifold $M$ with $\mathrm{Ric}\geq 0$, geodesic balls satisfy the \textit{Volume Doubling Property}: There exists a constant $C_d > 0$, such that for all $x_0 \in M$ and $r > 0$,
\begin{equation}\label{rmk:doubling}
    \mathrm{Vol}(B_{2r}(x_0)) \leq C_d\, \mathrm{Vol}(B_r(x_0)),
\end{equation}
see e.g. \cite{grigoryan2009}.

\begin{rmk}
We will repeatedly make use of the coarea formula: For any $x\in M$ and a.e. $r>0$ we have
\begin{equation*}
\frac{d}{dr}\mathrm{Vol}(B_r(x)) = \mathrm{Meas}(\partial B_r(x)),
\end{equation*}
where $\mathrm{Meas}(\partial B_r(x))$ denotes the surface measure of $\partial B_r(x)$.
\end{rmk}

\medskip

\noindent The following Assumption will only be imposed in the results requiring estimates uniform with respect to the pole, in particular in the global smoothing results.

\begin{assumpt}\label{A1}
$M$ is noncollapsing, i.e.,
\begin{equation}\label{noncollapsing}
    \alpha \coloneqq \inf_{x\in M} \mathrm{Vol}(B_1(x)) > 0.
\end{equation}
\end{assumpt}
This property is not a consequence of the previous assumptions, as there exist manifolds with $\mathrm{Ric} \geq 0$ where \eqref{noncollapsing} fails, see, e.g., \cite[Example 2]{croke-karcher1988}.

\medskip Independently of Assumption \ref{A1}, the preceding geometric hypotheses imply the following comparison property. Indeed we shall now prove a crucial bound for ratios of volumes of balls of different radii, which will be crucial in what follows and in particular will allow to prove existence of solutions to initial value problem for the fractional porous medium equation on $M$, see \eqref{maineq}, where solutions will be meant in the weak dual sense, see Definition \ref{defn:wds}, and the space of initial data will be $L^1_{G_M^s,o}(M)$,  where $o\in M$ is any fixed reference point, as defined in Section \ref{weakdual}, such space being larger than $L^1(M)$.

\begin{prop}\label{general} Let $M$ be a complete, \(s\)-nonparabolic Riemannian manifold with \(\partial M = \emptyset\), $n>2s$, and \(\mathrm{Ric} \ge 0\). 
Let also $o\in M$ be fixed, and define
\begin{equation}\label{deffo1}
f(t) \coloneqq \frac{\mathrm{Vol}(B_t(o))}{t^{2s-1}},\ \ \forall t>0.
\end{equation}
Let finally  $R_0(x):= 2d(x,o)+1\geq 1$ and $\gamma:=C_d^2$, where $C_d$ is the doubling constant of equation \eqref{rmk:doubling}. Then the manifold $M$ satisfies the following
{\renewcommand{\thecond}{\ensuremath{\mathbf B}}}
\begin{cond}\label{B}
For all $x\in M$, all $R_1, R_2$ s.t. $1\leq R_0(x)\leq R_1 \leq R_2$ one has
\begin{equation}\label{fcondition1}
       \frac{R_2^{2s-1}f(R_2)}{\mathrm{Vol}(B_{R_2}(x))} \leq \gamma \frac{R_1^{2s-1}f(R_1)}{\mathrm{Vol}(B_{R_1}(x))},
\end{equation}
with
\begin{equation}\label{betacondition1}
\beta \coloneqq \int_{1}^{\infty} \frac{1}{f(t)} \, dt < +\infty.
\end{equation}
\end{cond}
\end{prop}

The proof of this result will be given in the Appendix.

\noindent Inequalities of the form \eqref{fcondition1} may hold on $M$ for different choices of the function $f$. When this is the case, the corresponding choice of $f$ determines the refined smoothing effects that can be obtained for solutions to the evolution equation under consideration. Condition \ref{B}, in which the threshold radius $R_0$ is allowed to depend on the pole, is sufficient for the local smoothing estimates proved below. In order to obtain estimates that are uniform with respect to the pole, and hence global on $M$, we introduce the stronger Assumption \ref{A2}, where the same function $f$ is admissible for all poles with a common threshold radius $R_0$.

{\renewcommand{\theassumpt}{\ensuremath{\mathbf A_2}}
\begin{assumpt}\label{A2}There exist a constant $\gamma>0$ and a continuous function $f\colon[1,\infty)\to(0,\infty)$ such that $t^{2s-1}f(t)$ is increasing and such that  there exists $R_0\geq 1$ such that for all $x\in M$ and all $R_0\leq R_1 \leq R_2$:
\begin{equation}\label{fcondition}
       \frac{R_2^{2s-1}f(R_2)}{\mathrm{Vol}(B_{R_2}(x))} \leq \gamma \frac{R_1^{2s-1}f(R_1)}{\mathrm{Vol}(B_{R_1}(x))}
\end{equation}
and
\begin{equation}\label{betacondition}
    \beta \coloneqq \int_{1}^{\infty} \frac{1}{f(t)} \, dt < +\infty.
\end{equation}
\end{assumpt}}

\begin{rmk}


The proof of Proposition \ref{general} shows that, under the running assumptions on $M$, if \eqref{fcondition} holds at some $o \in M$ for some $\gamma,f$ and $\overline{R}$, then it holds for every $x \in M$, with $\overline{R}$ replaced by a suitable $R_0(x)$ (in fact, $R_0(x)=\overline{R}+2d(x,o)$ works), with the same function $f$ and with $\gamma$ replaced by $\gamma C_d^2$. Such pole-dependent radii are sufficient for the local smoothing estimates proved below; a radius independent of the pole is required
only for the corresponding global estimates.
\end{rmk}

\begin{rmk}
On manifolds with nonnegative Ricci curvature, property \eqref{fcondition} in Assumption \hyperref[A2]{$\mathbf{A_2}$} is invariant under Lipschitz equivalence of metrics. Specifically, if \eqref{fcondition} holds for $d$ and another metric $d'$ satisfies
\begin{equation*}
    K_1d(x,y)\leq d'(x,y)\leq K_2 d(x,y)\quad \text{for all }x,y\in M,
\end{equation*}
for positive constants $K_1, K_2$, then \eqref{fcondition} also holds for $d'$, for the same function $f$ and a larger constant $\gamma'$. The proof follows a similar strategy as in Proposition \ref{general}: One uses the metric equivalence to establish inclusions between $d$-balls and $d'$-balls and applies the Doubling Property \eqref{rmk:doubling}, both for $d$ and $d'$.
\end{rmk}

\begin{rmk}
The noncollapsing Assumption \hyperref[A1]{$\mathbf{A_1}$} combined with the Bishop--Gromov Volume Comparison \eqref{bishopgromov} implies in particular for every $R\geq1$ and $0<r\leq R$ and every $x\in M$,
\begin{equation*}
    \mathrm{Vol}(B_r(x)) \geq \mathrm{Vol}(B_R(x)) \frac{r^n}{R^n} \geq \mathrm{Vol}(B_1(x))\frac{r^n}{R^n} \geq \alpha \frac{r^n}{R^n}.
\end{equation*}
Note also that $\mathrm{Vol}(B_R(x)) \geq \alpha$ holds for all $R\geq1$, $x\in M$.
\end{rmk}

\subsection{Green-Weighted Spaces and Weak Dual Solutions}\label{weakdual}

Let $(M,g)$ be a complete, noncompact, s-nonparabolic $n$-dimensional Riemannian manifold, $n>2s$, with $\mathrm{Ric}\geq 0$. We investigate solutions to \eqref{maineq} with initial data in either $L^1(M)$ or the following weighted space, defined in terms of the fractional Green function $G^s_M$: for any fixed reference point $o\in M$
\begin{equation*}
L^1_{G^s_M,o}(M) \coloneqq \left\{
v\colon M \to \mathbb{R} \text{ measurable } \,\Bigg|\,
\|v\|_{L^1_{G_M^s,o}(M)} < +\infty
\right\},
\end{equation*}
with
\begin{equation}\label{L1Gnorm}
\|v\|_{L^1_{G_M^s,o}(M)} \coloneqq
\int_{B_1(o)} |v(x)| \, d\mu(x) +
\int_{M \setminus B_1(o)} |v(x)| G^s_M(x, o) \, d\mu(x).
\end{equation}

For our global smoothing estimates we will consider the space
\begin{equation*}
L^1_{G^s_M}(M) \coloneqq \left\{
v\colon M \to \mathbb{R} \text{ measurable } \,\Bigg|\,
 \|v\|_{L^1_{G_M^s}(M)} < +\infty
\right\},
\end{equation*}
with the norm
\begin{equation*}
\|v\|_{L^1_{G^s_M}(M)} \coloneqq \sup_{x_0 \in M} \|v\|_{L^1_{G_M^s,x_0}(M)}.
\end{equation*}
For any fixed compact nonempty set $V\subset M$ we also introduce the space
\begin{equation*}
L^1_{G^s_M,V}(M) \coloneqq \left\{
v\colon M \to \mathbb{R} \text{ measurable } \,\Bigg|\,
 \|v\|_{L^1_{G_M^s,V}(M)} < +\infty
\right\},
\end{equation*}
endowed with the norm
\begin{equation*}
\|v\|_{L^1_{G^s_M,V}(M)} \coloneqq \sup_{x_0 \in V} \|v\|_{L^1_{G_M^s,x_0}(M)}.
\end{equation*}

We now define weak dual solutions, motivated by the following formal observation: Applying the operator $(-\Delta)^{-s}$ to both sides of~\eqref{maineq} yields,
\begin{equation*}
    \partial_t [(-\Delta)^{-s} u] + u^m = 0.
\end{equation*}
This formal computation suggests the following weak formulation:

\begin{defn}\label{defn:wds}
Let $o\in M$ be a fixed reference point and let $u_0 \in L^1_{G^s_M,o}(M)$ be nonnegative. A measurable function $u \geq 0$ is called a \emph{Weak Dual Solution (WDS)} to Problem~\eqref{maineq} if for every $T > 0$:
\begin{enumerate}[label=(\roman*)]
    \item[(i)] $u \in C([0,T]; L^1_{G_M^s,x_0}(M))$ for all $x_0 \in M$;
    \item[(ii)] $u^m \in L^1((0,T); L^1_{\mathrm{loc}}(M))$;
    \item[(iii)] For all test functions $\psi \in C^1_c((0,T); L^\infty_c(M))$,
    \begin{equation}\label{wds-identity}
        \int_0^T \int_M \partial_t \psi (-\Delta)^{-s} u \, \mathrm{d}\mu\, \mathrm{d}t
        = \int_0^T \int_M u^m \psi \, \mathrm{d}\mu\, \mathrm{d}t;
    \end{equation}
    \item[(iv)] $u(0) = u_0$ almost everywhere in $M$.
\end{enumerate}
\end{defn}

The next technical result, and its main line of proof, will be used repeatedly in the paper. Its proof will be given in the Appendix.

\begin{lem}\label{rmk:contloc}
Let $M$ be a complete, noncompact, s-nonparabolic $n$-dimensional Riemannian manifold with $\mathrm{Ric}\geq0$.
If $u$ is a WDS  to \eqref{maineq} in the sense of
Definition~\ref{defn:wds}, then
$$(-\Delta)^{-s}u \in C\big([0,+\infty);L^1_{\rm loc}(M)\big).$$
\end{lem}


\subsection{Initial Data Spaces}

In Section~3, we prove local smoothing estimates for initial data
in $L^1(M)$ and in $L^1_{G_M^s,o}(M)$. Under the additional
Assumptions \ref{A1} and \ref{A2}, we obtain corresponding global estimates
for data in $L^1(M)$ and in $L^1_{G_M^s}(M)$.
\begin{prop}
\label{prop:compact-poles-equivalence}
Let $(M,g)$ be a complete, noncompact, $s$-nonparabolic
$n$-dimensional Riemannian manifold, with empty boundary,
$n> 2s$, $s\in(0,1]$, and $\operatorname{Ric}\geq 0$.
Let $V\subset M$ be nonempty and compact, and let $o\in M$.
Then there exists a constant $C=C(o,V)\geq 1$ such that
\[
 C^{-1}\|v\|_{L^1_{G^s_M,o}(M)}
 \leq
 \|v\|_{L^1_{G^s_M,V}(M)}
 \leq
 C\|v\|_{L^1_{G^s_M,o}(M)}
\]
for every measurable function $v$ on $M$. Consequently,
\[
 L^1_{G^s_M,V}(M)=L^1_{G^s_M,o}(M)
\]
with equivalent norms. In particular, the space
$L^1_{G^s_M,o}(M)$ is independent of the choice of the pole
$o\in M$, with equivalent norms.
\end{prop}

\begin{proof}
Set $D:=\max_{x_0\in V}d(o,x_0)$. The conclusion is immediate if $D=0$, so assume that $D>0$.
By the lower Green function estimate recalled in
Theorem~\ref{thm:bounGreen's}, together with
\eqref{bishopgromovD}, there exists $c_0=c_0(n,s)>0$ such that
\[
 G^s_M(x,y)\geq
 \frac{c_0}{d(x,y)^{n-2s}},
 \qquad x\neq y.
\]
Therefore,
\[
 G^s_M(o,x_0)\geq
 m:=c_0D^{-(n-2s)}>0,
 \qquad x_0\in V\setminus\{o\}.
\]

By \cite[Corollary~2.5]{GuHuangSun} if $s\in(0,1)$ or by \cite{GriSunVer} if $s=1$, the Green function satisfies
the quasi-metric property: there exists $\kappa\geq 1$ such that
\[
 G^s_M(x,y)\wedge G^s_M(y,z)
 \leq
 \kappa G^s_M(x,z)
\]
whenever the terms involved are finite. By symmetry and the upper bound in Theorem~\ref{thm:bounGreen's}, together with Theorem~\ref{thm:charnonpar},
\[
 G^s_M(x,o)\longrightarrow 0
 \qquad\text{as }d(x,o)\longrightarrow\infty.
\]
Hence, after choosing $R>D+1$ sufficiently large, one has
\[
 G^s_M(x,o)<\frac{m}{\kappa},
 \qquad x\in M\setminus B_R(o).
\]
Applying the quasi-metric property twice, with intermediate poles
$o$ and $x_0$, respectively, gives
\begin{equation}
 \kappa^{-1}G^s_M(x,o)
 \leq
 G^s_M(x,x_0)
 \leq
 \kappa G^s_M(x,o)
\label{eq:uniform-compact-poles}
\end{equation}
for every $x_0\in V$ and every $x\in M\setminus B_R(o)$, the case
$x_0=o$ being immediate.

The standard off-diagonal continuity of $G^s_M$ yields
\[
 A:=
 \sup_{x_0\in V}\;
 \sup_{x\in \overline{B_R(o)}\setminus B_1(x_0)}
 G^s_M(x,x_0)
 <\infty.
\]
The lower Green function estimate also gives
\[
\begin{aligned}
 \int_{B_R(o)}|v|\,d\mu
 &\leq
 \int_{B_1(o)}|v|\,d\mu
 +
 \frac{R^{n-2s}}{c_0}
 \int_{B_R(o)\setminus B_1(o)}
 |v(x)|G^s_M(x,o)\,d\mu(x) \\
 &\leq
 \left(1+\frac{R^{n-2s}}{c_0}\right)
 \|v\|_{L^1_{G^s_M,o}(M)}.
\end{aligned}
\]
Since $B_1(x_0)\subset B_R(o)$ for every $x_0\in V$, from
\eqref{L1Gnorm} and \eqref{eq:uniform-compact-poles} we obtain
\begin{equation*}
 \|v\|_{L^1_{G^s_M,x_0}(M)}\leq (1+A)\int_{B_R(o)}|v|\,d\mu+
 \kappa\int_{M\setminus B_R(o)}
 |v(x)|G^s_M(x,o)\,d\mu(x)
 \leq  C\|v\|_{L^1_{G^s_M,o}(M)},
\end{equation*}
with $C$ independent of $x_0\in V$. Taking the supremum over
$x_0\in V$ proves
\[
 \|v\|_{L^1_{G^s_M,V}(M)}
 \leq
 C\|v\|_{L^1_{G^s_M,o}(M)}.
\]

Finally, fix any $x_1\in V$. Applying the estimate just proved with
the compact set $\{o\}$ and pole $x_1$ yields
\[
 \|v\|_{L^1_{G^s_M,o}(M)}
 \leq
 C\|v\|_{L^1_{G^s_M,x_1}(M)}
 \leq
 C\|v\|_{L^1_{G^s_M,V}(M)},
\]
which proves the reverse inequality.
\end{proof}

For a measurable function $v$ on $M$, with no a priori integrability
assumption, we use the notation
\[
(-\Delta)^{-s}|v|(x)
:=
\int_M G_M^s(x,y)|v(y)|\,d\mu(y)
\in [0,+\infty].
\]

\begin{prop}\label{prop:properembedding}
Let $M$ be a complete, noncompact, $s$-nonparabolic $n$-dimensional
Riemannian manifold, $n>2s$, with $\operatorname{Ric}\geq 0$, and let $o\in M$.
Then, for every measurable function $v$ on $M$,
\[
v\in L^1_{G_M^s,o}(M)
\quad\Longleftrightarrow\quad
(-\Delta)^{-s}|v|\in L^1_{\mathrm{loc}}(M).
\]
Furthermore, both continuous inclusions
\[
L^1(M)\hookrightarrow L^1_{G_M^s,o}(M),
\qquad
L^1_{G_M^s}(M)\hookrightarrow L^1_{G_M^s,o}(M)
\]
are strict.
\end{prop}

The proof of this result will be given in the Appendix.

\begin{rmk}
If, in addition, Assumptions \ref{A1} and \ref{A2} hold, then
\[
L^1(M)\subsetneq L^1_{G_M^s}(M)
\subsetneq L^1_{G_M^s,o}(M).
\]
Indeed, the first strict inclusion is proved in Remark~\ref{rmk:appendix},
while the second one follows from Proposition~\ref{prop:properembedding}.
\end{rmk}

\subsection{Estimates for the Fractional Green Function} \label{sec:fracgreenest}

The existence of a global Green function is central to our approach. It is well known that a complete, noncompact Riemannian manifold $M$ with $\partial M = \emptyset$ and $\mathrm{Ric}\geq0$ satisfies both the volume doubling condition and the (scale-invariant) Poincaré inequality. The following result is then a consequence of Proposition 2.4
in \cite{GuHuangSun}, when $s\in(0,1)$. For the case $s=1$ we refer to \cite{liyau}.
\begin{thm}\label{thm:bounGreen's}
Let $M$ be a complete, noncompact manifold with $\partial M = \emptyset$ and $\mathrm{Ric}\geq0$. There exist constants $C_1,C_2>0$ depending only on $n=\mathrm{dim}(M)$ and $s$ such that for any distinct points $x,y\in M$:
\begin{equation*}
C_1\int_{d(x,y)}^{\infty} \frac{t^{2s-1}}{\mathrm{Vol}(B_t(x))}\;dt \leq G_M^s(x,y) \leq C_2 \int_{d(x,y)}^{\infty} \frac{t^{2s-1}}{\mathrm{Vol}(B_t(x))}\;dt.
\end{equation*}
\end{thm}
The following theorem is a direct consequence of Theorem \ref{thm:bounGreen's}. 

\begin{thm}\label{thm:charnonpar}
Let $M$ be a complete Riemannian manifold with $\partial M =\emptyset$ and $\mathrm{Ric}\geq0$. Then $M$ is $s$-nonparabolic if and only if for some, and hence for all, $x\in M$ and $r>0$
\begin{equation*}
    \int_{r}^{\infty} \frac{t^{2s-1}}{\mathrm{Vol}(B_t(x))}\;dt<+\infty.
\end{equation*}
\end{thm}

An immediate consequence of Theorem \ref{thm:charnonpar} is the following.

\begin{cor}
Let $M$ be a complete Riemannian manifold with $\partial M =\emptyset$ and $\mathrm{Ric}\geq0$. If there exist a point $x \in M$, a constant $C > 0$, and a radius $R>0$ such that
\begin{equation*}
\mathrm{Vol}(B_r(x)) \leq C r^{2s} \quad \text{for all } r > R,
\end{equation*}
then $M$ is $s$-parabolic.
\end{cor}

\section{Statements of the Main Results}
\label{S3}

\begin{thm}[Existence of Weak Dual Solutions]\label{thm:existence}
Let $M$ be a complete, noncompact, s-nonparabolic $n$-dimensional Riemannian manifold, $n>2s$, $s\in(0,1]$, with $\mathrm{Ric} \geq 0$ and let $o\in M$ be fixed. For any nonnegative initial datum $u_0 \in L^1_{G_M^s,o}(M)$ there exists a weak dual solution $u$ to problem \eqref{maineq} in the sense of Definition \ref{defn:wds}.
\end{thm}

\begin{rmk}
The WDS in Theorem \ref{thm:existence} is constructed as the monotone limit of WDS's associated with an increasing sequence of initial data in $L^1(M)\cap L^{\infty}(M)$, converging to $u_0$. Following the argument in Step 4 of \cite[Proof of Theorem 3.6]{ BonforteVazquez2015}, one can verify that this limit is independent of the choice of monotone approximating sequence, and hence uniquely determined within the class of weak dual solutions obtained by monotone approximation. In addition, the solution depends continuously on the initial datum, see Remark \ref{rem2222}. The key ingredient in this proof is the order preservation property \eqref{orderingpos} for mild solutions, established in Proposition \ref{prop:4.1} below.
\end{rmk}

In what follows, unless otherwise specified, weak dual solutions are always understood to be those obtained by monotone approximation as in Theorem \ref{thm:existence}.
We first establish local smoothing estimates under the basic
geometric assumptions, which hold for every weak dual solution with $u_0\in L^1(M)$ or with $u_0\in L^1_{G^s_M,o}(M)$. If the additional Assumptions \ref{A1} and \ref{A2}
hold, the relevant estimates can be made uniform with respect to
the pole, yielding global smoothing estimates on $M$ for solutions with data in $L^1(M)$ or in $L^1_{G^s_M}(M)$.

\begin{thm}[Local Smoothing Estimates for $L^1$-data]
\label{thm:smoothing}
Let $M$ be a complete, noncompact, $s$-nonparabolic
$n$-dimensional Riemannian manifold, $n>2s$, $s\in(0,1]$, with $\operatorname{Ric}\geq 0$.
Let $u$ be the weak dual solution to \eqref{maineq} (in the sense of Definition \ref{defn:wds}) with nonnegative initial datum $u_0 \in L^1(M)$ . If $u_0\equiv 0$, then $u\equiv 0$.
If $u_0\not\equiv 0$, then, for every nonempty compact set
$K\Subset M$, there exists a constant $C_K>0$ such that the
following estimates hold.

\textbf{Short-time behavior.} One has
\[
\|u(t)\|_{L^\infty(K)}
\leq
C_K
t^{-\frac{n}{n(m-1)+2s}}
\|u_0\|_{L^1(M)}^{\frac{2s}{n(m-1)+2s}}\qquad\textrm{ for every }0<t\leq \|u_0\|_{L^1(M)}^{-(m-1)},
\]

\textbf{Long-time behavior.} One has
\[
\|u(t)\|_{L^\infty(K)}
\leq
C_K t^{-1/m}\|u_0\|_{L^1(M)}^{1/m}\qquad\textrm{ for every }t\geq \|u_0\|_{L^1(M)}^{-(m-1)}.
\]

\textbf{Refined long-time behavior.} Fix $o\in M$ and, for $R\geq 1$, define
\[
h_o(R):=
\operatorname{Vol}(B_R(o))
\int_R^{+\infty}
\frac{r^{2s-1}}{\operatorname{Vol}(B_r(o))}\,dr
+\frac{R^{2s}}{2s},
\]
and
\[
\theta_o(R):=
\operatorname{Vol}(B_R(o))
h_o(R)^{\frac{1}{m-1}}.
\]
There exists a constant $C=C(s,m,n)>0$ such that for every nonempty compact set $K\Subset M$ there exists
$R_{K,o}\geq 1$ such that
one has the refined local large-time estimate
\[
\|u(t)\|_{L^\infty(K)}
\leq
Ct^{-\frac{1}{m-1}}
\left[
h_o\!\left(
\theta_o^{-1}\!\left(
t^{\frac{1}{m-1}}\|u_0\|_{L^1(M)}
\right)
\right)
\right]^{\frac{1}{m-1}}
\]
for every $t\geq
\theta_o(R_{K,o})^{m-1}
\|u_0\|_{L^1(M)}^{-(m-1)}$.
\end{thm}

\begin{rmk}\label{rem111}
It is easy to see that $h_o$ and $\theta_o$ are increasing and
$\theta_o(R)\to+\infty$ as $R\to+\infty$.

We actually prove the following local master estimate.
With the notation of Theorem~\ref{thm:smoothing}, for every
nonempty compact set $K\Subset M$ there exist $R_{K,o}\geq 1$ such that, for every $t>0$,
\[
\|u(t)\|_{L^\infty(K)}
\leq
Ct^{-\frac{1}{m-1}}
\inf_{R\geq R_{K,o}}
\left\{
h_o(R)^{\frac{1}{m-1}}
\left(
1+
\frac{
t^{\frac{1}{m-1}}\|u_0\|_{L^1(M)}
}{
\theta_o(R)
}
\right)^{\frac{1}{m}}
\right\}.
\]
In particular, whenever $t^{\frac{1}{m-1}}\|u_0\|_{L^1(M)}\geq\theta_o(R_{K,o})$, the refined local large-time estimate in
Theorem~\ref{thm:smoothing} follows by taking $R=\theta_o^{-1}\!\left(t^{\frac{1}{m-1}}\|u_0\|_{L^1(M)}\right)$.
\end{rmk}

\begin{cor}[Global Smoothing Estimates for $L^1$-data]
\label{cor:global-smoothing-L1}
Under the assumptions of Theorem~\ref{thm:smoothing}, assume in
addition that Assumptions \ref{A1} and \ref{A2} hold. For $R\geq 1$,
define
\begin{equation}
\label{defnfnch}
h(R):=
R^{2s-1}f(R)
\int_R^{+\infty}\frac{dr}{f(r)}
+\frac{R^{2s}}{2s},
\end{equation}
\begin{equation*}
F(R):=
\inf_{x\in M}\operatorname{Vol}(B_R(x)),
\end{equation*}
and
\begin{equation}
\label{thetadefn}
\theta(R):=
F(R)h(R)^{\frac{1}{m-1}},
\end{equation}
where $f$ and the constant $R_0\geq 1$ are those in
Assumption \ref{A2}. Then there exists a constant $C>0$ such that
the following estimates hold.

\textbf{Short-time behavior.} One has
\begin{equation}
\label{eq:shorttime}
\|u(t)\|_{L^\infty(M)}
\leq
Ct^{-\frac{n}{n(m-1)+2s}}
\|u_0\|_{L^1(M)}^{\frac{2s}{n(m-1)+2s}}\qquad\textrm{ for every }0<t\leq \|u_0\|_{L^1(M)}^{-(m-1)}.
\end{equation}

\textbf{Long-time behavior.} One has
\begin{equation}
\label{eq:longtime}
\|u(t)\|_{L^\infty(M)}
\leq
Ct^{-1/m}\|u_0\|_{L^1(M)}^{1/m}\qquad\textrm{ for every }t\geq \|u_0\|_{L^1(M)}^{-(m-1)}.
\end{equation}
In addition, one has
\begin{equation}
\label{eq:improved}
\|u(t)\|_{L^\infty(M)}
\leq
Ct^{-\frac{1}{m-1}}
\left[
h\!\left(
\theta^{-1}\!\left(
t^{\frac{1}{m-1}}\|u_0\|_{L^1(M)}
\right)
\right)
\right]^{\frac{1}{m-1}}.
\end{equation}
for every $t\geq \theta(R_0)^{m-1} \|u_0\|_{L^1(M)}^{-(m-1)}$.
\end{cor}

\begin{rmk}\label{rem3333}
We actually prove the following global master estimate: for
every $t>0$,
\[
\|u(t)\|_{L^\infty(M)}
\leq
Ct^{-\frac{1}{m-1}}
\inf_{R\geq R_0}
\left\{
h(R)^{\frac{1}{m-1}}
\left(
1+
\frac{
t^{\frac{1}{m-1}}\|u_0\|_{L^1(M)}
}{
\theta(R)
}
\right)^{\frac{1}{m}}
\right\}.
\]
In particular, \eqref{eq:improved} follows by taking $R=\theta^{-1}\!\left(t^{\frac{1}{m-1}}\|u_0\|_{L^1(M)}\right)$.
More generally, \eqref{eq:improved} remains valid with $\theta$
replaced by any increasing function $\widehat{\theta}$ such that
\[
\widehat{\theta}(R)\longrightarrow+\infty
\qquad\text{as }R\to+\infty,
\]
and
\[
\widehat{\theta}(R)
\leq
C F(R)h(R)^{\frac{1}{m-1}}
\qquad\text{for every }R\geq R_0,
\]
with the corresponding large-time condition $t\geq \widehat{\theta}(R_0)^{m-1}\|u_0\|_{L^1(M)}^{-(m-1)}$.
\end{rmk}

The following corollary establishes the time decay rate of solutions in the global $L^\infty$-norm based on their initial $L^1$-norm, under different volume growth conditions:

\begin{cor}\label{cor1}
Let $M$ be a complete, noncompact $n$-dimensional Riemannian manifold with $\mathrm{Ric}\geq 0$ satisfying Assumptions \ref{A1}, \ref{A2}.
\begin{itemize}
\item[(i)] Let $f(R) = R^{k-(2s-1)}(\log R)^\delta$ for some $k \in (2s,n]$ and $\delta \in \mathbb{R}$, for $R$ large. Suppose there exists $\lambda \in (2s,n] $ and some constant $C>0$ such that for every $x\in M$,
\begin{equation*}
\mathrm{Vol}(B_R(x)) \geq C R^\lambda \quad \text{for all sufficiently large } R>0.
\end{equation*}
Then there exists a constant $C>0$ such that for any weak dual solution $u$ of \eqref{maineq} (in the sense of Definition~\ref{defn:wds}) with nonnegative initial datum $u_0 \in L^1(M)$
\begin{equation*}
\|u(t)\|_{L^\infty(M)} \leq \frac{C}{t^{\frac{\lambda}{\lambda(m-1) + 2s}}} \|u_0\|_{L^1(M)}^{\frac{2s}{\lambda(m-1) + 2s}}
\end{equation*}
 for every sufficiently large $t>0$.
\item[(ii)] Let $f(R) = R(\log R)^\delta$ for some $\delta > 1$. Suppose there exists $\lambda \in (2s,n]$, $\sigma\in\mathbb{R}$ and $C>0$ such that for every $x\in M$,
\begin{equation*}
\mathrm{Vol}(B_R(x)) \geq C R^\lambda (\log R)^\sigma \quad \text{for all sufficiently large } \, R>0.
\end{equation*}
Then there exists a constant $C>0$ such that for any weak dual solution $u$ of \eqref{maineq} with nonnegative initial datum $u_0 \in L^1(M)$
\begin{equation*}
\|u(t)\|_{L^\infty(M)} \leq \frac{C}{t^{\frac{1}{m-1}}} G\left(t^{\frac{1}{m-1}} \|u_0\|_{L^1(M)}\right)^{\frac{2s}{m-1}} \left(\log G\left(t^{\frac{1}{m-1}} \|u_0\|_{L^1(M)}\right)\right)^{\frac{1}{m-1}}
\end{equation*}
for every sufficiently large $t>0$, where for $r$ sufficiently large
\begin{equation*}
G(r) = \begin{cases}
e^{\frac{b}{a} W_0\left(\frac{a}{b} r^{\frac{1}{b}}\right)} & \text{if } b > 0, \\
r^{\frac{m-1}{(m-1)\lambda + 2s}} & \text{if } b = 0,\\
e^{\frac{b}{a} W_{-1}\left(\frac{a}{b} r^{\frac{1}{b}}\right)} & \text{if } b < 0,
\end{cases}
\end{equation*}
with $a = \lambda + \frac{2s}{m-1}$, $b = \sigma + \frac{1}{m-1}$, $W_0$ is the inverse of $w(r) = re^r$ on $[-1,\infty)$ and $W_{-1}$ is the inverse of $w(r) = re^r$ on $(-\infty,-1]$.
\end{itemize}
\end{cor}

\begin{rmk}
  The estimate in Corollary~\ref{cor1}(ii) can be rewritten in a more explicit asymptotic form by using the classical expansion
$$\begin{aligned}W_0(z)&=\log z-\log|\log z|+o(1)\qquad\text{as }z\to+\infty,\\ W_{-1}(z)&=\log |z|-\log|\log|z||+o(1)\qquad\text{as }z\to0^-\end{aligned}$$
Let $a=\lambda+\frac{2s}{m-1}$, $b=\sigma+\frac{1}{m-1}$. Then, independently of the sign of $b$, one has
$$G(r)\asymp r^{1/a}(\log r)^{-b/a}\qquad\text{as }r\to\infty,$$
and therefore
$$\|u(t)\|_{L^\infty(M)}\leq C t^{-\frac{\lambda}{\lambda(m-1)+2s}}\|u_0\|_{L^1(M)}^{\frac{2s}{\lambda(m-1)+2s}}\left(\log\bigl(e+t^{1/(m-1)}\|u_0\|_{L^1(M)}\bigr)\right)^{\frac{\lambda-2s\sigma}{\lambda(m-1)+2s}}$$
for every sufficiently large $t$.
\end{rmk}

An immediate consequence of Corollary \ref{cor:global-smoothing-L1} and Corollary \ref{cor1} is the following.
\begin{cor}\label{cor:avr-decay}
Let \(M\) be a complete, noncompact \(n\)-dimensional Riemannian
manifold with \(\operatorname{Ric}\geq 0\) and
\[
\operatorname{AVR}(M)
:=
\lim_{R\to\infty}
\frac{\operatorname{Vol}(B_R(x))}{\omega_n R^n}
>0.
\]
Then Assumptions \ref{A1} and \ref{A2} hold, with $f(R)=R^{n-(2s-1)}$. Hence there exists a
constant \(C>0\), depending only on \(n\), \(s\), \(m\), and
\(\operatorname{AVR}(M)\), such that if \(u\) is the weak dual solution with nonnegative initial datum \(u_0\in L^1(M)\) one has
\[
\|u(t)\|_{L^\infty(M)}
\leq
C
t^{-\frac{n}{n(m-1)+2s}}
\|u_0\|_{L^1(M)}^{\frac{2s}{n(m-1)+2s}}\qquad\textrm{ for every }t>0.
\]
\end{cor}

\begin{thm}[Local Smoothing Estimates for
$L^1_{G_M^s,o}$-data]
\label{thm:smoothing2}
Let $M$ be a complete, noncompact, $s$-nonparabolic
$n$-dimensional Riemannian manifold, $n>2s$, $s\in(0,1]$, with $\operatorname{Ric}\geq 0$,
and let $o\in M$ be fixed.
Let $u$ be the weak dual solution to \eqref{maineq} with nonnegative initial datum $u_0\in L^1_{G_M^s,o}(M)$. If $u_0\equiv 0$, then $u\equiv 0$.
If $u_0\not\equiv 0$, then, for every nonempty compact set
$K\Subset M$, there exists a constant $C_{K,o}>0$ such that the
following estimates hold.

\textbf{Short-time behavior.} One has
\[
\|u(t)\|_{L^\infty(K)}
\leq
C_{K,o}
t^{-\frac{n}{n(m-1)+2s}}
\|u_0\|_{L^1_{G_M^s,o}(M)}^{
\frac{2s}{n(m-1)+2s}
}\qquad\textrm{ for every }0<t\leq
\|u_0\|_{L^1_{G_M^s,o}(M)}^{-(m-1)}.
\]

\textbf{Long-time behavior.} One has
\[
\|u(t)\|_{L^\infty(K)}
\leq
C_{K,o}
t^{-1/m}
\|u_0\|_{L^1_{G_M^s,o}(M)}^{1/m}\qquad\textrm{ for every }t\geq
\|u_0\|_{L^1_{G_M^s,o}(M)}^{-(m-1)}.
\]
\end{thm}

\begin{cor}[Global Smoothing Estimates for
$L^1_{G_M^s}$-data]
\label{cor:global-smoothing-weighted}
Under the assumptions of Theorem~\ref{thm:smoothing2}, assume in
addition that Assumptions \ref{A1} and \ref{A2} hold and that $u_0\in L^1_{G_M^s}(M)$.
Then there exists a constant $C>0$ such that the following
estimates hold.

\textbf{Short-time behavior.} One has
\[
\|u(t)\|_{L^\infty(M)}
\leq
C
t^{-\frac{n}{n(m-1)+2s}}
\|u_0\|_{L^1_{G_M^s}(M)}^{
\frac{2s}{n(m-1)+2s}
}\textrm{ for every }0<t\leq
\|u_0\|_{L^1_{G_M^s}(M)}^{-(m-1)}.
\]

\textbf{Long-time behavior.} One has
\[
\|u(t)\|_{L^\infty(M)}
\leq
C
t^{-1/m}
\|u_0\|_{L^1_{G_M^s}(M)}^{1/m}\qquad\textrm{ for every }t\geq
\|u_0\|_{L^1_{G_M^s}(M)}^{-(m-1)}.
\]
\end{cor}

\begin{rmk}
The \(L^\infty\)-smoothing estimates above also yield corresponding \(L^q\)-estimates, \(1\le q<\infty\), by interpolation. For \(L^1\)-data this follows directly from the \(L^1\)-nonexpansivity, while in the Green-weighted setting the analogous local estimates follow from the local \(L^1\)-control provided by the weighted estimates. We omit the resulting explicit formulas.
\end{rmk}

\subsection{Optimality}\label{sec:optimality}


We now discuss the sharpness of the smoothing estimates established above, summarized in the four propositions below. Two notions of sharpness will be considered. We say that an estimate is \emph{uniformly sharp} if its exponents cannot be improved while keeping the constant independent of both the initial datum and time in the relevant regime. We say that a rate is \emph{trajectory-wise sharp} if it is attained, at least along a sequence of times, by a single fixed solution.

The basic $L^1$-long-time bound with decay $t^{-1/m}$ in
Theorem~\ref{thm:smoothing} and Corollary~\ref{cor:global-smoothing-L1}
will only be used as a coarse intermediate estimate, and no sharpness claim is made for it.

\begin{prop}[Short-time sharpness]
\label{prop:optimality-short}
The short-time estimates in Theorems~\ref{thm:smoothing} and~\ref{thm:smoothing2} and Corollaries~\ref{cor:global-smoothing-L1} and~\ref{cor:global-smoothing-weighted} are uniformly sharp on $\mathbb{R}^n$, both in the time exponent and in the initial-data exponent.
\end{prop}

\begin{prop}[Green-weighted long-time sharpness]
\label{prop:optimality-weighted}
The long-time estimates in Theorem~\ref{thm:smoothing2} and Corollary~\ref{cor:global-smoothing-weighted} are uniformly sharp on $\mathbb{R}^n$, both in the time and initial-data exponents. Furthermore, the short- and long-time rates in Corollary~\ref{cor:global-smoothing-weighted} are simultaneously trajectory-wise sharp along a single solution.
\end{prop}

\begin{prop}[Sharpness of the refined $L^1$-estimate]
\label{prop:optimality-refined-L1}
On product manifolds $M=\mathbb{R}^k\times N$, with $k>2s$ and $N$ compact with nonnegative Ricci curvature, the refined long-time estimates in Theorem~\ref{thm:smoothing} and Corollary~\ref{cor:global-smoothing-L1} are trajectory-wise sharp in the time exponent and uniformly sharp in the $L^1$-norm exponent.
\end{prop}

\begin{cor}\label{cor:optimality-refined-L1}
Corollary~~\ref{cor:avr-decay} is sharp on $\mathbb{R}^n$, with trajectory-wise sharpness in the time exponent and uniform sharpness in the $L^1$-norm exponent.
\end{cor}

\begin{prop}[Local nature of Green-weighted smoothing]
\label{prop:optimality-local-global}
The local character of Theorem~\ref{thm:smoothing2} is optimal already on $\mathbb{R}^n$. More precisely, there exists $u_0\in L^1_{G^s_{\mathbb{R}^n},0}(\mathbb{R}^n)$ such that the corresponding weak dual solution satisfies
\[
\|u(t)\|_{L^\infty(\mathbb{R}^n)}=+\infty
\qquad\text{for every }t>0.
\]
\end{prop}

\begin{rem}
Proposition~\ref{prop:optimality-local-global} shows that the uniform-in-pole assumption in Corollary~\ref{cor:global-smoothing-weighted} is necessary for global smoothing.
\end{rem}

\begin{proof}[Proof of Proposition~\ref{prop:optimality-short}]
Let \(B_n\) be the unit-mass Barenblatt solution on \(\mathbb{R}^n\);
see \cite{Vazquez2014Barenblatt} if \(0<s<1\) and
\cite{Vazquez2007} if \(s=1\). Thus
\[
B_n(t,x)
=
t^{-\frac{n}{n(m-1)+2s}}
F_n\left(t^{-\frac{1}{n(m-1)+2s}}x\right),
\]
where
\[
F_n\in L^1(\mathbb{R}^n)\cap L^\infty(\mathbb{R}^n),
\qquad
F_n\geq0,
\qquad
\int_{\mathbb{R}^n}F_n\,dx=1,
\qquad
F_n(0)>0.
\]
Set
\[
v(t,x):=B_n(t+1,x),
\qquad
\phi(x):=B_n(1,x).
\]
For every \(A,R>0\), the scaling invariance of the equation yields the solution
\begin{equation*}
v_{A,R}(t,x)
:=
A\,v\left(A^{m-1}R^{-2s}t,\frac{x}{R}\right),
\qquad
v_{A,R}(0,x)
=
A\phi\left(\frac{x}{R}\right).
\end{equation*}

Since
\[
G^s_{\mathbb{R}^n}(x,y)
=
c_{n,s}|x-y|^{-(n-2s)},
\]
there exist \(R_*\in(0,1)\) and \(c,C>0\), independent of \(A,R\),
such that
\begin{equation}\label{eq:optimality-weighted-small}
cAR^n
\leq
\left\|
A\phi\left(\frac{\cdot}{R}\right)
\right\|_{L^1_{G^s_{\mathbb{R}^n},0}}
\leq
\left\|
A\phi\left(\frac{\cdot}{R}\right)
\right\|_{L^1_{G^s_{\mathbb{R}^n}}}
\leq
CAR^n
\end{equation}
for \(0<R\leq R_*\). Indeed,
\[
\int_{B_1(0)}
A\phi\left(\frac{x}{R}\right)\,dx
=
AR^n\int_{B_{1/R}(0)}\phi(z)\,dz
\sim AR^n,
\]
whereas, uniformly in \(x_0\in\mathbb{R}^n\),
\[
\int_{B_1(x_0)}
A\phi\left(\frac{x}{R}\right)\,dx
+
\int_{\mathbb{R}^n\setminus B_1(x_0)}
A\phi\left(\frac{x}{R}\right)
G^s_{\mathbb{R}^n}(x,x_0)\,dx
\leq
C
\left\|
A\phi\left(\frac{\cdot}{R}\right)
\right\|_{L^1(\mathbb{R}^n)}.
\]

Fix \(\lambda>0\) and take \(A=\lambda R^{-n}\). Then
\[
\|v_{A,R}(0)\|_{L^1(\mathbb{R}^n)}=\lambda,
\]
while both Green-weighted norms are comparable with \(\lambda\),
uniformly as \(R\downarrow0\). Setting
\[
t_{\lambda,R}
:=
\lambda^{1-m}R^{n(m-1)+2s},
\]
we obtain
\begin{equation}\label{eq:optimality-short}
\|v_{A,R}(t_{\lambda,R})\|_{L^\infty(\mathbb{R}^n)}
\geq
v_{A,R}\left(t_{\lambda,R},0\right)
=
B_n(2,0)
t_{\lambda,R}^{-\frac{n}{n(m-1)+2s}}
\lambda^{\frac{2s}{n(m-1)+2s}}.
\end{equation}
For \(K=\overline{B_1(0)}\), the same lower bound holds locally
\begin{equation}\label{eq1234}
\|v_{A,R}(t_{\lambda,R})\|_{L^\infty(K)}
\geq
B_n(2,0)
t_{\lambda,R}^{-\frac{n}{n(m-1)+2s}}
\lambda^{\frac{2s}{n(m-1)+2s}}.
\end{equation}
Furthermore,
\[
t_{\lambda,R}\|v_{A,R}(0)\|_{L^1}^{m-1}
=
R^{n(m-1)+2s},
\]
and, by \eqref{eq:optimality-weighted-small},
\[
t_{\lambda,R}
\|v_{A,R}(0)\|_{L^1_{G^s_{\mathbb{R}^n},0}}^{m-1}
\leq
CR^{n(m-1)+2s},\qquad
t_{\lambda,R}
\|v_{A,R}(0)\|_{L^1_{G^s_{\mathbb{R}^n}}}^{m-1}
\leq
CR^{n(m-1)+2s}.
\]
Hence \(t_{\lambda,R}\) lies in the short-time range of all four
estimates for \(R\) sufficiently small.

Since \(t_{\lambda,R}\downarrow0\), while the relevant initial-data
norms remain comparable with \(\lambda\),
\eqref{eq:optimality-short}--\eqref{eq1234} show that the time power $n/(n(m-1)+2s)$
cannot be decreased. Once this power is fixed, varying \(\lambda\) shows that the norm power $2s/(n(m-1)+2s)$
cannot be changed either. Thus all four short-time estimates are
uniformly sharp. 
\end{proof}

\begin{proof}[Proof of Proposition~\ref{prop:optimality-weighted}]
Let $\phi(x):=B_n(1,x)$, as in the proof of Proposition~\ref{prop:optimality-short}. There exist \(R_*>1\) and \(c,C>0\), independent of \(A,R\), such that
\begin{equation}\label{eq:optimality-weighted-large}
cAR^{2s}
\leq
\left\|
A\phi\left(\frac{\cdot}{R}\right)
\right\|_{L^1_{G^s_{\mathbb{R}^n},0}}
\leq
\left\|
A\phi\left(\frac{\cdot}{R}\right)
\right\|_{L^1_{G^s_{\mathbb{R}^n}}}
\leq
CAR^{2s}
\end{equation}
for \(R\geq R_*\). Indeed,
\[
\int_{\mathbb{R}^n\setminus B_1(0)}
A\phi\left(\frac{x}{R}\right)
G^s_{\mathbb{R}^n}(x,0)\,dx
=
c_{n,s}AR^{2s}
\int_{|z|\geq1/R}
\frac{\phi(z)}{|z|^{n-2s}}\,dz
\asymp AR^{2s},
\]
whereas
\[
\sup_{z_0\in\mathbb{R}^n}
\int_{\mathbb{R}^n}
\frac{\phi(z)}{|z-z_0|^{n-2s}}\,dz
<+\infty
\]
gives the corresponding upper bound.

Fix \(\lambda>0\) and take \(A=\lambda R^{-2s}\). By
\eqref{eq:optimality-weighted-large}, both Green-weighted norms of the
initial datum are comparable with \(\lambda\). Setting
\[
t_{\lambda,R}
:=
\lambda^{1-m}R^{2sm},
\]
we obtain
\begin{equation*}
\|v_{A,R}(t_{\lambda,R})\|_{L^\infty(\mathbb{R}^n)}
\geq
B_n(2,0)t_{\lambda,R}^{-1/m}\lambda^{1/m},
\end{equation*}
and, for \(K=\overline{B_1(0)}\),
\begin{equation*}
\|v_{A,R}(t_{\lambda,R})\|_{L^\infty(K)}
\geq
B_n(2,0)t_{\lambda,R}^{-1/m}\lambda^{1/m}.
\end{equation*}
Also, by \eqref{eq:optimality-weighted-large},
\[
t_{\lambda,R}
\|v_{A,R}(0)\|_{L^1_{G^s_{\mathbb{R}^n},0}}^{m-1}
\geq
cR^{2sm},
\qquad
t_{\lambda,R}
\|v_{A,R}(0)\|_{L^1_{G^s_{\mathbb{R}^n}}}^{m-1}
\geq
cR^{2sm}.
\]
Hence these times lie in the long-time range for \(R\) sufficiently
large. Since \(t_{\lambda,R}\to+\infty\), the decay power \(1/m\)
cannot be increased, and, once it is fixed, varying \(\lambda\) shows
that the norm power \(1/m\) cannot be changed either. Thus the
long-time estimates in Theorem~\ref{thm:smoothing2} and
Corollary~\ref{cor:global-smoothing-weighted} are uniformly sharp.

We now establish the trajectory-wise sharpness asserted in Proposition~\ref{prop:optimality-weighted} by constructing a single solution that attains both optimal time rates along suitable sequences. Let
\[
0\not\equiv g\in C_c^\infty(B_1(0)),
\qquad
g\geq0,
\]
and let \(w\) be the corresponding solution on \(\mathbb{R}^n\). Set
\[
c_*:=\|w(1)\|_{L^\infty(\mathbb{R}^n)}>0.
\]
The datum
\[
g_{a,R,x_0}(x)
:=
a\,g\left(\frac{x-x_0}{R}\right)
\]
generates the solution
\[
v_{a,R,x_0}(t,x)
=
a\,w\left(
a^{m-1}R^{-2s}t,
\frac{x-x_0}{R}
\right),
\]
so that, at \(t=a^{1-m}R^{2s}\),
\[
\|v_{a,R,x_0}(t)\|_{L^\infty(\mathbb{R}^n)}
=
ac_*.
\]

Let $x_k^\pm\in\mathbb{R}^n$ be translation parameters, to be chosen below. For \(r_k:=2^{-k}\) and \(a_k^-:=r_k^{-n}\), set
\[
\tau_k
:=
r_k^{n(m-1)+2s}\longrightarrow0.
\]
Then
\[
\left\|
v_{a_k^-,r_k,x_k^-}(\tau_k)
\right\|_{L^\infty}
=
c_*\tau_k^{-\frac{n}{n(m-1)+2s}}.
\]
For \(R_k:=2^k\) and \(a_k^+:=R_k^{-2s}\), set
\[
T_k:=R_k^{2sm}\longrightarrow+\infty.
\]
Then
\[
\left\|
v_{a_k^+,R_k,x_k^+}(T_k)
\right\|_{L^\infty}
=
c_*T_k^{-1/m}.
\]

The two families of initial bumps have uniformly bounded
\(L^1_{G^s_{\mathbb{R}^n}}\)-norm. For the first one this follows from
their uniformly bounded \(L^1\)-norm. For the second one it follows,
after scaling, from
\[
\sup_{z_0\in\mathbb{R}^n}
\int_{\mathbb{R}^n}
\frac{g(z)}{|z-z_0|^{n-2s}}\,dz
<+\infty.
\]

We shall use the elementary fact that any sequence
\((h_j)\subset L^1_{G^s_{\mathbb{R}^n}}(\mathbb{R}^n)\) of
nonnegative compactly supported functions with uniformly bounded
global Green-weighted norm can be translated sufficiently far apart so
that their sum still belongs to
\(L^1_{G^s_{\mathbb{R}^n}}(\mathbb{R}^n)\). Indeed, choose \(L_j\geq2\) so large that
\[
\operatorname{supp}h_j\subset B_{L_j}(0),
\qquad
c_{n,s}\|h_j\|_{L^1}L_j^{-(n-2s)}\leq2^{-j},
\]
and choose translations \(y_j\) such that the balls \(B_{2L_j}(y_j)\)
are pairwise disjoint. For every pole \(y\), at most one bump has
\(y\in B_{2L_j}(y_j)\); its contribution is uniformly bounded, while
all the others contribute at most \(\sum_j2^{-j}\).
Applying this observation to the preceding two families, we can choose
the centers \(x_k^\pm\) so that
\[
u_0
:=
\sum_{k=1}^\infty
g_{a_k^-,r_k,x_k^-}
+
\sum_{k=1}^\infty
g_{a_k^+,R_k,x_k^+}
\in
L^1_{G^s_{\mathbb{R}^n}}(\mathbb{R}^n).
\]
Let \(u\) be the corresponding weak dual solution obtained by monotone
approximation. By order preservation,
\[
u(t)\geq v_{a_k^-,r_k,x_k^-}(t),
\qquad
u(t)\geq v_{a_k^+,R_k,x_k^+}(t),
\]
and hence
\[
\tau_k^{\frac{n}{n(m-1)+2s}}
\|u(\tau_k)\|_{L^\infty(\mathbb{R}^n)}
\geq c_*,
\qquad
T_k^{1/m}
\|u(T_k)\|_{L^\infty(\mathbb{R}^n)}
\geq c_*.
\]
The corresponding upper bounds follow from
Corollary~\ref{cor:global-smoothing-weighted}.
Therefore, using also the Bénilan–Crandall monotonicity \eqref{prop:4.1} and the fact that both $(\tau_k)_k$ and $(T_k)_k$ are geometric sequences, the preceding lower bounds extend to all sufficiently small and sufficiently large times, respectively. Together with the upper bounds of Corollary \ref{cor:global-smoothing-weighted}, this yields
\[
0<\liminf_{t\to0}
t^{\frac{n}{n(m-1)+2s}}
\|u(t)\|_{L^\infty(\mathbb{R}^n)}
\leq\limsup_{t\to0}
t^{\frac{n}{n(m-1)+2s}}
\|u(t)\|_{L^\infty(\mathbb{R}^n)}
<+\infty
\]
and
\[
0<
\liminf_{t\to+\infty}
t^{1/m}
\|u(t)\|_{L^\infty(\mathbb{R}^n)}
\leq
\limsup_{t\to+\infty}
t^{1/m}
\|u(t)\|_{L^\infty(\mathbb{R}^n)}
<+\infty.
\]
Thus both time exponents in
Corollary~\ref{cor:global-smoothing-weighted} are sharp along one and
the same solution.
\end{proof}

\begin{proof}[Proof of Proposition~\ref{prop:optimality-refined-L1}]
For the product manifold $M=\mathbb{R}^k\times N$ introduced in the statement, functions independent of the $N$-variable satisfy
\[
(-\Delta_M)^s w = (-\Delta_{\mathbb{R}^k})^s w.
\]
For \(0<s<1\) this follows from factorization of the heat semigroup and
conservativity on \(N\), while for \(s=1\) it is immediate from the
product formula for the Laplacian.

Let \(B_{k,\mathcal M}\) be the Euclidean Barenblatt solution of mass
\(\mathcal M>0\). Then
\[
B_{k,\mathcal M}(t,x)
=
\mathcal M^{\frac{2s}{k(m-1)+2s}}
B_{k,1}\left(
t,
\mathcal M^{-\frac{m-1}{k(m-1)+2s}}x
\right),
\]
and therefore
\begin{equation*}
B_{k,\mathcal M}(t,0)
=
c_k
t^{-\frac{k}{k(m-1)+2s}}
\mathcal M^{\frac{2s}{k(m-1)+2s}},
\qquad
c_k>0.
\end{equation*}

Fix \(\varepsilon>0\) and set
\[
u(t,(x,p))
:=
B_{k,\mathcal M}(t+\varepsilon,x).
\]
This is a weak dual solution on \(M\) with
\[
u(0)\in L^1(M)\cap L^\infty(M),
\qquad
\|u(0)\|_{L^1(M)}
=
\mathcal M\operatorname{Vol}(N).
\]
If \(K\Subset M\) is a compact neighborhood of a point \((0,p)\), then
\begin{equation}\label{eq:optimality-product-lower}
\|u(t)\|_{L^\infty(M)}
\geq
\|u(t)\|_{L^\infty(K)}
\geq
c\,
t^{-\frac{k}{k(m-1)+2s}}
\mathcal M^{\frac{2s}{k(m-1)+2s}}
\end{equation}
for all sufficiently large \(t\). By Example~\ref{exm:last},
\[
\operatorname{Vol}(B_R(x_0))\asymp R^k
\qquad (R\to\infty),
\]
uniformly in \(x_0\). Hence
\[
h_o(R)\asymp h(R)\asymp R^{2s},
\qquad
F(R)\asymp R^k,\qquad
\theta_o(R)\asymp\theta(R)
\asymp
R^{k+\frac{2s}{m-1}}.
\]
Both refined estimates therefore reduce, for large times, to
\[
\|u(t)\|_{L^\infty}
\leq
C
t^{-\frac{k}{k(m-1)+2s}}
\|u(0)\|_{L^1(M)}^{\frac{2s}{k(m-1)+2s}}.
\]
Comparison with \eqref{eq:optimality-product-lower} proves
trajectory-wise sharpness of the time exponent. Varying
\(\mathcal M\) then proves optimality of the \(L^1\)-norm exponent.
\end{proof}

\begin{proof}[Proof of Corollary~\ref{cor:optimality-refined-L1}]
The conclusion follows from Proposition~\ref{prop:optimality-refined-L1} by taking $k=n$ and $N$ to be a single point.
\end{proof}

\begin{proof}[Proof of Proposition~\ref{prop:optimality-local-global}]
Let
\[
0\not\equiv g\in C_c^\infty(B_1(0)),
\qquad
g\geq0,
\]
let \(w\) be the corresponding solution, and set
\[
c_*:=\|w(1)\|_{L^\infty(\mathbb{R}^n)}>0,
\qquad
\tau_k:=2^{-k}.
\]
For \(k,j\geq1\), let
\[
R_{k,j}
:=
\left(\tau_kj^{m-1}\right)^{1/(2s)},
\qquad
g_{k,j}(x)
:=
j\,g\left(\frac{x-x_{k,j}}{R_{k,j}}\right).
\]
The corresponding solution is
\[
v_{k,j}(t,x)
=
j\,w\left(
j^{m-1}R_{k,j}^{-2s}t,
\frac{x-x_{k,j}}{R_{k,j}}
\right),
\]
so that
\[
\|v_{k,j}(\tau_k)\|_{L^\infty(\mathbb{R}^n)}
=
jc_*.
\]

Since
\[
G^s_{\mathbb{R}^n}(x,0)
=
c_{n,s}|x|^{-(n-2s)},
\]
the centers \(x_{k,j}\) can be chosen sufficiently far from the origin
and from one another so that the supports are pairwise disjoint and
\[
\|g_{k,j}\|_{L^1_{G^s_{\mathbb{R}^n},0}}
\leq
2^{-k-j}.
\]
Hence
\[
u_0:=\sum_{k,j\geq1}g_{k,j}
\in
L^1_{G^s_{\mathbb{R}^n},0}(\mathbb{R}^n).
\]
By order preservation,
\[
u(\tau_k)\geq v_{k,j}(\tau_k)
\]
for every \(j\), and therefore
\[
\|u(\tau_k)\|_{L^\infty(\mathbb{R}^n)}
=
+\infty
\qquad\text{for every }k.
\]
Since \(\tau_k\downarrow0\), the Bénilan--Crandall monotonicity
\eqref{eq:prop4.1mon} yields
\[
\|u(t)\|_{L^\infty(\mathbb{R}^n)}
=
+\infty
\qquad\text{for every }t>0.
\]
This completes the proof.
\end{proof}

\section{The Main Results in the Range $0<s<1/2$}
\label{subsec:subcritical-range}

In the range $0<s<1/2$, Assumption \ref{A2} and
$s$-nonparabolicity follow directly from the geometry, as we show below in Proposition \ref{prop:subcritical-A2}. Hence, in such range our results hold without assumptions. We state them in the present Section, together with the modifications needed in the proofs.

\begin{prop}\label{prop:subcritical-A2}
Let $M$ be a complete, noncompact, $n$-dimensional Riemannian
manifold without boundary, with $\mathrm{Ric}\geq0$, and let
$0<s<1/2$. Then Assumption \ref{A2} holds with
\[
R_0=1,
\qquad
f(R)=R^{2-2s}
\qquad\textrm{for all } R\geq1.
\]
In particular, $M$ is $s$-nonparabolic.
\end{prop}

\begin{proof}
By the scale-invariant Calabi--Yau volume growth estimate
\cite{Yau1976,PLi}, there exists $c_n>0$ such that
\begin{equation}\label{eq:subcritical-reverse-doubling}
\operatorname{Vol}(B_{R_2}(x))
\geq
c_n\frac{R_2}{R_1}
\operatorname{Vol}(B_{R_1}(x))
\end{equation}
for every $x\in M$ and $1\leq R_1\leq R_2$. Since $R^{2s-1}f(R)=R$, \eqref{eq:subcritical-reverse-doubling} gives
\[
\frac{R_2^{2s-1}f(R_2)}
{\operatorname{Vol}(B_{R_2}(x))}
\leq
\frac{1}{c_n}
\frac{R_1^{2s-1}f(R_1)}
{\operatorname{Vol}(B_{R_1}(x))}
\]
for every $x\in M$ and $1\leq R_1\leq R_2$, which is \eqref{fcondition}. One also immediately sees that \eqref{betacondition} holds, since $s\in(0,1/2)$. Thus
Assumption \ref{A2} is satisfied, and
$s$-nonparabolicity follows from a result of \cite{GuHuangSun} (see Theorem~\ref{thm:charnonpar}), since for every $x\in M$ by \eqref{eq:subcritical-reverse-doubling}
$$
\int_{1}^{\infty} \frac{t^{2s-1}}{\mathrm{Vol}(B_t(x))}\,dt\leq \int_{1}^{\infty} \frac{t^{2s-2}}{c_n\mathrm{Vol}(B_1(x))}\,dt\leq\frac{1}{c_n(1-2s)\mathrm{Vol}(B_1(x))}<+\infty\,.
$$

\end{proof}

We are now ready to state and prove our results in the case $s\in(0,1/2)$.

\begin{cor}\label{cor:subcritical-existence}
Let $M$ be a complete, noncompact, $n$-dimensional Riemannian
manifold without boundary, with $\mathrm{Ric}\geq0$, let
$0<s<1/2$ and $o\in M$. Then, for
every nonnegative initial datum $u_0\in L^1_{G_M^s,o}(M)$
there exists a weak dual solution to \eqref{maineq}.
\end{cor}

\begin{proof}
By Proposition~\ref{prop:subcritical-A2}, the manifold is
$s$-nonparabolic. The conclusion is therefore an immediate
consequence of Theorem~\ref{thm:existence}.
\end{proof}

The same geometric input also makes the smoothing results of
Section~\ref{S3} more explicit in the range $0<s<1/2$.
By Proposition~\ref{prop:subcritical-A2}, $s$-nonparabolicity is
automatic under the sole assumption $\operatorname{Ric}\geq0$, so the local
smoothing results of Theorems~\ref{thm:smoothing} and
\ref{thm:smoothing2} apply without any further geometric hypotheses. If Assumption~\ref{A1} is also imposed, Proposition~\ref{prop:subcritical-A2}
automatically yields Assumption~\ref{A2}, so the global smoothing results of
Corollaries~\ref{cor:global-smoothing-L1} and
\ref{cor:global-smoothing-weighted} apply as well.
The scale-invariant Calabi--Yau volume growth estimate also makes
the refined $L^1$-long-time decay explicit.

\begin{cor}[Local smoothing in the range $0<s<1/2$]
\label{cor:subcritical-local-smoothing}
Let $M$ be a complete, noncompact, $n$-dimensional Riemannian
manifold without boundary, with $\operatorname{Ric}\geq0$, and let $0<s<1/2$.
Then the following local smoothing estimates hold.
\begin{enumerate}
\item
For every nonempty compact set $K\Subset M$ there exists $C_K>0$ such that, if $u_0\in L^1(M)$ is nonnegative,
\[
\|u(t)\|_{L^\infty(K)}
\leq
C_K
t^{-\frac{n}{n(m-1)+2s}}
\|u_0\|_{L^1(M)}^{\frac{2s}{n(m-1)+2s}}\qquad\textrm{ for every }
0<t\leq
\|u_0\|_{L^1(M)}^{-(m-1)},
\]
whereas
\[
\|u(t)\|_{L^\infty(K)}
\leq
C_K
t^{-\frac{1}{m-1+2s}}
\|u_0\|_{L^1(M)}^{\frac{2s}{m-1+2s}}\qquad\textrm{ for every }
t\geq
\|u_0\|_{L^1(M)}^{-(m-1)}.
\]

\item
Fix $o\in M$. For every nonempty compact set $K\Subset M$, there exists
$C_{K,o}>0$ such that if $u_0\in L^1_{G_M^s,o}(M)$ is nonnegative
\[
\|u(t)\|_{L^\infty(K)}
\leq
C_{K,o}
t^{-\frac{n}{n(m-1)+2s}}
\|u_0\|_{L^1_{G_M^s,o}(M)}^{
\frac{2s}{n(m-1)+2s}}\qquad\textrm{ for every }
0<t\leq
\|u_0\|_{L^1_{G_M^s,o}(M)}^{-(m-1)},
\]
and
\[
\|u(t)\|_{L^\infty(K)}
\leq
C_{K,o}
t^{-1/m}
\|u_0\|_{L^1_{G_M^s,o}(M)}^{1/m}\qquad\textrm{ for every }
t\geq
\|u_0\|_{L^1_{G_M^s,o}(M)}^{-(m-1)}.
\]
\end{enumerate}
\end{cor}

\begin{proof}
The short-time estimate in part~(1) and both estimates in part~(2)
follow directly from Theorems~\ref{thm:smoothing} and
\ref{thm:smoothing2}, since Proposition~\ref{prop:subcritical-A2}
implies that $M$ is $s$-nonparabolic.

We prove the improved long-time estimate in part~(1).
Fix a nonempty compact set $K\Subset M$ and choose $o\in K$.
With the notation of Theorem~\ref{thm:smoothing}, for $R\geq1$,
\[
h_o(R)
=
\operatorname{Vol}(B_R(o))
\int_R^{+\infty}
\frac{r^{2s-1}}{\operatorname{Vol}(B_r(o))}\,dr
+
\frac{R^{2s}}{2s}.
\]
By~\eqref{eq:subcritical-reverse-doubling}, for every
$r\geq R\geq1$,
\[
\operatorname{Vol}(B_r(o))
\geq
c_n\frac{r}{R}\operatorname{Vol}(B_R(o)).
\]
Therefore
\[
h_o(R)
\leq
\frac{R}{c_n}
\int_R^{+\infty}r^{2s-2}\,dr
+
\frac{R^{2s}}{2s}
=
\left(
\frac{1}{c_n(1-2s)}
+
\frac{1}{2s}
\right)R^{2s}.
\]
On the other hand,
\[
h_o(R)\geq\frac{R^{2s}}{2s}.
\]
Again by~\eqref{eq:subcritical-reverse-doubling},
\[
\operatorname{Vol}(B_R(o))
\geq
c_n R\operatorname{Vol}(B_1(o)),
\qquad R\geq1.
\]
Hence
\[
\theta_o(R)
=
\operatorname{Vol}(B_R(o))
h_o(R)^{\frac{1}{m-1}}
\geq
c_{n,s,m}\operatorname{Vol}(B_1(o))
R^{1+\frac{2s}{m-1}},
\qquad R\geq1.
\]
Consequently,
\[
\theta_o^{-1}(r)
\leq
C_o
r^{\frac{m-1}{m-1+2s}},
\qquad
r\geq\theta_o(1).
\]
Using also the preceding upper bound for $h_o$, the refined estimate
in Theorem~\ref{thm:smoothing} gives
\[
\begin{aligned}
\|u(t)\|_{L^\infty(K)}
&\leq
C
t^{-\frac{1}{m-1}}
\left[
h_o\!\left(
\theta_o^{-1}\!\left(
t^{\frac{1}{m-1}}\|u_0\|_{L^1(M)}
\right)
\right)
\right]^{\frac{1}{m-1}}
\\
&\leq
C
t^{-\frac{1}{m-1}}
\left(
t^{\frac{1}{m-1}}\|u_0\|_{L^1(M)}
\right)^{\frac{2s}{m-1+2s}}
=
C
t^{-\frac{1}{m-1+2s}}
\|u_0\|_{L^1(M)}^{\frac{2s}{m-1+2s}}
\end{aligned}
\]
for every $t\geq\theta_o(R_{K,o})^{m-1}\|u_0\|_{L^1(M)}^{-(m-1)}$.

It remains to extend this estimate down to
$t=\|u_0\|_{L^1(M)}^{-(m-1)}$.
If $\theta_o(R_{K,o})^{m-1}\leq1$ there is nothing to prove. Otherwise, for almost every $\|u_0\|_{L^1(M)}^{-(m-1)}\leq t\leq
\theta_o(R_{K,o})^{m-1} \|u_0\|_{L^1(M)}^{-(m-1)}$,
the basic long-time estimate in Theorem~\ref{thm:smoothing} yields
\[
\|u(t)\|_{L^\infty(K)}
\leq
C_K
t^{-1/m}
\|u_0\|_{L^1(M)}^{1/m}.
\]
Since $0<s<1/2$,
\[
\frac{1}{m-1+2s}>\frac{1}{m},
\]
and throughout the preceding time interval
\[
\begin{aligned}
t^{-1/m}\|u_0\|_{L^1(M)}^{1/m}
&=
\left(
t\|u_0\|_{L^1(M)}^{m-1}
\right)^{-1/m}
\|u_0\|_{L^1(M)}
\\
&\leq
\left(
\theta_o(R_{K,o})^{m-1}
\right)^{
\frac{1}{m-1+2s}-\frac{1}{m}
}
\left(
t\|u_0\|_{L^1(M)}^{m-1}
\right)^{-\frac{1}{m-1+2s}}
\|u_0\|_{L^1(M)}
\\
&=
\left(
\theta_o(R_{K,o})^{m-1}
\right)^{
\frac{1}{m-1+2s}-\frac{1}{m}
}
t^{-\frac{1}{m-1+2s}}
\|u_0\|_{L^1(M)}^{\frac{2s}{m-1+2s}}.
\end{aligned}
\]
The fixed factor can be absorbed into the constant.
Since the reference point $o$ is fixed in $K$, all constants arising in the argument
may be absorbed into a constant depending only on $K$. This proves part~(1).
\end{proof}

If the manifold is also noncollapsing,
Proposition~\ref{prop:subcritical-A2} gives the hypotheses needed
for the global results of Section~3, and the preceding estimates
become global.

\begin{cor}[Global smoothing in the range $0<s<1/2$]
\label{cor:subcritical-global-smoothing}
Let $M$ be a complete, noncompact, $n$-dimensional Riemannian
manifold without boundary, with $\operatorname{Ric}\geq0$, let $0<s<1/2$, and
assume that $M$ satisfies Assumption \ref{A1}.
Then the following statements hold for every weak dual solution

\begin{enumerate}
\item
If $u_0\in L^1(M)$ is nonnegative, then there exists $C>0$ such that
\[
\|u(t)\|_{L^\infty(M)}
\leq
C
t^{-\frac{n}{n(m-1)+2s}}
\|u_0\|_{L^1(M)}^{\frac{2s}{n(m-1)+2s}}\qquad\text{ for every }0<t\leq
\|u_0\|_{L^1(M)}^{-(m-1)},
\]
whereas
\[
\|u(t)\|_{L^\infty(M)}
\leq
C
t^{-\frac{1}{m-1+2s}}
\|u_0\|_{L^1(M)}^{\frac{2s}{m-1+2s}}\qquad\text{ for every }
t\geq
\|u_0\|_{L^1(M)}^{-(m-1)}.
\]

\item
If $u_0\in L^1_{G_M^s}(M)$ is nonnegative, then there exists
$C>0$ such that
\[
\|u(t)\|_{L^\infty(M)}
\leq
C
t^{-\frac{n}{n(m-1)+2s}}
\|u_0\|_{L^1_{G_M^s}(M)}^{
\frac{2s}{n(m-1)+2s}}\qquad\textrm{ for every }
0<t\leq
\|u_0\|_{L^1_{G_M^s}(M)}^{-(m-1)},
\]
and
\[
\|u(t)\|_{L^\infty(M)}
\leq
C
t^{-1/m}
\|u_0\|_{L^1_{G_M^s}(M)}^{1/m}\qquad\textrm{ for every }
t\geq
\|u_0\|_{L^1_{G_M^s}(M)}^{-(m-1)}.
\]
\end{enumerate}
\end{cor}

\begin{proof}
By Proposition~\ref{prop:subcritical-A2}, Assumption \ref{A2} holds
with
\[
R_0=1,
\qquad
f(R)=R^{2-2s},
\qquad R\geq1.
\]
Hence Corollaries~\ref{cor:global-smoothing-L1} and
\ref{cor:global-smoothing-weighted} apply.
The short-time estimate in part~(1) and both estimates in part~(2)
follow directly from them. For the refined $L^1$-long-time estimate, the function in
\eqref{defnfnch} is given exactly by
\[
h(R)
=
R^{2s-1}R^{2-2s}
\int_R^{+\infty}r^{-(2-2s)}\,dr
+
\frac{R^{2s}}{2s}
=
\frac{R^{2s}}{2s(1-2s)},
\qquad R\geq1.
\]
Furthermore, by~\eqref{eq:subcritical-reverse-doubling} and
Assumption \ref{A1},
\[
F(R)\geq c_n\alpha R,
\qquad R\geq1.
\]
Thus
\[
\theta(R)
=
F(R)h(R)^{\frac{1}{m-1}}
\geq
C\alpha
R^{1+\frac{2s}{m-1}},
\qquad R\geq1,
\]
and consequently
\[
\theta^{-1}(r)
\leq
C r^{\frac{m-1}{m-1+2s}},
\qquad r\geq\theta(1).
\]
Substituting this estimate and the exact expression for $h$ into
\eqref{eq:improved} yields
\[
\|u(t)\|_{L^\infty(M)}
\leq
C
t^{-\frac{1}{m-1+2s}}
\|u_0\|_{L^1(M)}^{\frac{2s}{m-1+2s}}
\]
for every $t\geq
\theta(1)^{m-1}
\|u_0\|_{L^1(M)}^{-(m-1)}$. Since $F(1)=\alpha$,
\[
\theta(1)^{m-1}
=
F(1)^{m-1}h(1)
=
\frac{\alpha^{m-1}}{2s(1-2s)}.
\]

If $\theta(1)^{m-1}\leq1$, the asserted long-time range follows immediately.
If instead $\theta(1)^{m-1}>1$, then, for every $\|u_0\|_{L^1(M)}^{-(m-1)}
\leq t\leq\theta(1)^{m-1}\|u_0\|_{L^1(M)}^{-(m-1)}$, we use the basic long-time estimate from
Corollary~\ref{cor:global-smoothing-L1}.
As in the proof of
Corollary~\ref{cor:subcritical-local-smoothing},
the inequality $\frac{1}{m-1+2s}>\frac{1}{m}$
allows the resulting fixed factor to be absorbed into $C$.
Hence
\[
\|u(t)\|_{L^\infty(M)}
\leq
C
t^{-\frac{1}{m-1+2s}}
\|u_0\|_{L^1(M)}^{\frac{2s}{m-1+2s}}
\]
for every $t\geq\|u_0\|_{L^1(M)}^{-(m-1)}$.
\end{proof}

\begin{rmk}
The exponent $\frac{1}{m-1+2s}$ in the long-time $L^1$-estimates
of Corollaries~\ref{cor:subcritical-local-smoothing} and
\ref{cor:subcritical-global-smoothing} is sharp within this class.
Indeed, the argument of Section~3.1 applies to
$M=\mathbb{R}\times N$, with $N$ compact and $\operatorname{Ric}_N\geq0$, since
$1>2s$.
\end{rmk}

\section{A Priori Estimates}\label{sec:apriori}
We first state and prove some results on the existence of the Green function and on some related integrals.

\begin{lem}\label{lemlastlem}
Let \(M\) be a complete, noncompact, \(s\)-nonparabolic
\(n\)-dimensional Riemannian manifold, $n>2s$, with
\(\operatorname{Ric}\geq 0\). Then the following properties hold.
\begin{enumerate}
\item For every compact set \(V\subset M\) and every
\(\varepsilon>0\), there exists a constant
\(C=C(V,\varepsilon)>0\) such that
\begin{equation}\label{lastlemeq2}
\sup_{x\in V}\sup_{y\in M\setminus B_\varepsilon(x)}
G_M^s(x,y)
\leq C.
\end{equation}

\item For every \(\psi\in L^1(M)\cap L^\infty(M)\), the integral in $\eqref{defform}$
defining \((-\Delta)^{-s}\psi(x)\) is absolutely convergent for every
\(x\in M\). In addition, for every compact set \(V\subset M\), there
exists a constant \(C=C(V)>0\) such that
\begin{equation}\label{lastlemeq3}
\|(-\Delta)^{-s}\psi\|_{L^\infty(V)}
\leq
C\bigl(
\|\psi\|_{L^1(M)}+\|\psi\|_{L^\infty(M)}
\bigr).
\end{equation}
In particular,
\[
(-\Delta)^{-s}\psi\in L^\infty_{\mathrm{loc}}(M).
\]

\item For every compact set \(V\subset M\), there exists a constant
\(C=C(V)>0\) such that every \(\psi\in L^\infty(M)\) with
\(\operatorname{supp}\psi\subset V\) satisfies
\begin{equation}\label{lastlemeq4}
\|(-\Delta)^{-s}\psi\|_{L^\infty(M)}
\leq
C\|\psi\|_{L^\infty(M)}.
\end{equation}
\end{enumerate}
\end{lem}

\begin{proof}
We first establish a uniform tail estimate. Let \(K\subset M\) be
compact, let \(\varepsilon>0\), fix \(o\in M\), and set
\[
D:=\max_{x\in K}d(o,x),
\qquad
\rho:=\max\{\varepsilon,2D\}.
\]
By compactness, there exists \(\nu_{K,\varepsilon}>0\) such that
\[
\operatorname{Vol}(B_t(x))
\geq \nu_{K,\varepsilon}
\]
for every $x\in K$, $t\in[\varepsilon,\rho]$. If \(t\geq\rho\), then
\[
B_{t/2}(o)\subset B_t(x)
\qquad
\text{for every }x\in K.
\]
Therefore, by the volume doubling property,
\[
\frac{1}{\operatorname{Vol}(B_t(x))}
\leq
\frac{1}{\operatorname{Vol}(B_{t/2}(o))}
\leq
\frac{C_d}{\operatorname{Vol}(B_t(o))}.
\]
It follows that
\[
\sup_{x\in K}
\int_\varepsilon^\infty
\frac{t^{2s-1}}{\operatorname{Vol}(B_t(x))}\,dt
\leq
\frac{1}{\nu_{K,\varepsilon}}
\int_\varepsilon^\rho t^{2s-1}\,dt
+
C_d\int_\rho^\infty
\frac{t^{2s-1}}{\operatorname{Vol}(B_t(o))}\,dt
<\infty,
\]
where the last integral is finite by \(s\)-nonparabolicity and
Theorem~\ref{thm:charnonpar}. The upper Green-function estimate
recalled in Theorem~\ref{thm:bounGreen's} now yields
\eqref{lastlemeq2}.

We next prove the potential estimate. Fix a compact set \(V\subset M\).
As above,
\[
\nu_V
:=
\inf_{x\in V}\operatorname{Vol}(B_1(x))
>0.
\]
For every \(x\in V\) and every \(0<t\leq1\), the Bishop--Gromov
comparison gives
\[
\operatorname{Vol}(B_t(x))
\geq
t^n\operatorname{Vol}(B_1(x))
\geq
\nu_V t^n.
\]
Combining this estimate with the uniform tail estimate proved above,
we obtain, for \(x\in V\) and \(0<d(x,y)<1\),
\begin{equation*}
G_M^s(x,y)
\leq
C\left(
\int_{d(x,y)}^1
\frac{t^{2s-1}}{\operatorname{Vol}(B_t(x))}\,dt
+
\int_1^\infty
\frac{t^{2s-1}}{\operatorname{Vol}(B_t(x))}\,dt
\right)
\leq
C(V)d(x,y)^{-(n-2s)}.
\end{equation*}
Together with
\eqref{lastlemeq2}, this gives a constant \(C(V)>0\) such that for every $x\in V$, $y\in M$
\[
G_M^s(x,y)
\leq
C(V)
\begin{cases}
d(x,y)^{-(n-2s)},
&0<d(x,y)<1,\\
1,
&d(x,y)\geq1.
\end{cases}
\]
Consequently, for every \(x\in V\),
\begin{equation*}
\int_M|\psi(y)|G_M^s(x,y)\,d\mu(y)
\leq
C(V)\|\psi\|_{L^\infty(M)}
\int_{B_1(x)}
\frac{d\mu(y)}{d(x,y)^{n-2s}}
+
C(V)\|\psi\|_{L^1(M)}.
\end{equation*}
By the coarea formula and the Bishop--Gromov upper bound,
\begin{align*}
\int_{B_1(x)}
\frac{d\mu(y)}{d(x,y)^{n-2s}}
&=
\operatorname{Vol}(B_1(x))
+
(n-2s)\int_0^1
\frac{\operatorname{Vol}(B_r(x))}
{r^{n+1-2s}}\,dr
\\
&\leq
\omega_n
\left(
1+(n-2s)\int_0^1r^{2s-1}\,dr
\right)
=
\omega_n\frac{n}{2s}.
\end{align*}
This proves the absolute convergence of the potential and
\eqref{lastlemeq3}.

Finally, fix a compact set \(V\subset M\) and let
\(\psi\in L^\infty(M)\) satisfy
\(\operatorname{supp}\psi\subset V\). If \(\psi\equiv0\), there is
nothing to prove. Otherwise, we apply estimate \eqref{lastlemeq3} on
\[
V_1:=\{x\in M:d(x,V)\leq1\}.
\]
On the other hand, if \(x\notin V_1\), then \(d(x,y)>1\) for every \(y\in V\). Hence, by
\eqref{lastlemeq2} applied with
the compact set \(V\) and \(\varepsilon=1\),
\begin{equation*}
|(-\Delta)^{-s}\psi(x)|
\leq
\int_V|\psi(y)|G_M^s(y,x)\,d\mu(y)
\leq
C(V_1)\|\psi\|_{L^1(M)}.
\end{equation*}
Combining the two regions $V_1, M\setminus V_1$ gives
\begin{equation*}
\|(-\Delta)^{-s}\psi\|_{L^\infty(M)}
\leq
C\bigl(
\|\psi\|_{L^1(M)}+\|\psi\|_{L^\infty(M)}
\bigr)\leq C\|\psi\|_{L^\infty(M)}
\end{equation*}
with $C>0$ depending only on $V$, that is \eqref{lastlemeq4}.
\end{proof}

\begin{lem}
Let \(M\) be a complete, noncompact, \(n\)-dimensional Riemannian
manifold without boundary, with \(\operatorname{Ric}\geq0\), satisfying
Assumption \ref{A2}. Then \(M\) is \(s\)-nonparabolic.
\end{lem}

\begin{proof}
Fix \(x\in M\). Applying \eqref{fcondition} with
\(R_1=R_0\) and \(R_2=t\geq R_0\), we obtain
\[
\frac{t^{2s-1}f(t)}
{\operatorname{Vol}(B_t(x))}
\leq
\gamma
\frac{R_0^{2s-1}f(R_0)}
{\operatorname{Vol}(B_{R_0}(x))}.
\]
Consequently,
\begin{equation*}
\int_{R_0}^\infty
\frac{t^{2s-1}}{\operatorname{Vol}(B_t(x))}\,dt
\leq
\gamma
\frac{R_0^{2s-1}f(R_0)}
{\operatorname{Vol}(B_{R_0}(x))}
\int_{R_0}^\infty\frac{dt}{f(t)}
<\infty.
\end{equation*}
The conclusion follows from
Theorem~\ref{thm:charnonpar}.
\end{proof}

\section{Examples of Manifolds Satisfying the Global Geometric Assumptions}\label{S6}

We provide several examples satisfying Assumptions \ref{A1}, \ref{A2}.

%

\begin{exm}
Let $M$ be a complete, noncompact Riemannian manifold with $\mathrm{Ric} \ge 0$. Assume that there exist $C_1,C_2 > 0$, $R_0> 1$ large such that
\begin{equation*}
C_1 R^k (\log R)^b \le \mathrm{Vol}(B_R(x)) \le C_2 R^k (\log R)^b \quad \text{for all } R \ge R_0,
\end{equation*}
for every $x\in M$ and that either $k \in (2s, n)$, $b \in \mathbb{R}$; $k = 2s<n$, $b > 1$; or $k = n>2s$, $b \le 0$. Then \ref{A2} holds with
\begin{equation*}
f(R) = R^{k-(2s-1)} (\log R)^b, \quad R \ge R_0.
\end{equation*}

Indeed
\begin{equation*}
\int_{R_0}^{\infty} \frac{dt}{f(t)} \leq C\int_{R_0}^{\infty} \frac{dt}{t^{k-2s+1} (\log t)^b},
\end{equation*}
which converges for the above parameters.

Further, for every $x \in M$ and $R_2 \geq R_1 \geq R_0$,
\begin{equation*}
\frac{R_2^{2s-1} f(R_2)}{\mathrm{Vol}(B_{R_2}(x))} = \frac{R_2^k (\log R_2)^b}{\mathrm{Vol}(B_{R_2}(x))} \le \frac{1}{C_1} \le \frac{C_2}{C_1} \frac{R_1^k (\log R_1)^b}{\mathrm{Vol}(B_{R_1}(x))} = \gamma \frac{R_1^{2s-1} f(R_1)}{\mathrm{Vol}(B_{R_1}(x))},
\end{equation*}
so \eqref{fcondition} holds with $\gamma = \frac{C_2}{C_1}$. Finally, Assumption \ref{A1} holds. Indeed, by the Bishop--Gromov comparison \eqref{bishopgromov}, for every $x\in M$ we have
$$\operatorname{Vol}(B_1(x))\geq\frac{\operatorname{Vol}(B_{R_0}(x))}{R_0^n}\geq C_1R_0^{k-n}(\log R_0)^b.$$
\end{exm}

\begin{exm}
Let $M$ be a complete, noncompact Riemannian manifold with $n = \dim M$,  $n> 2s$, $s\in(0,1]$, $\mathrm{Ric} \ge 0$ and maximal volume growth, i.e. with asymptotic volume ratio
\begin{equation*}
A \equiv \mathrm{AVR}(M) := \lim_{R \to \infty} \frac{\mathrm{Vol}(B_R(x))}{\omega_n R^n} > 0,
\end{equation*}
independent of $x \in M$. By the Bishop--Gromov comparison \eqref{bishopgromov},
\begin{equation*}
A \, \omega_n R^n \le \mathrm{Vol}(B_R(x)) \le \omega_n R^n \quad \text{for all } x\in M, \, R>0.
\end{equation*}
Thus, the noncollapsing property \ref{A1} holds, and \ref{A2} is satisfied with $f(R) = R^{n-(2s-1)}$, $R_0=1$, $\gamma=A^{-1}$.
\end{exm}

\begin{exm}
Let $M$ be a complete, noncompact, $s$-nonparabolic Riemannian manifold with $\partial M = \emptyset$ and $\mathrm{Ric} \ge 0$, and suppose $M$ is homogeneous, i.e., for every $x, y \in M$ there exists an isometry $\phi : M \to M$ with $\phi(x) = y$. Then $M$ satisfies the noncollapsing property \ref{A1}: for any $x, y \in M$ and $R > 0$, $\phi(B_R(x)) = B_R(y)$ for some isometry $\phi$, so $\mathrm{Vol}(B_R(x)) = \mathrm{Vol}(B_R(y))$. In particular, one can consider \( R = 1 \) and \ref{A1} holds.

Furthermore, arguing as in Proposition \ref{general}, we see that Assumption \ref{A2} is satisfied, with $f$ given by \eqref{deffo1} for any fixed point $o\in M$ and $R_0=1$.
\end{exm}

\begin{exm}
Let $M$ be a complete, noncompact Riemannian manifold with $n = \dim M$,  $n> 2s$, $s\in(0,1]$, with $\mathrm{Ric} \ge 0$ and satisfying the noncollapsing property \ref{A1}. Assume that the following Reverse Doubling Property holds: there exist $\theta>2s$, $R_0\geq 1$ such that for every $x\in M$
\begin{equation}\label{epsilonexm}
\mathrm{Vol}(B_{2R}(x)) \geq 2^{\theta} \mathrm{Vol}(B_R(x)) \quad \text{for all } R \geq R_0.
\end{equation}
Then  \ref{A2} is satisfied, with $f(R) = R^{\theta-(2s-1)}$. Indeed, $f$ trivially satisfies \eqref{betacondition}. Further, if $x \in M$, $R_2 \geq R_1 \geq R_0$ let $k \in \mathbb{N}$ be such that $2^k R_1 \le R_2 < 2^{k+1} R_1$. Then, by \eqref{epsilonexm},
\begin{equation*}
\mathrm{Vol}(B_{R_2}(x)) \ge \mathrm{Vol}(B_{2^k R_1}(x)) \ge 2^{\theta k} \mathrm{Vol}(B_{R_1}(x)) \ge \frac{1}{2^{\theta}} \left( \frac{R_2}{R_1} \right)^{\theta} \mathrm{Vol}(B_{R_1}(x)).
\end{equation*}
Thus
\begin{equation*}
\frac{R_2^{2s-1} f(R_2)}{\mathrm{Vol}(B_{R_2}(x))} = \frac{R_2^{\theta}}{\mathrm{Vol}(B_{R_2}(x))} \leq 2^{\theta} \frac{R_1^{\theta}}{\mathrm{Vol}(B_{R_1}(x))} = 2^{\theta} \frac{R_1^{2s-1} f(R_1)}{\mathrm{Vol}(B_{R_1}(x))},
\end{equation*}
that is \eqref{fcondition} with $\gamma=2^{\theta}$.
\end{exm}

\begin{exm}
Let $M$ be a  complete, noncompact Riemannian manifold with $\mathrm{Ric} \geq 0$ satisfying the noncollapsing property Assumption \ref{A1}. Suppose there exist constants \( \gamma_1, \gamma_2 > 0 \), a continuous \( f : [1, +\infty) \to (0, +\infty) \) such that $R^{2s-1}f(R)$ is nondecreasing and satisfying \eqref{betacondition} in Assumption \ref{A2}, and a function \( h : M \times [R_0, \infty) \to (0, \infty) \), with $R_0\geq1$, such that for every \( x \in M \) the function \( h(x, R) \) is nondecreasing in $R$ and
\begin{equation*}
\gamma_1 h(x, R) \leq \frac{\mathrm{Vol}(B_R(x))}{R^{2s-1} f(R)} \leq \gamma_2 h(x, R) \quad \text{for every } R \geq R_0
\end{equation*}
Then Property \eqref{fcondition} in \ref{A2} holds on $M$. Indeed for every $x \in M$ and every $R > r > R_0$ we have
\begin{equation*}
\frac{R^{2s-1} f(R)}{\mathrm{Vol}(B_R(x))} \leq \frac{1}{\gamma_1 h(x, R)} \le \frac{1}{\gamma_1 h(x, r)} \leq \frac{\gamma_2}{\gamma_1} \frac{r^{2s-1} f(r)}{\mathrm{Vol}(B_r(x))}.
\end{equation*}
\end{exm}

\begin{exm}\label{exm:last}
Let \(X\) be a complete, noncompact Riemannian manifold with
\(\operatorname{Ric}_X\geq0\) satisfying Assumptions \ref{A1} and \ref{A2}, and let
\(N\) be a compact Riemannian manifold with \(\operatorname{Ric}_N\geq0\).
Then the product manifold \(M:=X\times N\), endowed with the product metric,
is complete and noncompact, satisfies \(\operatorname{Ric}_M\geq0\), and
satisfies Assumptions \ref{A1} and \ref{A2} with the same function \(f\) as \(X\).
Indeed, set
\[
\rho_0:=\operatorname{diam}(N),
\qquad
R_*:=\max\{1,2\rho_0\}.
\]
Since
\[
d_M\bigl((x,p),(y,q)\bigr)
=
\bigl(d_X(x,y)^2+d_N(p,q)^2\bigr)^{1/2},
\]
for every \(R\geq R_*\) one has
\[
B^X_{R/2}(x)\times N
\subset B^M_R((x,p))
\subset B^X_R(x)\times N
\]
for every $(x,p)\in X\times N$.
The volume doubling property on \(X\) therefore yields constants \(c,C>0\)
such that
\[
c\,\operatorname{Vol}_X(B^X_R(x))
\leq
\operatorname{Vol}_M(B^M_R((x,p)))
\leq
C\,\operatorname{Vol}_X(B^X_R(x))
\]
for every \((x,p)\in M\) and every \(R\geq R_*\). Hence, if
\[
\alpha_X:=
\inf_{x\in X}\operatorname{Vol}_X(B^X_1(x))>0,
\]
then Bishop--Gromov comparison on \(M\) gives
\[
\inf_{(x,p)\in M}
\operatorname{Vol}_M(B^M_1((x,p)))
\geq
R_*^{-n}c\alpha_X>0,
\]
with $n=dim M$, so \(M\) satisfies \ref{A1}. Finally, let \(R_0\) and \(\gamma\) be the constants in \ref{A2} for \(X\), and
set
\[
\widehat R_0:=\max\{R_0,R_*\}.
\]
For every \(R_2\geq R_1\geq\widehat R_0\), the preceding volume comparison
gives
\[
\frac{R_2^{2s-1}f(R_2)}
{\operatorname{Vol}_M(B^M_{R_2}((x,p)))}
\leq
\frac{\gamma C}{c}
\frac{R_1^{2s-1}f(R_1)}
{\operatorname{Vol}_M(B^M_{R_1}((x,p)))}
\]
for every $(x,p)\in M$. The remaining requirements on \(f\) are inherited unchanged from \(X\);
therefore \ref{A2} holds on \(M\) with the same function \(f\).
\end{exm}

\section{Estimates on the Green Function and the Potential of Test Functions}
\label{S7}

Throughout this section, let \(M\) be a complete, noncompact,
\(s\)-nonparabolic \(n\)-dimensional Riemannian manifold, $n>2s$, with
\(\operatorname{Ric}\geq0\). We fix a continuous function
\(f:[1,+\infty)\to(0,+\infty)\), a constant \(\gamma>0\), a number
\[
\beta:=\int_1^{+\infty}\frac{dt}{f(t)}<+\infty,
\]
and radii \(R_0(x)\geq1\), \(x\in M\), for which Condition \ref{B} holds. Such
data exist by Proposition~\ref{general}. No regularity of the map
\(x\mapsto R_0(x)\) will be used. When Assumption \ref{A2} is imposed, the
notation \(f,\gamma,\beta,R_0\) refers to the corresponding data, with
\(R_0\) independent of the pole.

For \(x_0\in M\) and \(R\geq R_0(x_0)\), set
\begin{equation*}
\mathcal A_{x_0}(R)
:=
\frac{1}{\operatorname{Vol}(B_R(x_0))}
\left(
\frac{R^n}{n-2s}
+
\beta\gamma f(R)R^{n-1}
\right).
\end{equation*}

\begin{lem}\label{thm:lem1}
Let \(G_M^s\) be the minimal positive fractional Green function on \(M\). For every \(x_0\in M\), the following properties hold.
\begin{enumerate}
\item[(i)] For every \(x\in M\setminus\{x_0\}\),
\begin{equation}\label{eq:lem81-lower}
G_M^s(x,x_0)
\geq
\frac{C_1}{\omega_n(n-2s)}
\frac{1}{d(x,x_0)^{n-2s}}.
\end{equation}

\item[(ii)] For every \(x\in M\) such that \(d(x,x_0)\geq R_0(x_0)\),
\begin{equation*}
G_M^s(x,x_0)
\leq
C_2\gamma
\frac{d(x,x_0)^{2s-1}f(d(x,x_0))}
{\operatorname{Vol}(B_{d(x,x_0)}(x_0))}
\int_{d(x,x_0)}^{+\infty}\frac{dt}{f(t)}.
\end{equation*}

\item[(iii)] For every \(x\in M\) such that \(d(x,x_0)\geq1\),
\begin{equation}\label{eq:lem81-away}
G_M^s(x,x_0)
\leq C_2\Bigg(
\frac{R_0(x_0)^{2s}}
{2s\,\operatorname{Vol}(B_1(x_0))}+
\frac{\beta\gamma R_0(x_0)^{2s-1}f(R_0(x_0))}
{\operatorname{Vol}(B_{R_0(x_0)}(x_0))}
\Bigg).
\end{equation}

\item[(iv)] For every \(x\in M\setminus\{x_0\}\) and every
\(R\geq\max\{d(x,x_0),R_0(x_0)\}\),
\begin{equation}\label{eq:lem81-local-upper}
G_M^s(x,x_0)
\leq
C_2\mathcal A_{x_0}(R)
\frac{1}{d(x,x_0)^{n-2s}}.
\end{equation}

\item[(v)] For every \(0<\rho\leq R_0(x_0)\),
\begin{equation}\label{eq:lem81-local-integral}
\int_{B_\rho(x_0)}G_M^s(x,x_0)\,d\mu(x)
\leq
C_2\omega_n\frac{n}{2s}
\mathcal A_{x_0}(R_0(x_0))\rho^{2s}.
\end{equation}

\item[(vi)] For every \(R\geq R_0(x_0)\),
\begin{equation}\label{eq:lem81-ball-integral}
\int_{B_R(x_0)}G_M^s(x,x_0)\,d\mu(x)
\leq
C_2\max\{\gamma,1\}h(R),
\end{equation}
where
\[
h(R):=
R^{2s-1}f(R)\int_R^{+\infty}\frac{dt}{f(t)}
+\frac{R^{2s}}{2s}.
\]
\end{enumerate}
Here \(C_1\) and \(C_2\) are the constants in Theorem~\ref{thm:bounGreen's}.
\end{lem}

The proof of Lemma~\ref{thm:lem1} is given in Appendix~\ref{proofs}.

We next compare the potential generated by a nonnegative bounded function
concentrated near a pole with the Green function having the same pole.

\begin{prop}\label{thm:3.4}
Fix \(x_0\in M\) and \(\sigma\in(0,1]\). Let
\(0\leq\psi\in L^\infty(M)\) vanish almost everywhere on
\(M\setminus B_\sigma(x_0)\). Then there exist constants
\(\gamma_1(x_0),\gamma_2(x_0)>0\), independent of \(\sigma\) and \(\psi\),
such that, for every \(x\in M\setminus\{x_0\}\),
\begin{equation}\label{resultthm3.4}
\begin{aligned}
\gamma_1(x_0)
\min\{1,d(x,x_0)^{n-2s}\}
&\|\psi\|_{L^1(M)}G_M^s(x,x_0)
\\
&\leq
(-\Delta)^{-s}\psi(x)
\leq
\gamma_2(x_0)
\|\psi\|_{L^\infty(M)}
\operatorname{Vol}(B_\sigma(x_0))G_M^s(x,x_0).
\end{aligned}
\end{equation}
\end{prop}

\begin{proof}
We first prove the upper bound. Assume that \(d(x,x_0)\leq2\sigma\). Since
\(\psi=0\) almost everywhere on \(M\setminus B_\sigma(x_0)\) and
\(B_\sigma(x_0)\subset B_{3\sigma}(x)\), Theorem~\ref{thm:bounGreen's},
Tonelli's theorem, the coarea formula, and integration by parts give
\begin{align}
(-\Delta)^{-s}\psi(x)
&\leq
C_2\|\psi\|_{L^\infty(M)}
\int_{B_{3\sigma}(x)}
\int_{d(x,y)}^{+\infty}
\frac{t^{2s-1}}{\operatorname{Vol}(B_t(x))}
\,dt\,d\mu(y)
\notag\\
&\leq
C_2\|\psi\|_{L^\infty(M)}
\left[
\operatorname{Vol}(B_{3\sigma}(x))
\int_{3\sigma}^{+\infty}
\frac{t^{2s-1}}{\operatorname{Vol}(B_t(x))}
\,dt
+
\frac{(3\sigma)^{2s}}{2s}
\right].
\label{prop1a18}
\end{align}
The inclusions
\[
B_{3\sigma}(x)
\subset B_{5\sigma}(x_0)
\subset B_{8\sigma}(x_0)
\]
and the volume doubling property implies
\begin{equation*}
\operatorname{Vol}(B_{3\sigma}(x))
\leq
C_d^3\operatorname{Vol}(B_\sigma(x_0)).
\end{equation*}
Since \(0<\sigma\leq1\), Bishop--Gromov comparison yields
\[
\operatorname{Vol}(B_\sigma(x_0))
\geq
\sigma^n\operatorname{Vol}(B_1(x_0)).
\]
Also, by \eqref{eq:lem81-lower} and \(d(x,x_0)\leq2\sigma\),
\begin{equation}\label{prop1a1nc}
\sigma^{2s}
\leq
\frac{2^{n-2s}(n-2s)\omega_n}
{C_1\operatorname{Vol}(B_1(x_0))}
\operatorname{Vol}(B_\sigma(x_0))G_M^s(x,x_0).
\end{equation}
Finally, because \(d(x,x_0)\leq2\sigma<3\sigma\), the lower estimate in
Theorem~\ref{thm:bounGreen's} gives
\[
\int_{3\sigma}^{+\infty}
\frac{t^{2s-1}}{\operatorname{Vol}(B_t(x))}\,dt
\leq
\frac{1}{C_1}G_M^s(x,x_0).
\]
Combining this estimate with \eqref{prop1a18}--\eqref{prop1a1nc}, we obtain
\begin{equation}\label{eq:prop82-upper-near}
(-\Delta)^{-s}\psi(x)
\leq
C_2\Bigg(
\frac{C_d^3}{C_1}
+
\frac{3^{2s}2^{n-1-2s}(n-2s)\omega_n}
{sC_1\operatorname{Vol}(B_1(x_0))}
\Bigg)
\|\psi\|_{L^\infty(M)}
\operatorname{Vol}(B_\sigma(x_0))G_M^s(x,x_0).
\end{equation}

Assume now that \(d(x,x_0)>2\sigma\). For almost every
\(y\in B_\sigma(x_0)\),
\[
d(x,y)\geq d(x,x_0)-\sigma.
\]
Therefore, after the change of variables \(r=t-\sigma\),
\begin{align*}
(-\Delta)^{-s}\psi(x)
&\leq
C_2\|\psi\|_{L^\infty(M)}
\operatorname{Vol}(B_\sigma(x_0))
\int_{d(x,x_0)-\sigma}^{+\infty}
\frac{r^{2s-1}}{\operatorname{Vol}(B_r(x))}\,dr
\\
&=
C_2\|\psi\|_{L^\infty(M)}
\operatorname{Vol}(B_\sigma(x_0))
\int_{d(x,x_0)}^{+\infty}
\frac{(t-\sigma)^{2s-1}}
{\operatorname{Vol}(B_{t-\sigma}(x))}\,dt.
\end{align*}
For \(t\geq d(x,x_0)>2\sigma\), one has \(t-\sigma\geq t/2\). Hence
\[
\operatorname{Vol}(B_{t-\sigma}(x))
\geq
C_d^{-1}\operatorname{Vol}(B_t(x))
\]
and
\[
(t-\sigma)^{2s-1}
\leq
2t^{2s-1}
\qquad\text{for every }s\in(0,1].
\]
Using the lower estimate in Theorem~\ref{thm:bounGreen's}, we conclude that
\begin{equation}\label{prop311}
(-\Delta)^{-s}\psi(x)
\leq
\frac{2C_2C_d}{C_1}
\|\psi\|_{L^\infty(M)}
\operatorname{Vol}(B_\sigma(x_0))G_M^s(x,x_0).
\end{equation}
The upper bound in \eqref{resultthm3.4} follows from
\eqref{eq:prop82-upper-near} and \eqref{prop311}.

We now prove the lower bound. Assume first that \(d(x,x_0)\geq1\). For
every \(y\in B_\sigma(x_0)\),
\[
d(x,y)\leq d(x,x_0)+\sigma.
\]
The lower estimate in Theorem~\ref{thm:bounGreen's} therefore gives
\begin{align}
(-\Delta)^{-s}\psi(x)
&\geq
C_1\int_{B_\sigma(x_0)}\psi(y)
\int_{d(x,x_0)+\sigma}^{+\infty}
\frac{t^{2s-1}}{\operatorname{Vol}(B_t(x))}
\,dt\,d\mu(y)
\notag\\
&=
C_1\|\psi\|_{L^1(M)}
\int_{d(x,x_0)}^{+\infty}
\frac{(\tau+\sigma)^{2s-1}}
{\operatorname{Vol}(B_{\tau+\sigma}(x))}\,d\tau.
\label{eq:prop82-lower-far}
\end{align}
Since \(\tau\geq d(x,x_0)\geq1\geq\sigma\), one has
\(\tau+\sigma\leq2\tau\). Thus,
\[
\operatorname{Vol}(B_{\tau+\sigma}(x))
\leq
C_d\operatorname{Vol}(B_\tau(x))
\]
and
\[
(\tau+\sigma)^{2s-1}
\geq
\frac12\tau^{2s-1}
\qquad\text{for every }s\in(0,1].
\]
The upper estimate in Theorem~\ref{thm:bounGreen's} and
\eqref{eq:prop82-lower-far} yield
\begin{equation}\label{eq:prop82-lower-far-final}
(-\Delta)^{-s}\psi(x)
\geq
\frac{C_1}{2C_2C_d}
\|\psi\|_{L^1(M)}G_M^s(x,x_0).
\end{equation}

Finally, assume that \(0<d(x,x_0)<1\). For every
\(y\in B_\sigma(x_0)\), one has \(d(x,y)\leq1+\sigma\leq2\). Hence
\eqref{eq:lem81-lower} gives
\begin{equation}\label{eq:prop82-lower-near-constant}
(-\Delta)^{-s}\psi(x)
\geq
\frac{C_1}{2^{n-2s}(n-2s)\omega_n}
\|\psi\|_{L^1(M)}.
\end{equation}
On the other hand, \eqref{eq:lem81-local-upper}, with
\(R=R_0(x_0)\), gives
\begin{equation}\label{eq:prop82-near-green-upper}
d(x,x_0)^{n-2s}G_M^s(x,x_0)
\leq
C_2\mathcal A_{x_0}(R_0(x_0)).
\end{equation}
Combining \eqref{eq:prop82-lower-far-final}, \eqref{eq:prop82-lower-near-constant} and
\eqref{eq:prop82-near-green-upper} proves the lower bound in
\eqref{resultthm3.4}.
\end{proof}

\begin{rmk}
The constants in Proposition~\ref{thm:3.4} may be chosen as
\begin{equation*}
\gamma_1(x_0)
:=
\min\left\{
\frac{C_1}{2C_2C_d},
\frac{C_1}
{2^{n-2s}(n-2s)\omega_nC_2
\mathcal A_{x_0}(R_0(x_0))}
\right\}
\end{equation*}
and
\begin{equation*}
\gamma_2(x_0)
:=
\max\left\{
\frac{2C_2C_d}{C_1},
\frac{C_2C_d^3}{C_1}
+
\frac{3^{2s}2^{n-1-2s}(n-2s)\omega_nC_2}
{sC_1\operatorname{Vol}(B_1(x_0))}
\right\}.
\end{equation*}
\end{rmk}

\begin{rmk}\label{rmkprop3.4}
Fix \(\sigma_0>1\). The upper estimate in \eqref{resultthm3.4} remains
valid for every \(\sigma\in(0,\sigma_0]\), with a constant independent of
\(\sigma\) and \(\psi\). Indeed, Bishop--Gromov comparison gives
\[
\operatorname{Vol}(B_\sigma(x_0))
\geq
\left(\frac{\sigma}{\sigma_0}\right)^n
\operatorname{Vol}(B_{\sigma_0}(x_0))
\geq
\frac{\sigma^n}{\sigma_0^n}
\operatorname{Vol}(B_1(x_0)),
\]
and the proof of Proposition~\ref{thm:3.4} applies with
\begin{equation*}
\gamma_{2,\sigma_0}(x_0)
:=
\max\left\{
\frac{2C_2C_d}{C_1},
\frac{C_2C_d^3}{C_1}
+
\frac{3^{2s}2^{n-1-2s}(n-2s)\omega_n\sigma_0^nC_2}
{sC_1\operatorname{Vol}(B_1(x_0))}
\right\}.
\end{equation*}
\end{rmk}

\begin{rmk}\label{rmk:gamcont}
Assume in addition that \ref{A1} and \ref{A2} hold, and set
\[
\alpha:=\inf_{x\in M}\operatorname{Vol}(B_1(x))>0.
\]
Then, for every \(x_0\in M\) and every \(R\geq R_0\),
\begin{equation*}
\mathcal A_{x_0}(R)
\leq
\frac{1}{\alpha}
\left(
\frac{R^n}{n-2s}
+
\beta\gamma f(R)R^{n-1}
\right).
\end{equation*}
Consequently, the constants in \eqref{eq:lem81-away},
\eqref{eq:lem81-local-upper}, and \eqref{eq:lem81-local-integral} can be
chosen uniformly with respect to the pole. The constants
\(\gamma_1(x_0)\), \(\gamma_2(x_0)\), and
\(\gamma_{2,\sigma_0}(x_0)\) in Proposition~\ref{thm:3.4} and
Remark~\ref{rmkprop3.4} can likewise be replaced by positive constants
independent of \(x_0\). 
\end{rmk}

The next proposition identifies the Green-weighted norm with an integral
weighted by the potential of a compactly supported bounded function.

\begin{prop}\label{thm:lem3.6}
Fix \(x_0\in M\) and define
\begin{equation}\label{eq:psi-x0}
\psi_{x_0}
:=
\frac{\mathbf 1_{B_{1/2}(x_0)}}
{\operatorname{Vol}(B_{1/2}(x_0))}.
\end{equation}
Then there exist constants \(K_1(x_0),K_2(x_0)>0\) such that, for every
measurable function \(\phi\) on \(M\),
\begin{equation}\label{lem363r}
K_1(x_0)
\int_M|\phi|(-\Delta)^{-s}\psi_{x_0}\,d\mu
\leq
\|\phi\|_{L^1_{G_M^s,x_0}(M)}
\leq
K_2(x_0)
\int_M|\phi|(-\Delta)^{-s}\psi_{x_0}\,d\mu
\end{equation}
In particular,
\[
\phi\in L^1_{G_M^s,x_0}(M)
\qquad\textrm{ if and only if }\qquad
\int_M|\phi|(-\Delta)^{-s}\psi_{x_0}\,d\mu<+\infty.
\]
If \ref{A1} and \ref{A2} hold, the constants \(K_1\) and \(K_2\) can be chosen
independently of \(x_0\in M\).
\end{prop}

\begin{proof}
Set
\[
\Psi_{x_0}:=(-\Delta)^{-s}\psi_{x_0}.
\]
Then
\[
\|\psi_{x_0}\|_{L^1(M)}=1=
\|\psi_{x_0}\|_{L^\infty(M)}
\operatorname{Vol}(B_{1/2}(x_0)).
\]
Therefore, Proposition~\ref{thm:3.4}, applied with \(\sigma=1/2\), gives
for every \(x\in M\setminus B_1(x_0)\),
\begin{equation}\label{lem363f1}
\gamma_1(x_0)G_M^s(x,x_0)
\leq
\Psi_{x_0}(x)
\leq
\gamma_2(x_0)G_M^s(x,x_0).
\end{equation}
For every \(x\in B_1(x_0)\), Proposition~\ref{thm:3.4} and
\eqref{eq:lem81-lower} yield
\begin{equation}\label{lem36l1}
\Psi_{x_0}(x)
\geq
\gamma_1(x_0)d(x,x_0)^{n-2s}G_M^s(x,x_0)
\geq
\frac{\gamma_1(x_0)C_1}{\omega_n(n-2s)}.
\end{equation}
On the other hand, Lemma~\ref{lemlastlem}(3) gives
\[
\Psi_{x_0}\in L^\infty(M).
\]
Hence there exist constants \(0<c(x_0)\leq C(x_0)<+\infty\) such that,
for every \(x\in M\),
\begin{align*}
c(x_0)\Bigl(
\mathbf 1_{B_1(x_0)}(x)
+&G_M^s(x,x_0)\mathbf 1_{M\setminus B_1(x_0)}(x)
\Bigr)\\
&\leq
\Psi_{x_0}(x)
\leq
C(x_0)\Bigl(
\mathbf 1_{B_1(x_0)}(x)
+G_M^s(x,x_0)\mathbf 1_{M\setminus B_1(x_0)}(x)
\Bigr).
\end{align*}
Multiplying by \(|\phi|\) and integrating over \(M\) proves
\eqref{lem363r}.

Assume now that \ref{A1} and \ref{A2} hold. By Remark~\ref{rmk:gamcont}, the
constants in \eqref{lem363f1} and \eqref{lem36l1} can be chosen
independently of \(x_0\). Furthermore, Bishop--Gromov comparison gives
\[
\operatorname{Vol}(B_{1/2}(x_0))
\geq
2^{-n}\operatorname{Vol}(B_1(x_0))
\geq
2^{-n}\alpha.
\]
Set $\widehat R:=\max\left\{\frac32,R_0\right\}$. If \(x\in B_1(x_0)\), then
\(B_{1/2}(x_0)\subset B_{3/2}(x)\subset B_{\widehat R}(x)\). By \eqref{eq:lem81-ball-integral}
\begin{align*}
\Psi_{x_0}(x)
&\leq
\frac{1}{\operatorname{Vol}(B_{1/2}(x_0))}
\int_{B_{1/2}(x_0)}G_M^s(x,y)\,d\mu(y)\\
&\leq
\frac{2^n}{\alpha}
\int_{B_{\widehat R}(x)}G_M^s(y,x)\,d\mu(y)
\leq
\frac{2^nC_2}{\alpha}\max\{\gamma,1\}h(\widehat R).
\end{align*}
Thus, the constants \(c(x_0)\) and \(C(x_0)\), and consequently
\(K_1(x_0)\) and \(K_2(x_0)\), can be chosen independently of the pole.
\end{proof}

s\section{From Mild Solutions to Weak Dual Solutions via Approximation}
\label{sec:frommstowds}

We first study existence, uniqueness, and basic qualitative properties of mild solutions to
Problem~\eqref{maineq}. These solutions are obtained by nonlinear semigroup theory in
$L^1(M)$ and can be constructed as limits of implicit Euler approximations. The arguments in
Proposition~\ref{prop:4.1} are independent of Assumption \ref{A1} and of $s$-nonparabolicity.
The latter assumption enters only when the Green potential is used to identify mild solutions
with weak dual solutions.

Throughout this section, by a mild solution to \eqref{maineq} we mean the nonlinear semigroup
solution in $L^1(M)$. By the Crandall--Liggett Theorem, see, e.g., \cite{Vazquez2007}, this
solution is obtained as the limit in $C([0,T];L^1(M))$, for every $T>0$, of the implicit Euler
approximations introduced below.

\begin{prop}\label{prop:4.1}
Let $M$ be a complete, noncompact $n$-dimensional Riemannian manifold with
$\operatorname{Ric}\geq0$. For any nonnegative initial data
$u_0,v_0\in L^1(M)\cap L^\infty(M)$, there exist unique nonnegative mild solutions
$u,v\in C([0,+\infty);L^1(M))$ to Problem~\eqref{maineq}. In addition, the following
properties hold:
\begin{enumerate}
\item[(i)] \emph{Time monotonicity:}
\begin{equation}\label{eq:prop4.1mon}
 t\longmapsto t^{\frac1{m-1}}u(t,x)
 \quad\text{is nondecreasing on }(0,+\infty)
 \text{ for a.e. }x\in M.
\end{equation}

\item[(ii)] \emph{$L^p$-nonexpansivity:}
\begin{equation}\label{nonexpms}
 \|u(t)\|_{L^p(M)}\leq \|u_0\|_{L^p(M)},
 \qquad t\geq0,\quad 1\leq p\leq\infty.
\end{equation}

\item[(iii)] \emph{Order preservation:}
\begin{equation}\label{orderingpos}
 \int_M (u(t,x)-v(t,x))_+\,d\mu(x)
 \leq
 \int_M (u_0(x)-v_0(x))_+\,d\mu(x),
 \qquad t\geq0.
\end{equation}

\item[(iv)] \emph{$L^1$-contraction:}
\begin{equation*}
 \|u(t)-v(t)\|_{L^1(M)}
 \leq
 \|u_0-v_0\|_{L^1(M)},
 \qquad t\geq0.
\end{equation*}
\end{enumerate}
\end{prop}

\begin{proof}
Set $A:=(-\Delta)^s$ and extend the porous-medium nonlinearity to
$\Phi(r):=|r|^{m-1}r$ on $\mathbb R$; on the nonnegative cone, $\Phi(r)=r^m$.
For $s=1$ we set $T^1_t=T_t$, the heat semigroup on $M$. The heat semigroup on $M$ is a strongly continuous
sub-Markovian contraction semigroup on $L^1(M)$. For $s\in(0,1)$, subordination through the Bernstein
function $\lambda\mapsto\lambda^s$ yields a strongly continuous sub-Markovian contraction
semigroup $(T_t^s)_{t\geq0}$ on $L^1(M)$ whose generator is $-A$; see
\cite[Chapters~3--4]{Jacob2001}. In particular, for every $s\in(0,1]$ $A$ is $m$-accretive and satisfies the
complete-accretivity property required in the theory of Crandall and Pierre. We shall use in
particular the standard consequence
\begin{equation*}
 \int_M \beta(w)Aw\,d\mu\geq0,
\end{equation*}
for every nondecreasing Lipschitz function $\beta:\mathbb R\to\mathbb R$ with
$\beta(0)=0$ and every $w\in D(A)$ for which the integral is well defined; the usual
monotone approximations allow one to use the corresponding truncated sign and power
functions. The abstract theory of \cite{crandallpierre} therefore gives existence and
uniqueness of the mild solution and the B\'enilan--Crandall time-monotonicity
\eqref{eq:prop4.1mon}; see also \cite[Theorem~2.1]{BonforteVazquez2016}.

The proofs of \emph{(ii)--(iv)} require only the semigroup and accretivity properties of $A$ established above. In particular, the proof given in \cite[Proposition~5.1]{BerchioBonforteGrilloMuratori2024} applies in the present setting.

Since the time-discretization used in the construction of mild solutions will also be needed below, we recall it here.
Fix $T>0$ and $\ell\in\mathbb N$, and set
\[
 t_k:=\frac{kT}{\ell},\qquad k=0,\ldots,\ell,
 \qquad h:=\frac{T}{\ell}.
\]
Write $u_k:=u_{k,\ell}$ for the discrete states and define the piecewise constant approximation by
\begin{equation*}
 u_\ell(t):=u_k
 \quad\text{for }t_k\leq t<t_{k+1},
 \qquad
 u_\ell(T):=u_{\ell,\ell},
\end{equation*}
where $u_0=u_{0,\ell}$ is the initial datum and the states $u_{k+1}$ are determined recursively by
\begin{equation}\label{fracellpb}
 hA(u_{k+1}^m)+u_{k+1}=u_k.
\end{equation}
By the Crandall--Liggett construction,
\[
 \sup_{t\in[0,T]}\|u_\ell(t)-u(t)\|_{L^1(M)}\longrightarrow0
 \qquad\text{as }\ell\to\infty.
\]
We shall use this approximation below in the passage from mild to weak dual solutions.
\end{proof}

\begin{rmk}
In particular, the order-preserving semigroup leaves the positive cone invariant: if
$u_0\geq0$, then $u(t)\geq0$ for every $t\geq0$.
\end{rmk}

\begin{rmk}
Proposition~\ref{prop:4.1} is purely semigroup-theoretic and does not use
$s$-nonparabolicity. From this point on we assume that $M$ is $s$-nonparabolic, so that the
Green potential $(-\Delta)^{-s}$ is available.
\end{rmk}

For $s\in(0,1)$, let $(T_t^s)_{t\geq0}$ denote the semigroup subordinated to the heat semigroup
$(T_t)_{t\geq0}$ through the Bernstein function $\lambda\mapsto\lambda^s$. When $s=1$, let $(T_t^1)_{t\geq0}=(T_t)_{t\geq0}$ denote the heat semigroup on $M$. In either case, the generator of $(T_t^s)_{t\geq0}$ is $-(-\Delta)^s$.

\begin{lem}\label{lem_leftinv}
Let $M$ be a complete, noncompact, $s$-nonparabolic $n$-dimensional Riemannian manifold, $n>2s$,
with $\operatorname{Ric}\geq0$.
\begin{enumerate}
\item[(i)] The operator $(-\Delta)^{-s}$ defined in \eqref{defform} is continuous from
$L^1(M)\cap L^\infty(M)$ into $L^\infty_{\mathrm{loc}}(M)$. Also,
for every compact set $V\Subset M$, it is continuous from
\[
 \{\psi\in L^\infty(M):\operatorname{supp}\psi\subset V\},
\]
into $L^\infty(M)$.

\item[(ii)] For every $g\in L^1(M)\cap L^\infty(M)$,
\begin{equation}\label{eq:integral_convergence}
 \lim_{r\to+\infty}\int_0^r T_t^s g(x)\,dt
 =\int_0^{+\infty}T_t^s g(x)\,dt
 =(-\Delta)^{-s}g(x)
\end{equation}
for every $x\in M$. The integral is absolutely convergent. In addition, the convergence in
\eqref{eq:integral_convergence} holds in $L^p_{\mathrm{loc}}(M)$ for every
$1\leq p<\infty$.

\item[(iii)] If $g\in D(( -\Delta)^s)$ and
\[
 g,\;(-\Delta)^s g\in L^1(M)\cap L^\infty(M),
\]
then
\begin{equation}\label{eq:left_inverse}
 (-\Delta)^{-s}\bigl(( -\Delta)^s g\bigr)=g
 \qquad\text{a.e. in }M.
\end{equation}
\end{enumerate}
\end{lem}

\begin{proof}
Part~(i) is precisely Lemma~\ref{lemlastlem}(2)--(3).

We prove part~(ii). For $s=1$, since $(T_t^1)_{t\geq0}=(T_t)_{t\geq0}$, \eqref{eq:integral_convergence} follows directly from \eqref{fracGreen'sheat}, Tonelli's theorem and the definition of $(-\Delta)^{-1}$ in \eqref{defform}.

From now on we suppose $0<s<1$. By subordination, see \cite[Section~4.3]{Jacob2001}, there exists, for
every $t>0$, a nonnegative probability density $\eta_t^s$ on $(0,+\infty)$ such that
\begin{equation*}
 T_t^s g(x)=\int_0^{+\infty}T_\tau g(x)\eta_t^s(\tau)\,d\tau.
\end{equation*}
Equivalently,
\[
 T_t^s g(x)
 =\int_M p_M^s(t,x,y)g(y)\,d\mu(y),
 \qquad
 p_M^s(t,x,y):=
 \int_0^{+\infty}p_M(\tau,x,y)\eta_t^s(\tau)\,d\tau.
\]
The stable subordinator satisfies
\begin{equation}\label{etaid}
 \int_0^{+\infty}\eta_t^s(\tau)\,dt
 =\frac{\tau^{s-1}}{\Gamma(s)},
 \qquad \tau>0.
\end{equation}
Indeed, the Laplace-transform identity
\[
 \int_0^{+\infty}e^{-\lambda\tau}\eta_t^s(\tau)\,d\tau
 =e^{-t\lambda^s}
\]
and Tonelli's theorem give
\[
 \int_0^{+\infty}e^{-\lambda\tau}
 \left(\int_0^{+\infty}\eta_t^s(\tau)\,dt\right)d\tau
 =\lambda^{-s},
\]
which is the Laplace transform of $\tau^{s-1}/\Gamma(s)$.

For $g\in L^1(M)\cap L^\infty(M)$, Lemma~\ref{lemlastlem}(2), positivity of the kernels,
and Tonelli's theorem yield, for every $x\in M$,
\begin{align*}
 \int_0^{+\infty}|T_t^s g(x)|\,dt
 &\leq
 \int_M |g(y)|
 \left(\int_0^{+\infty}p_M^s(t,x,y)\,dt\right)d\mu(y)
 \\
 &=
 \frac1{\Gamma(s)}
 \int_M |g(y)|
 \left(\int_0^{+\infty}p_M(\tau,x,y)\tau^{s-1}\,d\tau\right)d\mu(y)
 \\
 &=\int_M |g(y)|G_M^s(x,y)\,d\mu(y)<+\infty.
\end{align*}
Using \eqref{etaid} once more therefore gives
\[
 \int_0^{+\infty}T_t^s g(x)\,dt
 =\int_M G_M^s(x,y)g(y)\,d\mu(y)
 =(-\Delta)^{-s}g(x),
\]
which proves the pointwise assertion in \eqref{eq:integral_convergence}. If $V\Subset M$ is
compact, then
\[
 \left|\int_r^{+\infty}T_t^s g(x)\,dt\right|
 \leq(-\Delta)^{-s}|g|(x),
 \qquad x\in V,
\]
and the right-hand side belongs to $L^\infty(V)$ by part~(i). Dominated convergence on $V$
proves convergence in $L^p(V)$ for every $1\leq p<\infty$.

We finally prove part~(iii). Set again $A:=(-\Delta)^s$. Since
$g,Ag\in L^1(M)\cap L^\infty(M)$, both functions belong to $L^2(M)$. For every $r>0$,
the semigroup identity gives
\begin{equation*}
 \int_0^r T_t^s Ag\,dt=g-T_r^s g
 \qquad\text{in }L^2(M).
\end{equation*}
By the spectral theorem, $T_r^s g$ converges in $L^2(M)$, as $r\to+\infty$, to the
orthogonal projection of $g$ onto $\ker A$. Since
$\ker A=\ker(-\Delta)$, every element of $\ker A\subset L^2(M)$ has vanishing Dirichlet
energy and hence is constant. Furthermore, a complete noncompact manifold with
$\operatorname{Ric}\geq0$ has infinite volume, so this constant must be zero. Consequently,
\begin{equation}\label{leminvp22}
 \int_0^r T_t^s Ag\,dt\longrightarrow g
 \qquad\text{in }L^2(M)\text{ as }r\to+\infty.
\end{equation}
On the other hand, part~(ii), applied to $Ag$, gives pointwise convergence of the same
integrals to $(-\Delta)^{-s}Ag$. Taking an a.e.-convergent subsequence in
\eqref{leminvp22} proves \eqref{eq:left_inverse}.
\end{proof}

We now show that the mild solutions constructed in Proposition~\ref{prop:4.1} are weak dual
solutions in the sense of Definition~\ref{defn:wds}.

\begin{prop}\label{prop:existence}
Let $M$ be a complete, noncompact, $s$-nonparabolic $n$-dimensional Riemannian manifold
with $\operatorname{Ric}\geq0$. For every nonnegative initial datum
$u_0\in L^1(M)\cap L^\infty(M)$, the mild solution $u$ to \eqref{maineq} constructed in
Proposition~\ref{prop:4.1} is a weak dual solution in the sense of
Definition~\ref{defn:wds}.
\end{prop}

\begin{proof}
Fix $T>0$ and consider the implicit Euler scheme \eqref{fracellpb} in
Proposition~\ref{prop:4.1}. By construction, $u_{k+1}^m\in D(A)$, where
$A:=(-\Delta)^s$, and
\[
A(u_{k+1}^m)
=
\frac{u_k-u_{k+1}}{h}
\in L^1(M)\cap L^\infty(M).
\]
Since also $u_{k+1}^m\in L^1(M)\cap L^\infty(M)$,
Lemma~\ref{lem_leftinv}(iii) yields
\begin{equation}\label{eq:dual-discrete}
(-\Delta)^{-s}u_{k+1}-(-\Delta)^{-s}u_k
=
-hu_{k+1}^m
\qquad\text{a.e. in }M.
\end{equation}

Let $\psi\in C_c^1((0,T);L_c^\infty(M))$ and choose $K\Subset M$
containing the spatial supports of $\psi(t)$ and $\partial_t\psi(t)$ for
all $t\in[0,T]$. Multiplying \eqref{eq:dual-discrete} by $\psi_k:=\psi(t_k)$ and summing over $k=0,\ldots,\ell-1$, the remainder of the discrete summation-by-parts argument and the passage to the limit $\ell\to\infty$ follow along the same lines as in \cite[Proof of Proposition~5.3]{BerchioBonforteGrilloMuratori2024}. We only record the key ingredients here.

The passage to the limit relies on the
Crandall--Liggett convergence
\[
\sup_{t\in[0,T]}
\|u_\ell(t)-u(t)\|_{L^1(M)}
\longrightarrow0,
\]
the $L^\infty$- and $L^1$-nonexpansivity \eqref{nonexpms} of
Proposition~\ref{prop:4.1}, and the following estimate: If
$\varphi\in L^\infty(M)$ satisfies $\operatorname{supp}\varphi\subset K$,
then, by symmetry of the Green kernel, Tonelli's theorem, and
Lemma~\ref{lemlastlem}(3), for every
$w\in L^1(M)\cap L^\infty(M)$,
\begin{equation*}
\left|
\int_M \varphi\,(-\Delta)^{-s}w\,d\mu
\right|
=
\left|
\int_M w\,(-\Delta)^{-s}\varphi\,d\mu
\right|
\le
C_K\|\varphi\|_{L^\infty(M)}\|w\|_{L^1(M)}.
\end{equation*}
These ingredients allow one to pass to the limit $\ell\rightarrow\infty$ in the discrete identity
as in \cite[Proof of Proposition~5.3]{BerchioBonforteGrilloMuratori2024}. We then obtain
\[
\int_0^T\int_M
\partial_t\psi\,(-\Delta)^{-s}u\,d\mu\,dt
=
\int_0^T\int_M u^m\psi\,d\mu\,dt.
\]

Finally, $u\in C([0,T];L^1(M))$ by Proposition~\ref{prop:4.1}, and the
continuous embedding
$L^1(M)\hookrightarrow L^1_{G_M^s,x_0}(M)$ gives the required continuity
in the weighted space. The $L^\infty$ bound also gives
$u^m\in C([0,T];L^1(M))$, while $u(0)=u_0$ in $L^1(M)$.
Hence $u$ satisfies all the requirements of
Definition~\ref{defn:wds}.
\end{proof}

\begin{cor}
Let $M$ be a complete, noncompact, $s$-nonparabolic $n$-dimensional
Riemannian manifold with $\operatorname{Ric}\geq0$. Let
$u_0\in L^1(M)$ be nonnegative, and let $u$ be a WDS to
\eqref{maineq} obtained by monotone approximation from WDSs associated,
through Proposition~\ref{prop:existence}, with nonnegative data in
$L^1(M)\cap L^\infty(M)$. Then
\[
\int_M u(t,x)\,d\mu(x)
=
\int_M u_0(x)\,d\mu(x),
\qquad t\geq0.
\]
\end{cor}

\begin{proof}
We first consider $u_0\in L^1(M)\cap L^\infty(M)$ and the mild solution
constructed in Proposition~\ref{prop:4.1}.
Let \((T_t^s)_{t\geq0}\) be the semigroup generated by
\(-A=-(-\Delta)^s\). If \(0<s<1\), then \((T_t^s)_{t\geq0}\) is the
semigroup subordinated to the heat semigroup, whereas for \(s=1\) one
has \(T_t^1=T_t\). Since \(\operatorname{Ric}\geq0\), \(M\) is
stochastically complete, and hence the heat semigroup is conservative.
It follows that \((T_t^s)_{t\geq0}\) is conservative for every
\(s\in(0,1]\): by subordination if \(0<s<1\), and directly from
\(T_t^1=T_t\) if \(s=1\). Therefore, for every
$w\in D(A)$,
\[
\int_M Aw\,d\mu
=
\lim_{\tau\downarrow0}
\frac{1}{\tau}
\left(
\int_M w\,d\mu-\int_M T_\tau^s w\,d\mu
\right)
=0.
\]
Consider now the implicit Euler scheme \eqref{fracellpb}. Since
$u_{k+1}^m\in D(A)$, integration over $M$ gives
\[
\int_M u_{k+1}\,d\mu
=
\int_M u_k\,d\mu
-
h\int_M A(u_{k+1}^m)\,d\mu
=
\int_M u_k\,d\mu.
\]
Thus every discrete approximation preserves the mass. Passing to the
Crandall--Liggett limit in $C([0,T];L^1(M))$ yields
\[
\int_M u(t)\,d\mu=\int_M u_0\,d\mu,
\qquad t\geq0.
\]
For a general nonnegative $u_0\in L^1(M)$, the claim follows by monotone approximation and the monotone convergence theorem, passing to the limit in the preceding identity.
\end{proof}

\section{Properties of Weak Dual Solutions}\label{sec: propswds}

In this section we study properties of weak dual solutions (WDS) to \eqref{maineq} obtained by monotone approximation. Constants may vary from line to line.

\begin{prop}\label{thm:estonwds}
Let $M$ be a complete, noncompact, $s$-nonparabolic $n$-dimensional Riemannian manifold, $n>2s$, with $\operatorname{Ric}(M)\geq 0$. For any WDS $u$ of \eqref{maineq} with nonnegative initial datum $u_0\in L^1(M)\cap L^\infty(M)$, we have
\begin{equation}\label{5.1N}
\int_M u(t,x)G_M^s(x,x_0)\,d\mu(x)
\leq
\int_M u_0(x)G_M^s(x,x_0)\,d\mu(x),
\qquad x_0\in M,\ t\geq 0,
\end{equation}
and, for every $0<t_0\leq t_1<t$ and almost every $x_0\in M$,
\begin{equation}\label{wdsmonf}
\begin{split}
\left(\frac{t_0}{t_1}\right)^{\frac{m}{m-1}}
(t_1-t_0)u^m(t_0,x_0)
&\leq
\int_M\bigl[u(t_0,x)-u(t_1,x)\bigr]G_M^s(x_0,x)\,d\mu(x)\\
&\leq
(m-1)\frac{t^{\frac{m}{m-1}}}{t_0^{\frac{1}{m-1}}}u^m(t,x_0).
\end{split}
\end{equation}
\end{prop}

\begin{proof}
The argument follows along the same lines as \cite[Proposition~5.4]{BerchioBonforteGrilloMuratori2024}
and the references therein; see also \cite[Proposition~5.1]{FPMD}. We nevertheless include the proof to highlight the adjustments needed in our setting, in particular the pole dependence of the estimates.

Fix $T>0$, $0<t_0\le t_1\le T$, and
$0\le \eta\in L^\infty_c(M)$. Approximating
$\mathbf 1_{[t_0,t_1]}$ by smooth functions in the weak dual
formulation in Definition \ref{defn:wds}, and using Lemma~\ref{rmk:contloc}, we obtain
\begin{equation}\label{wdsagtestmon1}
\int_M \eta(x)\Bigl[(-\Delta)^{-s}u(t_0,x)-(-\Delta)^{-s}u(t_1,x)\Bigr] \,d\mu(x)
=
\int_{t_0}^{t_1}\int_M u^m(\tau,x)\eta(x)\,d\mu(x)\,d\tau
\geq 0.
\end{equation}
In particular, by Fubini--Tonelli,
\begin{equation}\label{5.3N}
\int_M (-\Delta)^{-s}\eta(x)u(t_1,x)\,d\mu(x)
\leq
\int_M (-\Delta)^{-s}\eta(x)u(t_0,x)\,d\mu(x).
\end{equation}

Fix $x_0\in M$ and, for $k\in\mathbb{N}$, set
\begin{equation*}
\eta_k(x):=\frac{\mathbf 1_{B_{1/k}(x_0)}(x)}{\operatorname{Vol}(B_{1/k}(x_0))}.
\end{equation*}
We claim that, for every $\tau\geq 0$,
\begin{equation}\label{5.6N}
\lim_{k\to\infty}
\int_M u(\tau,x)(-\Delta)^{-s}\eta_k(x)\,d\mu(x)
=
\int_M u(\tau,x)G_M^s(x,x_0)\,d\mu(x).
\end{equation}

Before proving this claim, let us show that it implies \eqref{5.1N}. Taking $t_1=t$ in \eqref{5.3N} and then letting $k\to\infty$ in view of \eqref{5.6N}, we obtain, for every $0<t_0\leq t$,
\[
\int_M u(t,x)G_M^s(x,x_0)\,d\mu(x)
\leq
\int_M u(t_0,x)G_M^s(x,x_0)\,d\mu(x).
\]
We next let $t_0\downarrow0$. On $M\setminus B_{R_0(x_0)}(x_0)$, Lemma~\ref{lemlastlem}(1) gives a constant $C(x_0)>0$ such that
\begin{equation*}
\int_{M\setminus B_{R_0(x_0)}(x_0)}
|u(t_0,x)-u_0(x)|G_M^s(x,x_0)\,d\mu(x)
\leq
C(x_0)\|u(t_0)-u_0\|_{L^1(M)}\longrightarrow0.
\end{equation*}
On $B_{R_0(x_0)}(x_0)$, Proposition~\ref{prop:4.1} gives
\[
|u(t_0,x)-u_0(x)|G_M^s(x,x_0)
\leq
2\|u_0\|_{L^\infty(M)}G_M^s(x,x_0),
\]
and the right-hand side belongs to $L^1(B_{R_0(x_0)}(x_0))$ by Lemma~\ref{thm:lem1}(v). Since $u(t_0)\to u_0$ in $L^1(M)$, hence in measure, Vitali's theorem yields
\[
\int_{B_{R_0(x_0)}(x_0)}
|u(t_0,x)-u_0(x)|G_M^s(x,x_0)\,d\mu(x)
\longrightarrow0.
\]
This proves \eqref{5.1N}.

We now establish \eqref{5.6N}. For every $x\neq x_0$, the off-diagonal continuity of the Green kernel gives
\begin{equation}\label{5.10}
(-\Delta)^{-s}\eta_k(x)
=
\frac{1}{\operatorname{Vol}(B_{1/k}(x_0))}
\int_{B_{1/k}(x_0)}G_M^s(x,y)\,d\mu(y)
\longrightarrow G_M^s(x,x_0).
\end{equation}
Proposition~\ref{thm:3.4} yields
\begin{equation}\label{5.09}
0\leq(-\Delta)^{-s}\eta_k(x)
\leq\gamma_2(x_0)G_M^s(x,x_0),
\qquad k\in\mathbb N,\quad x\neq x_0.
\end{equation}
Consequently,
\begin{equation}\label{5.11}
0\leq u(\tau,x)(-\Delta)^{-s}\eta_k(x)
\leq\gamma_2(x_0)u(\tau,x)G_M^s(x,x_0)
\end{equation}
for every $k$ and almost every $x\in M$. Let $R\geq R_0(x_0)$. Since $u(\tau)\in L^1_{G_M^s,x_0}(M)$, \eqref{5.10}--\eqref{5.11} and dominated convergence give
\begin{equation}\label{5.1ls1}
\int_{M\setminus B_R(x_0)}
 u(\tau,x)\Bigl|(-\Delta)^{-s}\eta_k(x)-G_M^s(x,x_0)\Bigr|\,d\mu(x)
\longrightarrow0.
\end{equation}
For the integral over $B_R(x_0)$, Lemma~\ref{thm:lem1}(iv) implies
\[
G_M^s(\cdot,x_0)\in L^p(B_R(x_0))
\qquad\text{for every }1\leq p<\frac{n}{n-2s}.
\]
Together with \eqref{5.10} and \eqref{5.09}, dominated convergence gives
\[
(-\Delta)^{-s}\eta_k\longrightarrow G_M^s(\cdot,x_0)
\quad\text{in }L^p(B_R(x_0))
\qquad\text{for every }1\leq p<\frac{n}{n-2s}.
\]
Choose such a $p$ and let $p'$ be its Hölder conjugate. By Proposition~\ref{prop:4.1}, $u(\tau)\in L^{p'}(M)$, therefore
\begin{equation}\label{5.1ls2}
\begin{aligned}
&\int_{B_R(x_0)}
 u(\tau,x)\Bigl|(-\Delta)^{-s}\eta_k(x)-G_M^s(x,x_0)\Bigr|\,d\mu(x)
\\
&\qquad\leq
\|u(\tau)\|_{L^{p'}(M)}
\|(-\Delta)^{-s}\eta_k-G_M^s(\cdot,x_0)\|_{L^p(B_R(x_0))}
\longrightarrow0.
\end{aligned}
\end{equation}
Combining \eqref{5.1ls1} and \eqref{5.1ls2} proves \eqref{5.6N}.

It remains to prove \eqref{wdsmonf}. This follows as in
\cite[Proposition~3.3, Steps~3 and~4]{BerchioBonforteGangulyGrillo2020} and references therein,
using the time-monotonicity \eqref{eq:prop4.1mon}. The argument carries over
to the present setting line by line, with the limiting passages justified as
in the first part of the proof.
\end{proof}

\begin{prop}\label{lastpropwdsprops}
Let $M$ be a complete, noncompact, $s$-nonparabolic $n$-dimensional Riemannian manifold, $n>2s$, with $\operatorname{Ric}(M)\geq0$. Let $u$ be a WDS to \eqref{maineq} with nonnegative initial datum $u_0\in L^1(M)\cap L^\infty(M)$. Then, for every $x_0\in M$, there exists a constant $C(x_0)>0$ such that, for all $t\geq0$,
\begin{equation}\label{nonexpwighted}
\|u(t)\|_{L^1_{G_M^s,x_0}(M)}
\leq C(x_0)\|u_0\|_{L^1_{G_M^s,x_0}(M)}.
\end{equation}
More generally, if $u,v$ are two ordered WDSs to \eqref{maineq} corresponding to nonnegative initial data $u_0,v_0\in L^1(M)\cap L^\infty(M)$, then
\begin{equation}\label{ordering weighted2}
\|u(t)-v(t)\|_{L^1_{G_M^s,x_0}(M)}
\leq C(x_0)\|u_0-v_0\|_{L^1_{G_M^s,x_0}(M)},
\qquad t\geq0.
\end{equation}
In addition, for every $0<R\leq1$,
\begin{equation}\label{4.10nn}
R^{n-2s}\int_{M\setminus B_R(x_0)}u(t,x)G_M^s(x,x_0)\,d\mu(x)
\leq C(x_0)\|u(t)\|_{L^1_{G_M^s,x_0}(M)}.
\end{equation}
If, in addition, Assumptions \ref{A1} and \ref{A2} hold, the constants in \eqref{nonexpwighted}--\eqref{4.10nn} can be chosen independently of $x_0\in M$.
\end{prop}

\begin{proof}
Estimates \eqref{nonexpwighted} and \eqref{ordering weighted2} follow as in
\cite[Proposition~5.5]{BerchioBonforteGrilloMuratori2024}. The proof carries
over line by line, the only point to keep track of being the dependence of the
constant on the pole $x_0$. Indeed, fixing $x_0\in M$ and taking
$\eta=\psi_{x_0}$ in \eqref{5.3N}, with $\psi_{x_0}$ defined in
\eqref{eq:psi-x0}, Proposition~\ref{thm:lem3.6} yields
\eqref{nonexpwighted}; the same argument applied to the difference of two
ordered solutions gives \eqref{ordering weighted2}.

Finally, \eqref{4.10nn} follows as in \cite[Lemma~4.5]{FPMD}; see also \cite[Proposition~5.5] {BerchioBonforteGrilloMuratori2024}. Indeed, Proposition~\ref{thm:3.4}, applied to $\psi_{x_0}$ with $\sigma=1/2$, and Proposition~\ref{thm:lem3.6} give the same argument in the present fractional setting, with constants possibly depending on the pole $x_0$.

If \ref{A1} and \ref{A2} hold, Proposition~\ref{thm:lem3.6} and Remark~\ref{rmk:gamcont} show that the constants in the norm equivalence and the constant $\gamma_1$ in Proposition~\ref{thm:3.4} can be chosen independently of the pole. Hence, the constants in \eqref{nonexpwighted}--\eqref{4.10nn} can be chosen uniform with respect to $x_0$.

\end{proof}

\begin{rmk}\label{rem2222}
The ordering assumption in \eqref{ordering weighted2} can be dropped for
WDSs obtained by monotone approximation. More precisely, let $u$ and $v$
be WDSs corresponding to nonnegative initial data
$u_0,v_0\in L^1_{G_M^s,o}(M)$ and obtained as monotone limits of WDSs
associated with data in $L^1(M)\cap L^\infty(M)$. Then, for every
$x_0\in M$,
\[
\|u(t)-v(t)\|_{L^1_{G_M^s,x_0}(M)}
\leq
C(x_0)
\|u_0-v_0\|_{L^1_{G_M^s,x_0}(M)},
\qquad t\geq0.
\]
In particular, this yields continuous dependence on the initial datum and
uniqueness within this approximation class. If Assumptions \ref{A1} and \ref{A2}
hold, the constant $C(x_0)$ can be chosen independently of $x_0$.

Indeed, let
\[
0\leq u_{0,k}\uparrow u_0,
\qquad
0\leq v_{0,k}\uparrow v_0
\quad\text{a.e. in }M,
\]
with $u_{0,k},v_{0,k}\in L^1(M)\cap L^\infty(M)$, and let $u_k,v_k$ be
the corresponding WDSs. Set
\[
a_{0,k}:=u_{0,k}\wedge v_{0,k},
\qquad
b_{0,k}:=u_{0,k}\vee v_{0,k},
\]
and let $a_k,b_k$ be the corresponding WDSs. By the order-preserving
property of the mild solutions,
\[
a_k\leq u_k,v_k\leq b_k.
\]
Hence, using \eqref{ordering weighted2} for the ordered pair $(b_k,a_k)$,
\[
\begin{aligned}
\|u_k(t)-v_k(t)\|_{L^1_{G_M^s,x_0}(M)}
&\leq
\|b_k(t)-a_k(t)\|_{L^1_{G_M^s,x_0}(M)}
\\
&\leq
C(x_0)
\|b_{0,k}-a_{0,k}\|_{L^1_{G_M^s,x_0}(M)}=
C(x_0)
\|u_{0,k}-v_{0,k}\|_{L^1_{G_M^s,x_0}(M)}.
\end{aligned}
\]
By Proposition~\ref{prop:compact-poles-equivalence} and dominated
convergence, for every $x_0\in M$,
\[
\|u_{0,k}-u_0\|_{L^1_{G_M^s,x_0}(M)}
+
\|v_{0,k}-v_0\|_{L^1_{G_M^s,x_0}(M)}
\longrightarrow0.
\]
The same argument used in the proof of Theorem~\ref{thm:existence}, together
with the order-preserving property of the mild solutions, then gives
\[
\|u_k(t)-u(t)\|_{L^1_{G_M^s,x_0}(M)}
+
\|v_k(t)-v(t)\|_{L^1_{G_M^s,x_0}(M)}
\longrightarrow0.
\]
Passing to the limit in the above inequality yields the claim. The uniformity of
$C(x_0)$ under Assumptions \ref{A1} and \ref{A2} follows directly from
Proposition~\ref{lastpropwdsprops}.
\end{rmk}

\section{Proof of the Main Results}\label{sec:proofsofmainres}

\begin{proof}[Proof of Theorem \ref{thm:existence}]
Let
\[
u_{0,k}:=\min\{u_0,k\}\,\mathbf 1_{B_k(o)},
\qquad k\in\mathbb N.
\]
Then
\[
0\leq u_{0,k}\leq u_{0,k+1}\leq u_0,
\qquad
u_{0,k}\uparrow u_0
\quad\text{a.e. in }M,
\]
and \(u_{0,k}\in L^1(M)\cap L^\infty(M)\). By
Proposition~\ref{prop:compact-poles-equivalence}, for every \(x_0\in M\)
one has $u_0\in L^1_{G_M^s,x_0}(M)$
and, by dominated convergence,
\[
\|u_0-u_{0,k}\|_{L^1_{G_M^s,x_0}(M)}
\longrightarrow0.
\]
Let \(u_k\) be the weak dual solution associated with \(u_{0,k}\)
by Proposition~\ref{prop:existence}. By the order-preserving property
in Proposition~\ref{prop:4.1}, $0\leq u_k\leq u_{k+1}$.
We define
\[
u(t,x):=\lim_{k\to\infty}u_k(t,x).
\]
Fix \(x_0\in M\). For \(j\geq k\), the weighted stability estimate
\eqref{ordering weighted2} gives
\begin{equation*}
\|u_j(t)-u_k(t)\|_{L^1_{G_M^s,x_0}(M)}\leq C(x_0)
\|u_{0,j}-u_{0,k}\|_{L^1_{G_M^s,x_0}(M)}
\leq
C(x_0)
\|u_0-u_{0,k}\|_{L^1_{G_M^s,x_0}(M)}
\end{equation*}
for every \(t\geq0\). Letting \(j\to\infty\) and using monotone
convergence, we obtain, for every \(T>0\),
\[
\sup_{t\in[0,T]}
\|u(t)-u_k(t)\|_{L^1_{G_M^s,x_0}(M)}
\leq
C(x_0)
\|u_0-u_{0,k}\|_{L^1_{G_M^s,x_0}(M)}
\longrightarrow0.
\]
Since $u_k\in C\bigl([0,T];L^1_{G_M^s,x_0}(M)\bigr)$, it follows that
\[
u\in C\bigl([0,T];L^1_{G_M^s,x_0}(M)\bigr)
\]
for every \(x_0\in M\). Also, $u(0)=u_0$ in $L^1_{G_M^s,x_0}(M)$,
and hence almost everywhere in \(M\).

We next prove the local space-time integrability of \(u^m\). Let
\(K\Subset M\). Choose \(x_1,\ldots,x_N\in M\) such that
\[
K\subset\bigcup_{i=1}^N B_{1/2}(x_i),
\]
and set, as in \eqref{eq:psi-x0},
\[
\psi_i:=
\frac{\mathbf 1_{B_{1/2}(x_i)}}
{\operatorname{Vol}(B_{1/2}(x_i))}.
\]
For \(0<\delta<T\), applying \eqref{wdsagtestmon1} to \(u_k\) with
\(\eta=\psi_i\), and using Fubini--Tonelli, gives
\[
\begin{aligned}
\int_\delta^T\int_M
u_k^m(t,x)\psi_i(x)\,d\mu(x)\,dt
&=
\int_M
\bigl[u_k(\delta,x)-u_k(T,x)\bigr]
(-\Delta)^{-s}\psi_i(x)\,d\mu(x)
\\
&\leq
\int_M
u_k(\delta,x)(-\Delta)^{-s}\psi_i(x)\,d\mu(x).
\end{aligned}
\]
By Proposition~\ref{thm:lem3.6} and \eqref{nonexpwighted},
\[
\begin{aligned}
\int_M
u_k(\delta)(-\Delta)^{-s}\psi_i\,d\mu
&\leq
C(x_i)
\|u_k(\delta)\|_{L^1_{G_M^s,x_i}(M)}
\\
&\leq
C(x_i)
\|u_{0,k}\|_{L^1_{G_M^s,x_i}(M)}
\leq
C(x_i)
\|u_0\|_{L^1_{G_M^s,x_i}(M)}.
\end{aligned}
\]
Multiplying by \(\operatorname{Vol}(B_{1/2}(x_i))\), summing over
\(i=1,\ldots,N\), and then letting \(\delta\downarrow0\), we obtain
\[
\sup_{k\in\mathbb N}
\int_0^T\int_K u_k^m\,d\mu\,dt<+\infty.
\]
Since \(u_k\uparrow u\), the monotone convergence theorem yields
\[
u^m\in L^1\bigl((0,T);L^1(K)\bigr).
\]
Furthermore,
\[
u_k^m\longrightarrow u^m
\qquad\text{in }L^1((0,T)\times K).
\]

It remains to pass to the limit in the weak dual identity. We first
state the local estimate for the Green potential that will be used
below. Let
\[
B_i:=B_{1/2}(x_i).
\]
For every measurable \(w\) for which the right-hand side below is
finite, Tonelli's
theorem and Proposition~\ref{thm:lem3.6} give
\[
\begin{aligned}
\|(-\Delta)^{-s}w\|_{L^1(B_i)}
&\leq
\int_{B_i}(-\Delta)^{-s}|w|(x)\,d\mu(x)
=
\int_M
|w(y)|(-\Delta)^{-s}\mathbf 1_{B_i}(y)\,d\mu(y)
\\
&=
\operatorname{Vol}(B_i)
\int_M
|w(y)|(-\Delta)^{-s}\psi_i(y)\,d\mu(y)
\leq
C(x_i)
\|w\|_{L^1_{G_M^s,x_i}(M)}.
\end{aligned}
\]
Using the finite covering of \(K\) and the uniform weighted
convergence proved above, we conclude that
\[
\sup_{t\in[0,T]}
\|(-\Delta)^{-s}(u_k(t)-u(t))\|_{L^1(K)}
\longrightarrow0.
\]

Now let
\[
\psi\in C_c^1((0,T);L_c^\infty(M)),
\]
and let \(K\Subset M\) contain the spatial supports of
\(\psi(t)\) and \(\partial_t\psi(t)\) for all \(t\in[0,T]\).
Since each \(u_k\) is a weak dual solution,
\[
\int_0^T\int_M
\partial_t\psi\,(-\Delta)^{-s}u_k\,d\mu\,dt
=
\int_0^T\int_M
u_k^m\psi\,d\mu\,dt.
\]
The left-hand side converges to the corresponding expression with
\(u\) by the preceding local potential estimate, while the
right-hand side converges by
\[
u_k^m\longrightarrow u^m
\qquad\text{in }L^1((0,T)\times K).
\]
Therefore
\[
\int_0^T\int_M
\partial_t\psi\,(-\Delta)^{-s}u\,d\mu\,dt
=
\int_0^T\int_M
u^m\psi\,d\mu\,dt.
\]
Thus \(u\) satisfies all the requirements of
Definition~\ref{defn:wds} and is a weak dual solution with initial
datum \(u_0\).
\end{proof}

\begin{proof}[Proof of Theorem~\ref{thm:smoothing}]
If $u_0\equiv0$, then $u\equiv0$, so we assume throughout that
$u_0\not\equiv0$.

By an approximation argument using monotone sequences in $L^1(M)\cap L^\infty(M)$, the validity of the estimates below for initial data $u_0 \in L^1(M) \cap L^\infty(M)$, $u_0 \geq 0$, extends to general $u_0 \in L^1(M)$, $u_0 \geq 0$. Indeed, this can be proven by a simple adaptation of \cite[Proof of Theorem 2.4]{BerchioBonforteGrilloMuratori2024}.
In particular, if $0\leq u_{0,k}\uparrow u_0$ almost everywhere, with
$u_{0,k}\in L^1(M)\cap L^\infty(M)$ and
$\|u_{0,k}\|_{L^1(M)}\leq\|u_0\|_{L^1(M)}$, then all the estimates
proved below are independent of
$\|u_{0,k}\|_{L^\infty(M)}$ and pass to the corresponding monotone
limit. We may therefore assume without loss of generality that $u_0 \in L^1(M) \cap L^\infty(M)$ with $u_0 \geq 0$.

In Proposition \ref{thm:estonwds}, taking $t_1 = 2t_0$, Equation \eqref{wdsmonf} yields, for a.e. $x_0\in M, t>0$,
\begin{equation}\label{thm1.6split'}
u^m(t,x_0) \leq C \left( \frac{1}{t} \int_{M} u(t,x)G_M^s(x, x_0) \, d\mu(x) \right).
\end{equation}

\medskip
\noindent\emph{Short-time and basic long-time estimates.}
Fix a nonempty compact set $K\Subset M$ and set
\[
K_1:=\{x\in M:d(x,K)\leq1\}.
\]
We claim that there exists a constant
$C_K>0$ such that, for every $x_0\in K_1$ and every $0<r\leq1$,
\begin{equation}
\label{eq:proof33-local-green}
\int_{B_r(x_0)}G_M^s(x,x_0)\,d\mu(x)
\leq
C_Kr^{2s},
\qquad
\sup_{x\in M\setminus B_r(x_0)}
G_M^s(x,x_0)
\leq
C_Kr^{-(n-2s)}.
\end{equation}
Indeed, since
\[
\inf_{x_0\in K_1}\operatorname{Vol}(B_1(x_0))>0,
\]
Bishop--Gromov comparison, the upper Green-function estimate in
Theorem \ref{thm:bounGreen's}, and the uniform tail estimate established in the proof of
Lemma \ref{lemlastlem} give, uniformly for $x_0\in K_1$,
\[
G_M^s(x,x_0)
\leq
C_Kd(x,x_0)^{-(n-2s)}
\qquad
\text{if }0<d(x,x_0)<1.
\]
For $d(x,x_0)\geq1$, Lemma \ref{lemlastlem}(1) gives a uniform bound for
$x_0\in K_1$. The second inequality in
\eqref{eq:proof33-local-green} follows immediately. For the first one,
the coarea formula and Bishop--Gromov comparison yield
\[
\int_{B_r(x_0)}
\frac{d\mu(x)}{d(x,x_0)^{n-2s}}
=
r^{-(n-2s)}\operatorname{Vol}(B_r(x_0))
+(n-2s)
\int_0^r
\frac{\operatorname{Vol}(B_\rho(x_0))}
{\rho^{n-2s+1}}
\,d\rho
\leq
Cr^{2s}.
\]
This proves the first inequality in~\eqref{eq:proof33-local-green}.

For $0\leq\rho\leq1$, write
\[
K_\rho:=\{x\in M:d(x,K)\leq\rho\}.
\]
Let $0\leq\rho<\sigma\leq1$. For every $x_0\in K_\rho$, one has $B_{\sigma-\rho}(x_0)\subset K_\sigma$.
Splitting the integral in \eqref{thm1.6split'} over
$B_{\sigma-\rho}(x_0)$ and its complement, and using
\eqref{eq:proof33-local-green} and \eqref{nonexpms}, we obtain
\begin{align*}
\|u(t)\|_{L^\infty(K_\rho)}^m
&\leq
\frac{C_K}{t}
\left[
(\sigma-\rho)^{2s}
\|u(t)\|_{L^\infty(K_\sigma)}
+
(\sigma-\rho)^{-(n-2s)}
\|u_0\|_{L^1(M)}
\right].
\end{align*}
Taking the $m$-th root and applying Young's inequality, for every
$\eta\in(0,1)$ we get
\begin{equation}\label{eq:proof33-local-young}
\|u(t)\|_{L^\infty(K_\rho)}
\leq
\eta
\|u(t)\|_{L^\infty(K_\sigma)}
+\frac{C_{K,\eta}}{t^{\frac{1}{m-1}}}
(\sigma-\rho)^{\frac{2s}{m-1}}
+
\frac{C_K}{t^{\frac1m}}
\|u_0\|_{L^1(M)}^{\frac1m}
(\sigma-\rho)^{-\frac{n-2s}{m}}.
\end{equation}

Fix $0<R\leq1$ and define for every $j\in\mathbb{N}$
\[
\rho_j:=R(1-2^{-j}).
\]
Then
\[
\rho_{j+1}-\rho_j=2^{-(j+1)}R.
\]
Applying \eqref{eq:proof33-local-young} successively with
$\rho=\rho_j$ and $\sigma=\rho_{j+1}$ gives, for every $k\geq1$,
\[
\begin{aligned}
\|u(t)\|_{L^\infty(K)}
&\leq
\eta^k
\|u(t)\|_{L^\infty(K_{\rho_k})}
+
C_{K,\eta}
t^{-\frac{1}{m-1}}
R^{\frac{2s}{m-1}}
\sum_{j=0}^{k-1}
\eta^j
2^{-\frac{2s(j+1)}{m-1}}
\\
&\quad
+
C_K
t^{-\frac1m}
\|u_0\|_{L^1(M)}^{\frac1m}
R^{-\frac{n-2s}{m}}
\sum_{j=0}^{k-1}
\eta^j
2^{\frac{(n-2s)(j+1)}{m}}.
\end{aligned}
\]
Choose
\[
0<\eta<
2^{-\frac{n-2s}{m}}.
\]
Both geometric sums are then bounded independently of $k$. Proposition \ref{prop:4.1} yields
\[
\|u(t)\|_{L^\infty(K_{\rho_k})}
\leq
\|u_{0}\|_{L^\infty(M)},
\]
so that the first term tends to zero as $k\to+\infty$. Hence, for every
$0<R\leq1$,
\begin{equation}\label{eq:proof33-small-radius-master}
\|u(t)\|_{L^\infty(K)}
\leq
C_K
\left[
t^{-\frac{1}{m-1}}
R^{\frac{2s}{m-1}}
+
t^{-\frac1m}
\|u_0\|_{L^1(M)}^{\frac1m}
R^{-\frac{n-2s}{m}}
\right].
\end{equation}

Now if
\[
0<t\leq
\|u_0\|_{L^1(M)}^{-(m-1)},
\]
we choose
\[
R=
\left(
t\|u_0\|_{L^1(M)}^{m-1}
\right)^{\frac{1}{n(m-1)+2s}}
\leq1.
\]
Then
\[
t^{-\frac{1}{m-1}}
R^{\frac{2s}{m-1}}
=
t^{-\frac{n}{n(m-1)+2s}}
\|u_0\|_{L^1(M)}^{\frac{2s}{n(m-1)+2s}},
\]
and
\[
t^{-\frac1m}
\|u_0\|_{L^1(M)}^{\frac1m}
R^{-\frac{n-2s}{m}}
=
t^{-\frac{n}{n(m-1)+2s}}
\|u_0\|_{L^1(M)}^{\frac{2s}{n(m-1)+2s}}.
\]
This proves the short-time estimate. If instead
\[
t\geq
\|u_0\|_{L^1(M)}^{-(m-1)},
\]
we choose $R=1$ in \eqref{eq:proof33-small-radius-master}. Since
\[
t^{-\frac{1}{m-1}}
=
t^{-\frac1m}
\|u_0\|_{L^1(M)}^{\frac1m}
\left(
t\|u_0\|_{L^1(M)}^{m-1}
\right)^{-\frac{1}{m(m-1)}}
\leq
t^{-\frac1m}
\|u_0\|_{L^1(M)}^{\frac1m},
\]
we obtain
\[
\|u(t)\|_{L^\infty(K)}
\leq
C_K
t^{-\frac1m}
\|u_0\|_{L^1(M)}^{\frac1m}.
\]
This proves the basic long-time estimate.

\medskip
\noindent\emph{Refined local long-time estimate.}
Fix $o\in M$ and a nonempty compact set $K\Subset M$. Choose
\[
R_{K,o}:=
\max\left\{
1,
2\max_{x\in K}d(x,o)
\right\}.
\]
For $R\geq R_{K,o}$ and $j\in\mathbb{N}$, define
\[
K_j:=
\left\{
x\in M:
d(x,K)\leq
\frac{3^j-1}{2}R
\right\}.
\]
Then $B_{3^jR}(x_0)\subset K_{j+1}$  for every $x_0\in K_j$, and
\[
d(x_0,o)
\leq
\max_{x\in K}d(x,o)
+
\frac{3^j-1}{2}R
\leq
\frac{3^jR}{2}
\]
for every $x_0\in K_j$. Consequently, for every $j\in\mathbb{N}$, $x_0\in K_j$, $r\geq3^jR$,
\[
B_{r/2}(o)\subset B_r(x_0),
\]
while
\[
B_{3^jR}(x_0)
\subset
B_{2\cdot3^jR}(o).
\]
Then by volume doubling,
\begin{equation}
\label{eq:proof33-volume-comparison-moving-pole}
\operatorname{Vol}(B_r(x_0))
\geq
C_d^{-1}\operatorname{Vol}(B_r(o))
\end{equation}
and
\[
\operatorname{Vol}(B_{3^jR}(x_0))
\leq
C_d\operatorname{Vol}(B_{3^jR}(o)).
\]

Using the upper Green-function
estimate in Theorem \ref{thm:bounGreen's}, Tonelli's theorem, the coarea formula, and
integration by parts, we obtain
\[
\begin{aligned}
\int_{B_{3^jR}(x_0)}
G_M^s(x,x_0)\,d\mu(x)
&\leq
C\operatorname{Vol}(B_{3^jR}(x_0))
\int_{3^jR}^{+\infty}
\frac{r^{2s-1}}
{\operatorname{Vol}(B_r(x_0))}
\,dr
+
C\frac{(3^jR)^{2s}}{2s}
\\
&\leq
C\operatorname{Vol}(B_{3^jR}(o))
\int_{3^jR}^{+\infty}
\frac{r^{2s-1}}
{\operatorname{Vol}(B_r(o))}
\,dr
+
C\frac{(3^jR)^{2s}}{2s}
\leq
Ch_o(3^jR).
\end{aligned}
\]
Also, if $x\notin B_{3^jR}(x_0)$, then by Theorem \ref{thm:bounGreen's} and
\eqref{eq:proof33-volume-comparison-moving-pole},
\[
G_M^s(x,x_0)
\leq
C
\int_{d(x,x_0)}^{+\infty}
\frac{r^{2s-1}}
{\operatorname{Vol}(B_r(x_0))}
\,dr
\leq
C
\int_{3^jR}^{+\infty}
\frac{r^{2s-1}}
{\operatorname{Vol}(B_r(o))}
\,dr
\leq
C
\frac{h_o(3^jR)}
{\operatorname{Vol}(B_{3^jR}(o))}.
\]

We shall also use the following growth bounds. By Bishop--Gromov
comparison, for every $\lambda\geq1$ and every $R\geq1$,
\begin{equation*}\begin{aligned}
h_o(\lambda R)
&\leq
\lambda^n h_o(R),
\\
\frac{h_o(\lambda R)}
{\operatorname{Vol}(B_{\lambda R}(o))}
&\leq
\lambda^n
\frac{h_o(R)}
{\operatorname{Vol}(B_R(o))}.
\end{aligned}
\end{equation*}
Indeed,
\[
\begin{aligned}
h_o(\lambda R)
&=
\operatorname{Vol}(B_{\lambda R}(o))
\int_{\lambda R}^{+\infty}
\frac{r^{2s-1}}
{\operatorname{Vol}(B_r(o))}
\,dr
+
\frac{(\lambda R)^{2s}}{2s}
\\
&\leq
\lambda^n
\operatorname{Vol}(B_R(o))
\int_R^{+\infty}
\frac{r^{2s-1}}
{\operatorname{Vol}(B_r(o))}
\,dr
+
\lambda^n
\frac{R^{2s}}{2s}
=
\lambda^nh_o(R),
\end{aligned}
\]
where we used $2s\leq n$. The second estimate follows immediately from the first.

Applying \eqref{thm1.6split'} with $x_0\in K_j$ and splitting
at radius $3^jR$, the preceding estimates, the inclusion
$B_{3^jR}(x_0)\subset K_{j+1}$, and the $L^1$-nonexpansivity of $u$
give
\begin{equation*}
\begin{aligned}
\|u(t)\|_{L^\infty(K_j)}^m
&\leq
\frac{C}{t}
\left[
h_o(3^jR)
\|u(t)\|_{L^\infty(K_{j+1})}
+
\frac{h_o(3^jR)}
{\operatorname{Vol}(B_{3^jR}(o))}
\|u_0\|_{L^1(M)}
\right]
\notag\\
&\leq
\frac{C3^{jn}}{t}
\left[
h_o(R)
\|u(t)\|_{L^\infty(K_{j+1})}
+
\frac{h_o(R)}
{\operatorname{Vol}(B_R(o))}
\|u_0\|_{L^1(M)}
\right].
\end{aligned}
\end{equation*}
Taking the $m$-th root and applying Young's inequality to the term
containing
$\|u(t)\|_{L^\infty(K_{j+1})}^{1/m}$, for every
$\eta\in(0,1)$ we obtain
\begin{align}
\|u(t)\|_{L^\infty(K_j)}
&\leq
\eta
\|u(t)\|_{L^\infty(K_{j+1})}
+
C_\eta
3^{\frac{jn}{m-1}}
t^{-\frac{1}{m-1}}
h_o(R)^{\frac{1}{m-1}}
\notag\\
&\quad
+
C
3^{\frac{jn}{m}}
t^{-\frac1m}
\left(
\frac{h_o(R)}
{\operatorname{Vol}(B_R(o))}
\|u_0\|_{L^1(M)}
\right)^{\frac1m}.
\label{eq:proof33-large-young}
\end{align}
Choose
\[
0<\eta<
\min\left\{
3^{-\frac{n}{m-1}},
3^{-\frac{n}{m}}
\right\}.
\]
Iterating \eqref{eq:proof33-large-young} from $j=0$ to $j=k-1$ yields
\[
\begin{aligned}
\|u(t)\|_{L^\infty(K)}
&\leq
\eta^k
\|u(t)\|_{L^\infty(K_k)}
+
C_\eta
t^{-\frac{1}{m-1}}
h_o(R)^{\frac{1}{m-1}}
\sum_{j=0}^{k-1}
\eta^j
3^{\frac{jn}{m-1}}
\\
&\quad
+
C
t^{-\frac1m}
\left(
\frac{h_o(R)}
{\operatorname{Vol}(B_R(o))}
\|u_0\|_{L^1(M)}
\right)^{\frac1m}
\sum_{j=0}^{k-1}
\eta^j
3^{\frac{jn}{m}}.
\end{aligned}
\]
The two geometric sums are bounded independently of $k$. Proposition \ref{prop:4.1} gives
\[
\|u(t)\|_{L^\infty(K_k)}
\leq
\|u_{0}\|_{L^\infty(M)},
\]
so the first term tends to zero as $k\to+\infty$. We thus obtain, for every
$R\geq R_{K,o}$,
\begin{align}
\|u(t)\|_{L^\infty(K)}
&\leq
Ct^{-\frac{1}{m-1}}
h_o(R)^{\frac{1}{m-1}}
+
C
t^{-\frac1m}
\left(
\frac{h_o(R)}
{\operatorname{Vol}(B_R(o))}
\|u_0\|_{L^1(M)}
\right)^{\frac1m}\notag\\
&\leq
C
t^{-\frac{1}{m-1}}
h_o(R)^{\frac{1}{m-1}}
\left(
1+
\frac{
t^{\frac{1}{m-1}}
\|u_0\|_{L^1(M)}
}{
\operatorname{Vol}(B_R(o))
h_o(R)^{\frac{1}{m-1}}
}
\right)^{\frac1m}
\notag\\
&=
C
t^{-\frac{1}{m-1}}
h_o(R)^{\frac{1}{m-1}}
\left(
1+
\frac{
t^{\frac{1}{m-1}}
\|u_0\|_{L^1(M)}
}{
\theta_o(R)
}
\right)^{\frac1m}
\label{eq:proof33-local-master-R}
\end{align}
for every $t>0$. Taking the infimum
over $R\geq R_{K,o}$ gives the local master estimate stated in
Remark \ref{rem111}.

It is easy to see that $h_0$ is increasing, and therefore so is
$\theta_o$. Also, $\theta_o$ tends to $\infty$ as $R$ tends to $\infty$.
Hence, whenever
\[
t\geq
\theta_o(R_{K,o})^{m-1}
\|u_0\|_{L^1(M)}^{-(m-1)},
\]
we may take
\[
R=
\theta_o^{-1}\!\left(
t^{\frac{1}{m-1}}
\|u_0\|_{L^1(M)}
\right)
\geq
R_{K,o}
\]
in \eqref{eq:proof33-local-master-R}. This gives
\[
\|u(t)\|_{L^\infty(K)}
\leq
C
t^{-\frac{1}{m-1}}
\left[
h_o\!\left(
\theta_o^{-1}\!\left(
t^{\frac{1}{m-1}}
\|u_0\|_{L^1(M)}
\right)
\right)
\right]^{\frac{1}{m-1}},
\]
which is the refined local long-time estimate.
\end{proof}

\begin{proof}[Proof of Corollary~\ref{cor:global-smoothing-L1}]
By the monotone-approximation argument at the beginning of the proof
of Theorem~\ref{thm:smoothing}, it is enough to consider
$u_0\in L^1(M)\cap L^\infty(M)$, $u_0\geq0$.

We first derive the global short-time and basic long-time estimates.
Under Assumptions \ref{A1} and \ref{A2}, Remark \ref{rmk:gamcont} implies
that the small-scale Green-function estimates used in the proof of
Theorem~\ref{thm:smoothing} are uniform with respect to the pole.
Therefore, for every $x_0\in M$ and every $0<R\leq1$,
\[
\int_{B_R(x_0)}
G_M^s(x,x_0)\,d\mu(x)
\leq
CR^{2s},
\qquad
\sup_{x\in M\setminus B_R(x_0)}
G_M^s(x,x_0)
\leq
CR^{-(n-2s)},
\]
with a constant independent of $x_0$.

Starting from the potential estimate established at the beginning of
the proof of Theorem~\ref{thm:smoothing}, splitting at radius $R$,
using the $L^1$-nonexpansivity, and taking the essential supremum over
$x_0\in M$, we obtain
\[
\|u(t)\|_{L^\infty(M)}^m
\leq
\frac{C}{t}
R^{2s}
\|u(t)\|_{L^\infty(M)}
+
\frac{C}{t}
R^{-(n-2s)}
\|u_0\|_{L^1(M)}.
\]
Taking the $m$-th root and applying Young's inequality,
\[
\|u(t)\|_{L^\infty(M)}
\leq
C
\left[
t^{-\frac{1}{m-1}}
R^{\frac{2s}{m-1}}
+
t^{-\frac1m}
R^{-\frac{n-2s}{m}}
\|u_0\|_{L^1(M)}^{\frac1m}
\right].
\]
This is the global counterpart of the small-radius estimate obtained
in the proof of Theorem~\ref{thm:smoothing}. Choosing $R=\left(t\|u_0\|_{L^1(M)}^{m-1}
\right)^{\frac{1}{n(m-1)+2s}}$ when $0<t\leq \|u_0\|_{L^1(M)}^{-(m-1)}$ gives~\eqref{eq:shorttime}, while taking $R=1$ when
$t\geq\|u_0\|_{L^1(M)}^{-(m-1)}$
gives~\eqref{eq:longtime}, exactly as in the proof of
Theorem~\ref{thm:smoothing}.

We next prove the refined global estimate. Let $R\geq R_0$.
Lemma~\ref{thm:lem1}(vi) gives, uniformly in $x_0\in M$,
\[
\int_{B_R(x_0)}
G_M^s(x,x_0)\,d\mu(x)
\leq
Ch(R).
\]
If $x\in M\setminus B_R(x_0)$ and
$\rho=d(x,x_0)$, then Lemma~\ref{thm:lem1}(ii),
Assumption \ref{A2}, and the monotonicity of
$r\mapsto\int_r^{+\infty}f(\tau)^{-1}\,d\tau$ yield
\[
G_M^s(x,x_0)\leq
C
\frac{\rho^{2s-1}f(\rho)}
{\operatorname{Vol}(B_\rho(x_0))}
\int_\rho^{+\infty}\frac{d\tau}{f(\tau)}
\leq
C
\frac{R^{2s-1}f(R)}
{\operatorname{Vol}(B_R(x_0))}
\int_R^{+\infty}\frac{d\tau}{f(\tau)}
\leq
C\frac{h(R)}{F(R)}.
\]
Repeating the splitting argument above therefore gives
\[
\|u(t)\|_{L^\infty(M)}^m
\leq
\frac{C}{t}
h(R)\|u(t)\|_{L^\infty(M)}
+
\frac{C}{t}
\frac{h(R)}{F(R)}
\|u_0\|_{L^1(M)}.
\]
After Young's inequality and absorption,
\[
\|u(t)\|_{L^\infty(M)}
\leq
C
t^{-\frac{1}{m-1}}
h(R)^{\frac{1}{m-1}}
\left(
1+
\frac{
t^{\frac{1}{m-1}}
\|u_0\|_{L^1(M)}
}{
F(R)h(R)^{\frac{1}{m-1}}
}
\right)^{\frac1m}.
\]
Recalling~\eqref{thetadefn}, this becomes
\[
\|u(t)\|_{L^\infty(M)}
\leq
C
t^{-\frac{1}{m-1}}
h(R)^{\frac{1}{m-1}}
\left(
1+
\frac{
t^{\frac{1}{m-1}}
\|u_0\|_{L^1(M)}
}{
\theta(R)
}
\right)^{\frac1m}
\]
for every $R\geq R_0$. Taking the infimum over $R\geq R_0$
proves the global master estimate stated in Remark~\ref{rem3333}.

It is easy to see that $h,F,\theta$ are increasing and that
$\theta(R)\to+\infty$ as $R\to+\infty$.
Hence, if
\[
t\geq
\theta(R_0)^{m-1}
\|u_0\|_{L^1(M)}^{-(m-1)},
\]
we may choose
\[
R=
\theta^{-1}\!\left(
t^{\frac{1}{m-1}}
\|u_0\|_{L^1(M)}
\right)
\]
in the preceding estimate. Then we obtain
\[
\|u(t)\|_{L^\infty(M)}
\leq
C
t^{-\frac{1}{m-1}}
\left[
h\!\left(
\theta^{-1}\!\left(
t^{\frac{1}{m-1}}
\|u_0\|_{L^1(M)}
\right)
\right)
\right]^{\frac{1}{m-1}},
\]
namely~\eqref{eq:improved}. The last assertion of Remark~\ref{rem3333} follows in the same way: if
$\widehat{\theta}$ is increasing, divergent, and satisfies
\[
\widehat{\theta}(R)
\leq
C F(R)h(R)^{\frac{1}{m-1}},
\]
one takes
\[
R=
\widehat{\theta}^{-1}\!\left(
t^{\frac{1}{m-1}}
\|u_0\|_{L^1(M)}
\right)
\]
in the global master estimate.
\end{proof}

\begin{proof}[Proof of Corollary \ref{cor1}]
We begin with (i). Under the stated assumptions, we obtain the following scaling relations:
\begin{equation*}
    \int_R^\infty \frac{\mathrm{d}t}{f(t)} \asymp \frac{1}{R^{k-2s}(\log R)^\delta}, \quad
    h(R) \asymp R^{2s}, \quad
    F(R) \geq CR^\lambda.
\end{equation*}
Define $\widehat{\theta}(R) := R^{\lambda + \frac{2s}{m-1}}$. One then verifies that:
\begin{equation*}
    \widehat{\theta}(R) \leq C F(R) \big(h(R)\big)^{\frac{1}{m-1}} \quad \text{and} \quad \widehat{\theta}^{-1}(r) = r^{\frac{m-1}{(m-1)\lambda + 2s}}.
\end{equation*}
For any nonnegative initial datum $u_0 \in L^1(M)$, the claimed decay estimate follows directly from Corollary~\ref{cor:global-smoothing-L1} and Remark~\ref{rem3333}.

We now address (ii). Under the assumptions, we establish the following scaling relations:
\begin{equation*}
    \int_R^\infty \frac{\mathrm{d}t}{f(t)} \asymp \frac{1}{(\log R)^{\delta-1}}, \quad
    h(R) \asymp R^{2s} \log R, \quad
    F(R) \geq CR^\lambda (\log R)^\sigma.
\end{equation*}
Define $\widehat{\theta}(R) := R^{\lambda + \frac{2s}{m-1}} (\log R)^{\sigma + \frac{1}{m-1}}$, which satisfies:
\begin{equation*}
    \widehat{\theta}(R) \leq C F(R) \big(h(R)\big)^{\frac{1}{m-1}}.
\end{equation*}
For $r$ large enough, its inverse function $G(r) := \widehat{\theta}^{-1}(r)$ takes the form
\begin{equation*}
    G(r) = \begin{cases}
    \exp\left(\frac{b}{a} W_0\left(\frac{a}{b} r^{1/b}\right)\right) & \text{if }b> 0,\\\  r^{\frac{m-1}{\lambda(m-1) + 2s}} & \text{if } b = 0,\\
     \exp\left(\frac{b}{a} W_{-1}\left(\frac{a}{b} r^{1/b}\right)\right) & \text{if } b<0,
    \end{cases}
\end{equation*}
where
\begin{equation*}
    a = \lambda + \frac{2s}{m-1}, \quad b = \sigma + \frac{1}{m-1},
\end{equation*}
$W_0$ denotes the inverse of $x \mapsto xe^x$ on $[-1,\infty)$ and $W_{-1}$ denotes the inverse of $x \mapsto xe^x$ on $(-\infty,-1]$.

For any nonnegative initial datum $u_0 \in L^1(M)$, the stated decay estimate follows immediately from Corollary~\ref{cor:global-smoothing-L1} and Remark~\ref{rem3333}.
\end{proof}

\begin{proof}[Proof of Theorem~\ref{thm:smoothing2}]
If \(u_0\equiv0\), then \(u\equiv0\), so we assume that \(u_0\not\equiv0\).

We first assume that $u_0\in L^1(M)\cap L^\infty(M)$, $u_0\geq0$.
Exactly as at the beginning of the proof of Theorem~\ref{thm:smoothing},
Proposition~\ref{thm:estonwds} and \eqref{wdsmonf} give, for almost every
\(x_0\in M\) and every \(t>0\),
\[
u^m(t,x_0)
\leq
\frac{C}{t}
\int_M u(t,x)G_M^s(x,x_0)\,d\mu(x).
\]

Fix a nonempty compact set \(K\Subset M\), and set, as in the proof of
Theorem~\ref{thm:smoothing},
\[
K_1:=\{x\in M:d(x,K)\leq1\}.
\]
By Proposition~\ref{prop:compact-poles-equivalence} and
\eqref{nonexpwighted}, applied at the fixed pole \(o\), there exists a
constant \(C_{K,o}>0\) such that, for every \(x_0\in K_1\) and every
\(t\geq0\),
\[
\| u(t)\|_{L^1_{G_M^s,x_0}(M)}
\leq
\| u(t)\|_{L^1_{G_M^s,K_1}(M)}
\leq
C_{K,o}\| u(t)\|_{L^1_{G_M^s,o}(M)}
\leq
C_{K,o}\| u_0\|_{L^1_{G_M^s,o}(M)}.
\]

We also use the local Green-function bounds
\eqref{eq:proof33-local-green}, already established in the proof of
Theorem~\ref{thm:smoothing}. Thus, uniformly for \(x_0\in K_1\) and
\(0<r\leq1\),
\[
\int_{B_r(x_0)}G_M^s(x,x_0)\,d\mu(x)
\leq
C_Kr^{2s},
\qquad
\sup_{x\in M\setminus B_r(x_0)}G_M^s(x,x_0)
\leq
C_Kr^{-(n-2s)}.
\]
Since \(n-2s>0\), for every \(x_0\in K_1\) and \(0<r\leq1\) we therefore
have
\begin{align*}
\int_{M\setminus B_r(x_0)}
u(t,x)&G_M^s(x,x_0)\,d\mu(x)
\\
&=
\int_{B_1(x_0)\setminus B_r(x_0)}
u(t,x)G_M^s(x,x_0)\,d\mu(x)
+
\int_{M\setminus B_1(x_0)}
u(t,x)G_M^s(x,x_0)\,d\mu(x)
\\
&\leq
C_Kr^{-(n-2s)}
\int_{B_1(x_0)}u(t,x)\,d\mu(x)
+
\int_{M\setminus B_1(x_0)}
u(t,x)G_M^s(x,x_0)\,d\mu(x)
\\
&\leq
C_Kr^{-(n-2s)}
\| u(t)\|_{L^1_{G_M^s,x_0}(M)}
\leq
C_{K,o}r^{-(n-2s)}
\| u_0\|_{L^1_{G_M^s,o}(M)}.
\end{align*}

For \(0\leq\rho\leq1\), let
\[
K_\rho:=\{x\in M:d(x,K)\leq\rho\}.
\]
If \(0\leq\rho<\sigma\leq1\), then for every $x_0\in K_\rho$ one has $B_{\sigma-\rho}(x_0)\subset K_\sigma$.
Splitting the potential at radius \(\sigma-\rho\) and using the preceding
estimates, we obtain
\[
\| u(t)\|_{L^\infty(K_\rho)}^m
\leq
\frac{C_{K,o}}{t}
\left[
(\sigma-\rho)^{2s}
\| u(t)\|_{L^\infty(K_\sigma)}
+
(\sigma-\rho)^{-(n-2s)}
\| u_0\|_{L^1_{G_M^s,o}(M)}
\right].
\]
This is precisely the analogue of the recursive estimate obtained in the proof of Theorem \ref{thm:smoothing} immediately before \eqref{eq:proof33-local-young}, with
\(\| u_0\|_{L^1(M)}\) replaced by
\(\| u_0\|_{L^1_{G_M^s,o}(M)}\).  Hence the Young-inequality argument and the dyadic iteration used in the proof of
Theorem~\ref{thm:smoothing},  leading to
\eqref{eq:proof33-small-radius-master}, apply without change and yield,
for every \(0<R\leq1\),
\begin{equation}\label{eq11223344}
\| u(t)\|_{L^\infty(K)}
\leq
C_{K,o}
\left[
t^{-\frac{1}{m-1}}
R^{\frac{2s}{m-1}}
+
t^{-\frac{1}{m}}
\| u_0\|_{L^1_{G_M^s,o}(M)}^{\frac{1}{m}}
R^{-\frac{n-2s}{m}}
\right]
\end{equation}
for every \(t>0\).

We now remove the assumption \(u_0\in L^1(M)\cap L^\infty(M)\). For a
general nonnegative \(u_0\in L^1_{G_M^s,o}(M)\), let
\[
u_{0,k}:=\min\{u_0,k\}\mathbf{1}_{B_k(o)},
\qquad
k\in\mathbb{N},
\]
and let \(u_k\) be the corresponding weak dual solutions. As in the proof
of Theorem~\ref{thm:existence},
\[
0\leq u_{0,k}\uparrow u_0
\quad\text{a.e. in }M,
\qquad
0\leq u_k\uparrow u
\quad\text{a.e. in }(0,+\infty)\times M,
\]
and
\[
\| u_{0,k}\|_{L^1_{G_M^s,o}(M)}
\leq
\| u_0\|_{L^1_{G_M^s,o}(M)}.
\]
Applying the preceding estimate to \(u_k\), and passing to the monotone limit we obtain, for every \(0<R\leq1\), the validity of \eqref{eq11223344}
for every \(t>0\), also for a generic nonnegative initial datum \(u_0\in L^1_{G_M^s,o}(M)\).
%
Choosing
\[
R=
\left(
t\|u_0\|_{L^1_{G_M^s,o}(M)}^{m-1}
\right)^{\frac{1}{n(m-1)+2s}}
\]
when
\[
0<t\leq
\|u_0\|_{L^1_{G_M^s,o}(M)}^{-(m-1)},
\]
and taking $R=1$ otherwise, the same computation as in the proof of
Theorem~\ref{thm:smoothing} yields the asserted short- and long-time
estimates.
\end{proof}

\begin{proof}[Proof of Corollary~\ref{cor:global-smoothing-weighted}]
By the same monotone-approximation argument used in the proof of
Theorem~\ref{thm:smoothing2}, it is enough to consider $u_0\in L^1(M)\cap L^\infty(M)$, $u_0\geq0$.

Under Assumptions \ref{A1} and \ref{A2}, Remark~\ref{rmk:gamcont} implies that
the small-scale Green-function estimates used in the proof of
Theorem~\ref{thm:smoothing2} are uniform with respect to the pole, while
Proposition~\ref{lastpropwdsprops} gives the corresponding weighted
estimates with constants independent of \(x_0\in M\). Hence, arguing as
in the proof of Theorem~\ref{thm:smoothing2}, for every \(0<R\leq1\) we
obtain
\[
\| u(t)\|_{L^\infty(M)}^m
\leq
\frac{C}{t}
\left[
R^{2s}\| u(t)\|_{L^\infty(M)}
+
R^{-(n-2s)}
\| u_0\|_{L^1_{G_M^s}(M)}
\right]
\]
for every \(t>0\). Young's inequality therefore yields
\[
\| u(t)\|_{L^\infty(M)}
\leq
C
\left[
t^{-\frac{1}{m-1}}R^{\frac{2s}{m-1}}
+
t^{-\frac{1}{m}}
\| u_0\|_{L^1_{G_M^s}(M)}^{\frac{1}{m}}
R^{-\frac{n-2s}{m}}
\right].
\]

The same monotone-approximation argument extends this estimate to every
nonnegative \(u_0\in L^1_{G_M^s}(M)\). Finally, choosing
\[
R=
\left(
t\| u_0\|_{L^1_{G_M^s}(M)}^{m-1}
\right)^{\frac{1}{n(m-1)+2s}}
\]
when $0<t\leq\| u_0\|_{L^1_{G_M^s}(M)}^{-(m-1)}$
and \(R=1\) otherwise, exactly as in the proof of
Theorem~\ref{thm:smoothing2}, gives the two claimed estimates.
\end{proof}

\section{Extension to General Nonlinearities}\label{S11}

The preceding arguments extend to equations of the form
\[
\partial_t u+(-\Delta)^s\Phi(u)=0.
\]
Assume that \(\Phi:\mathbb{R}\to\mathbb{R}\) is continuous and nondecreasing, with
\(\Phi(0)=0\) and \(\Phi\in C^1(\mathbb{R}\setminus\{0\})\), and that the quotient
\(\Phi/\Phi'\), defined wherever \(\Phi'\neq0\) and set equal to zero wherever
\(\Phi=\Phi'=0\), extends to a Lipschitz function on \(\mathbb{R}\). Suppose also
that there exist \(\mu_0,\mu_1>0\), with \(\mu_0<1\), such that
\[
1-\mu_1
\le
\left(\frac{\Phi}{\Phi'}\right)'
\le
1-\mu_0
\qquad\text{a.e. in }\mathbb{R}.
\]
These are the structural assumptions underlying the Crandall--Pierre regularizing
estimates, see \cite{crandallpierre}. The notion of weak dual solution is obtained from Definition~\ref{defn:wds} by
requiring
\[
\Phi(u)\in L^1\bigl((0,T);L^1_{\mathrm{loc}}(M)\bigr)
\]
in item~(ii) and by replacing \(u^m\) with \(\Phi(u)\) in
\eqref{wds-identity}. The nonlinear-semigroup construction, the \(L^p\)-nonexpansivity,
and the order-preserving arguments of Section~\ref{sec:frommstowds} carry over to this setting; the B\'enilan--Crandall time-monotonicity
\eqref{eq:prop4.1mon} is replaced by
\[
 t\longmapsto t^{1/\mu_0}\Phi(u(t,x))
 \quad\text{is nondecreasing on }(0,+\infty)
 \text{ for a.e. }x\in M.
\]
The potential and
weighted-stability arguments of Section~\ref{sec: propswds} can be adapted accordingly. Consequently, for every nonnegative \(u_0\in L^1_{G_M^s,o}(M)\) there exists a weak dual solution,
obtained by monotone approximation.

Furthermore, the Crandall--Pierre estimate and the argument of Proposition~\ref{thm:estonwds}
give
\begin{equation}\label{eq:general-fundamental}
\Phi(u(t,x_0))
\le
\frac{C}{t}
\int_M u(t,x)G_M^s(x,x_0)\,d\mu(x)
\end{equation}
for every \(t>0\) and almost every \(x_0\in M\).

Assume in addition that \(\Phi(r)>0\) for every \(r>0\), and set
\[
H_\Phi(0):=0,
\qquad
H_\Phi(r):=\frac{\Phi(r)}{r^{1-2s/n}},
\qquad r>0.
\]
The structural assumptions imply that \(\Phi\) is convex and superlinear on
\([0,+\infty)\); in particular, \(H_\Phi\) is nondecreasing and
\(H_\Phi(r)\to+\infty\) as \(r\to+\infty\). We denote by \(\Phi^{-1}\) and
\(H_\Phi^{-1}\) the corresponding generalized inverses.

The localization and iteration arguments in the proofs of Theorems~\ref{thm:smoothing}
and~\ref{thm:smoothing2}, with \eqref{eq:general-fundamental} in place of the
power estimate, then give the following bounds. If
\(0\not\equiv u_0\in L^1(M)\), then for every nonempty compact set \(K\Subset M\)
there exist \(C_K,c_K>0\) such that
\begin{equation}\label{eq:general-L1-short}
\|u(t)\|_{L^\infty(K)}
\le
H_\Phi^{-1}\left(
\frac{C_K\|u_0\|_{L^1(M)}^{2s/n}}{t}
\right)
\end{equation}
for every $0<t\le c_K\frac{\|u_0\|_{L^1(M)}}{\Phi(\|u_0\|_{L^1(M)})}$,
whereas
\begin{equation}\label{eq:general-L1-long}
\|u(t)\|_{L^\infty(K)}
\le
\Phi^{-1}\left(
\frac{C_K\|u_0\|_{L^1(M)}}{t}
\right)
\end{equation}
for every $t\ge c_K\frac{\|u_0\|_{L^1(M)}}{\Phi(\|u_0\|_{L^1(M)})}$.

Likewise, if \(0\not\equiv u_0\in L^1_{G_M^s,o}(M)\), then for every nonempty
compact set \(K\Subset M\) there exist \(C_{K,o},c_{K,o}>0\) such that
\begin{equation}\label{eq:general-G-short}
\|u(t)\|_{L^\infty(K)}
\le
H_\Phi^{-1}\left(
\frac{C_{K,o}\|u_0\|_{L^1_{G_M^s,o}(M)}^{2s/n}}{t}
\right)
\end{equation}
for every $0<t\le c_{K,o}
\frac{\|u_0\|_{L^1_{G_M^s,o}(M)}}
{\Phi(\|u_0\|_{L^1_{G_M^s,o}(M)})}$,
and
\begin{equation}\label{eq:general-G-long}
\|u(t)\|_{L^\infty(K)}
\le
\Phi^{-1}\left(
\frac{C_{K,o}\|u_0\|_{L^1_{G_M^s,o}(M)}}{t}
\right)
\end{equation}
for every $t\ge
c_{K,o} \frac{\|u_0\|_{L^1_{G_M^s,o}(M)}}
{\Phi(\|u_0\|_{L^1_{G_M^s,o}(M)})}$.

Under Assumptions \ref{A1} and \ref{A2}, the preceding arguments are uniform with respect
to the pole. Hence \eqref{eq:general-L1-short}--\eqref{eq:general-L1-long} hold
globally with \(L^\infty(K)\) replaced by \(L^\infty(M)\), while in
\eqref{eq:general-G-short}--\eqref{eq:general-G-long} one assumes
\(u_0\in L^1_{G_M^s}(M)\), replaces the fixed-pole norm by
\(\|u_0\|_{L^1_{G_M^s}(M)}\), and again replaces \(L^\infty(K)\) by
\(L^\infty(M)\).

For \(\Phi(r)=|r|^{m-1}r\), \(m>1\), one may take
\[
\mu_0=\mu_1=\frac{m-1}{m},
\qquad
\Phi^{-1}(r)=r^{1/m},
\qquad
H_\Phi^{-1}(r)=r^{\frac{n}{n(m-1)+2s}}.
\]
Thus \eqref{eq:general-L1-short}--\eqref{eq:general-G-long} recover exactly the
short-time and basic long-time estimates of Theorems~\ref{thm:smoothing}
and~\ref{thm:smoothing2} and Corollaries~\ref{cor:global-smoothing-L1}
and~\ref{cor:global-smoothing-weighted}.

\appendix

\section{Proofs of Some Technical Results}\label{proofs}

\noindent\bf Proof of Lemma \ref{rmk:contloc}\rm.
Fix \(T>0\) and let \(V\Subset M\) be compact. Choose \(x_0\in M\) and
\(\sigma>0\) such that \(V\subset B_\sigma(x_0)\), and set
\[
\Psi_V(y):=(-\Delta)^{-s}\chi_V(y)
=\int_V G_M^s(y,z)\,d\mu(z).
\]
By Lemma~\ref{lemlastlem}, since \(\chi_V\in L^\infty_c(M)\), one has
\[
\Psi_V\in L^\infty(M).
\]
The upper estimate in Proposition~\ref{thm:3.4}, together with
Remark~\ref{rmkprop3.4} when \(\sigma>1\), yields
\[
\Psi_V(y)\leq C(V,x_0)G_M^s(y,x_0)
\qquad\text{for every }y\in M\setminus\{x_0\}.
\]

We claim that, for every \(w\in L^1_{G_M^s,x_0}(M)\),
\[
\|(-\Delta)^{-s}w\|_{L^1(V)}
\leq C(V,x_0)\|w\|_{L^1_{G_M^s,x_0}(M)}.
\]
Indeed, by Tonelli's theorem,
\begin{align*}
\int_V\int_M &G_M^s(z,y)|w(y)|\,d\mu(y)\,d\mu(z)\\&=\int_M |w(y)|\Psi_V(y)\,d\mu(y)\\
&\leq \|\Psi_V\|_{L^\infty(M)} \int_{B_1(x_0)}|w(y)|\,d\mu(y)+C(V,x_0)
\int_{M\setminus B_1(x_0)} |w(y)|G_M^s(y,x_0)\,d\mu(y)\\
&\leq
C(V,x_0)\|w\|_{L^1_{G_M^s,x_0}(M)}<\infty.
\end{align*}
It follows that \((-\Delta)^{-s}w\) is well-defined by absolute convergence for almost every
point of \(V\), belongs to \(L^1(V)\), and satisfies the claimed estimate.

Since \(u\) is a weak dual solution, Definition~\ref{defn:wds} gives
\[
u\in C\bigl([0,T];L^1_{G_M^s,x_0}(M)\bigr).
\]
Applying the preceding estimate to \(w=u(t)-u(\tau)\), with $t,\tau\in[0,T]$, we obtain
\begin{equation*}
\|(-\Delta)^{-s}u(t)-(-\Delta)^{-s}u(\tau)\|_{L^1(V)}\leq C(V,x_0) \|u(t)-u(\tau)\|_{L^1_{G_M^s,x_0}(M)} \longrightarrow 0
\end{equation*}
as \(t\to\tau\). Since \(V\Subset M\) and \(T>0\) are arbitrary, the
conclusion follows.
\qed

\medskip

\noindent \bf Proof of Proposition \ref{general}\rm. Take first $x=o$. The validity of a stronger analogue of \eqref{fcondition1}, with $C_d$ replaced by 1, for any $R_1, R_2$ s.t. $1\leq R_1 \leq R_2$, is then trivial. If $x\not=o$, letting $R_0(x) = 2d(x,o) + 1$ as in the statement, we note that for all $R \geq d(x,o)$ one has:
\begin{equation*}
    B_{R-d(x,o)}(o)\subset B_R(x) \quad \text{and} \quad B_{R-d(x,o)}(x)\subset B_R(o)
\end{equation*}
Now, let $R_2 \geq R_1 \geq  2d(x,o) + 1$. Then $R_i - d(x,o) \geq R_i/2$ for $i = 1, 2$. Using the Doubling Property \eqref{rmk:doubling}, we obtain, using \eqref{fcondition1} for $x=o$:
\begin{equation*}
\begin{split}
    \frac{R_2^{2s-1}f(R_2)}{\mathrm{Vol}(B_{R_2}(x))} & \leq  \frac{R_2^{2s-1}f(R_2)}{\mathrm{Vol}(B_{R_2-d(x,o)}(o))} \leq \frac{R_2^{2s-1}f(R_2)}{\mathrm{Vol}(B_{R_2/2}(o))}\\
    & \leq C_d \frac{R_2^{2s-1}f(R_2)}{\mathrm{Vol}(B_{R_2}(o))}
    \leq  C_d \frac{R_1^{2s-1}f(R_1)}{\mathrm{Vol}(B_{R_1}(o))}\\
    & \leq C_d \frac{R_1^{2s-1}f(R_1)}{\mathrm{Vol}(B_{R_1-d(x,o)})(x)}
     \leq  C_d \frac{R_1^{2s-1}f(R_1)}{\mathrm{Vol}(B_{R_1/2}(x))}\\
     &  \leq  C_d^2 \frac{R_1^{2s-1}f(R_1)}{\mathrm{Vol}(B_{R_1}(x))},\\
\end{split}
\end{equation*}
which is the desired inequality at $x\in M$, for every $R_2\geq R_1\geq R_0(x)$. 
In the fourth inequality, we have used the mentioned analogue of \eqref{fcondition1}, with $C_d$ replaced by 1, for any $R_1, R_2$ s.t. $1\leq R_1 \leq R_2$, when $x=o$.

Finally, \eqref{betacondition1} immediately holds by Theorem \ref{thm:charnonpar} and the definition of $f$. \qed

\medskip

\noindent \bf Proof of Proposition \ref{prop:properembedding}\rm.
Fix $x_0\in M$ and set
\[
\psi_{x_0}
:=
\frac{\mathbf 1_{B_{1/2}(x_0)}}
{\operatorname{Vol}(B_{1/2}(x_0))}.
\]
By Tonelli's theorem
\begin{align*}
\int_{B_{1/2}(x_0)}
(-\Delta)^{-s}|v|(x)\,d\mu(x)
&=
\int_M |v(y)|
\left(
\int_{B_{1/2}(x_0)}
G_M^s(x,y)\,d\mu(x)
\right)d\mu(y)
\\
&=
\operatorname{Vol}(B_{1/2}(x_0))
\int_M
|v(y)|(-\Delta)^{-s}\psi_{x_0}(y)\,d\mu(y).
\end{align*}
Hence Proposition~\ref{thm:lem3.6} gives
\[
v\in L^1_{G_M^s,x_0}(M)
\quad\Longleftrightarrow\quad
(-\Delta)^{-s}|v|
\in L^1(B_{1/2}(x_0)).
\]
If $v\in L^1_{G_M^s,o}(M)$,
Proposition~\ref{prop:compact-poles-equivalence} implies that
$v\in L^1_{G_M^s,x_0}(M)$ for every $x_0\in M$. A finite covering of
any compact subset of $M$ by balls of radius $1/2$ therefore yields
\[
(-\Delta)^{-s}|v|\in L^1_{\mathrm{loc}}(M).
\]
Conversely, if $(-\Delta)^{-s}|v|\in L^1_{\mathrm{loc}}(M)$,
then in particular
\[
(-\Delta)^{-s}|v|\in L^1(B_{1/2}(o)),
\]
and the preceding equivalence with $x_0=o$ gives
$v\in L^1_{G_M^s,o}(M)$.

We next prove the strict inclusions. By the upper Green-function estimate,
for every $x\in M\setminus B_1(o)$,
\[
G_M^s(x,o)
\leq
C_2
\int_{d(x,o)}^\infty
\frac{t^{2s-1}}{\operatorname{Vol}(B_t(o))}\,dt
\leq
C_2
\int_1^\infty
\frac{t^{2s-1}}{\operatorname{Vol}(B_t(o))}\,dt
=:C_o<\infty,
\]
where the finiteness follows from $s$-nonparabolicity. Hence, for every
$v\in L^1(M)$,
\[
\|v\|_{L^1_{G_M^s,o}(M)}
\leq
(1+C_o)\|v\|_{L^1(M)},
\]
so that $L^1(M)\hookrightarrow L^1_{G_M^s,o}(M)$. In addition, $L^1_{G_M^s}(M)\hookrightarrow L^1_{G_M^s,o}(M)$ is immediate from the definition of $L^1_{G_M^s}(M)$.
Since
\[
\int_R^\infty
\frac{t^{2s-1}}{\operatorname{Vol}(B_t(o))}\,dt
\longrightarrow0
\qquad\text{as }R\to\infty,
\]
we can choose inductively a sequence
$\{o_j\}_{j\geq1}\subset M$ such that
\[
d(o_1,o)>2,
\qquad
d(o_j,o)>d(o_{j-1},o)+3
\quad\text{for every }j\geq2,
\]
and
\[
C_2
\int_{d(o_j,o)-1}^\infty
\frac{t^{2s-1}}{\operatorname{Vol}(B_t(o))}\,dt
\leq
\frac{2^{-j}}{j}
\qquad\text{for every }j\geq1.
\]
Then the balls $B_1(o_j)$ are pairwise disjoint and contained in
$M\setminus B_1(o)$. Furthermore, if $x\in B_1(o_j)$, then
$d(x,o)\geq d(o_j,o)-1$, and therefore
\[
G_M^s(x,o)
\leq
C_2
\int_{d(x,o)}^\infty
\frac{t^{2s-1}}{\operatorname{Vol}(B_t(o))}\,dt
\leq
\frac{2^{-j}}{j}.
\]
Define
\[
v(x)
:=
\sum_{j=1}^\infty
\frac{j}{\operatorname{Vol}(B_1(o_j))}
\mathbf 1_{B_1(o_j)}(x).
\]
Since the balls $B_1(o_j)$ are pairwise disjoint and contained in
$M\setminus B_1(o)$,
\begin{equation*}
\|v\|_{L^1_{G_M^s,o}(M)}
=
\sum_{j=1}^\infty
\frac{j}{\operatorname{Vol}(B_1(o_j))}
\int_{B_1(o_j)}
G_M^s(x,o)\,d\mu(x)\leq
\sum_{j=1}^\infty 2^{-j}
<\infty.
\end{equation*}
Thus $v\in L^1_{G_M^s,o}(M)$. On the other hand,
\[
\|v\|_{L^1(M)}
=
\sum_{j=1}^\infty j
=
\infty,
\]
while, for every $j\geq1$,
\[
\|v\|_{L^1_{G_M^s,o_j}(M)}
\geq
\int_{B_1(o_j)}v\,d\mu
=
j.
\]
Therefore
\[
\sup_{x_0\in M}
\|v\|_{L^1_{G_M^s,x_0}(M)}
=
\infty,
\]
so that $v\notin L^1_{G_M^s}(M)$.
Consequently,
\[
v\in
L^1_{G_M^s,o}(M)
\setminus
\bigl(L^1(M)\cup L^1_{G_M^s}(M)\bigr),
\]
and both inclusions are strict.
\qed

\begin{rmk}\label{rmk:appendix}
Under Assumptions \ref{A1} and \ref{A2}, one has
\[
L^1(M)\subsetneq L^1_{G_M^s}(M).
\]
Indeed, the continuous inclusion follows from
Lemma~\ref{thm:lem1}(iii). For the strictness, we adapt the
construction in the proof of Proposition~\ref{prop:properembedding} to obtain a function
$v\in L^1_{G_M^s}(M)\setminus L^1(M)$.
Since the radius \(R_0\) in Assumption \ref{A2} is independent of the
pole, Lemma~\ref{thm:lem1}(ii), together with Assumption \ref{A1}, yields a
constant \(C>0\), independent of \(x_0,z\in M\), such that
\[
\frac{1}{\operatorname{Vol}(B_1(z))}
\int_{B_1(z)}G_M^s(x,x_0)\,d\mu(x)
\leq
C\int_{d(z,x_0)-1}^{\infty}\frac{dt}{f(t)}
\]
whenever \(d(z,x_0)\geq R_0+1\). Choose a monotone sequence \(R_j\to\infty\), with $R_1\geq R_0+1$,
and points \(o_j\in M\) such that the balls \(B_{R_j}(o_j)\) are
pairwise disjoint and
\[
C\int_{R_j-1}^{\infty}\frac{dt}{f(t)}
\leq 2^{-j}.
\]
Define
\[
v(x):=
\sum_{j=1}^{\infty}
\frac{\mathbf{1}_{B_1(o_j)}(x)}
{\operatorname{Vol}(B_1(o_j))}.
\]
Then \(v\notin L^1(M)\). For every \(x_0\in M\), at most one of the
balls \(B_{R_j}(o_j)\) contains \(x_0\), while the contribution of
every other \(B_1(o_j)\) to the Green-weighted integral is bounded by
\(2^{-j}\). Furthermore, since \(R_j\geq 2\), at most one ball
\(B_1(o_j)\) intersects \(B_1(x_0)\), so the local contribution is
bounded by \(1\). For the possible index \(j_0\) such that
\(x_0\in B_{R_{j_0}}(o_{j_0})\), Lemma~\ref{thm:lem1}(iii) and the
fact that \(R_0\) is independent of the pole yield
\[
\frac{1}{\operatorname{Vol}(B_1(o_{j_0}))}
\int_{B_1(o_{j_0})\setminus B_1(x_0)}
G_M^s(x,x_0)\,d\mu(x)\leq C,
\]
where \(C\) is independent of \(x_0\) and \(j_0\). Therefore,
\[
\sup_{x_0\in M}
\|v\|_{L^1_{G_M^s,x_0}(M)}<\infty.
\]
Thus \(v\in L^1_{G_M^s}(M)\setminus L^1(M)\), and the inclusion is
strict.
\end{rmk}

\noindent \bf Proof of Lemma \ref{thm:lem1}\rm.
We use the data \(f,\gamma,\beta\), and the function \(R_0(\cdot)\) fixed at the
beginning of Section~\ref{S7}.

The lower estimate in Theorem~\ref{thm:bounGreen's} and the
Bishop--Gromov upper bound give, for every \(x\neq x_0\),
\begin{equation*}
G_M^s(x,x_0)\geq
C_1\int_{d(x,x_0)}^{+\infty}
\frac{t^{2s-1}}{\operatorname{Vol}(B_t(x))}\,dt
\geq
\frac{C_1}{\omega_n}
\int_{d(x,x_0)}^{+\infty}t^{-n-1+2s}\,dt
=
\frac{C_1}{\omega_n(n-2s)}
\frac{1}{d(x,x_0)^{n-2s}},
\end{equation*}
which proves part~(i).

Let \(d(x,x_0)\geq R_0(x_0)\). By the
upper estimate in Theorem~\ref{thm:bounGreen's} and Condition \ref{B}, see Proposition \ref{general},
\begin{align*}
G_M^s(x,x_0)
&\leq
C_2\int_{d(x,x_0)}^{+\infty}
\frac{t^{2s-1}}{\operatorname{Vol}(B_t(x_0))}\,dt
=
C_2\int_{d(x,x_0)}^{+\infty}
\frac{t^{2s-1}f(t)}{\operatorname{Vol}(B_t(x_0))}
\frac{dt}{f(t)}
\\
&\leq
C_2\gamma
\frac{d(x,x_0)^{2s-1}f(d(x,x_0))}
{\operatorname{Vol}(B_{d(x,x_0)}(x_0))}
\int_{d(x,x_0)}^{+\infty}\frac{dt}{f(t)}.
\end{align*}
This proves part~(ii).

If \(d(x,x_0)\geq1\), the upper estimate in
Theorem~\ref{thm:bounGreen's} yields
\begin{align*}
G_M^s(x,x_0)
&\leq
C_2\int_1^{R_0(x_0)}
\frac{t^{2s-1}}{\operatorname{Vol}(B_t(x_0))}\,dt
+
C_2\int_{R_0(x_0)}^{+\infty}
\frac{t^{2s-1}}{\operatorname{Vol}(B_t(x_0))}\,dt
\\
&\leq
C_2\frac{R_0(x_0)^{2s}}
{2s\,\operatorname{Vol}(B_1(x_0))}
+C_2\gamma
\frac{R_0(x_0)^{2s-1}f(R_0(x_0))}
{\operatorname{Vol}(B_{R_0(x_0)}(x_0))}
\int_{R_0(x_0)}^{+\infty}\frac{dt}{f(t)}
\\
&\leq
C_2\left(
\frac{R_0(x_0)^{2s}}
{2s\,\operatorname{Vol}(B_1(x_0))}
+
\frac{\beta\gamma R_0(x_0)^{2s-1}f(R_0(x_0))}
{\operatorname{Vol}(B_{R_0(x_0)}(x_0))}
\right),
\end{align*}
which proves part~(iii).

Let \(x\neq x_0\) and
\(R\geq\max\{d(x,x_0),R_0(x_0)\}\). By the upper estimate
in Theorem~\ref{thm:bounGreen's},
\begin{align*}
G_M^s(x,x_0)
&\leq
C_2\int_{d(x,x_0)}^R
\frac{t^{2s-1}}{\operatorname{Vol}(B_t(x_0))}\,dt
+
C_2\int_R^{+\infty}
\frac{t^{2s-1}}{\operatorname{Vol}(B_t(x_0))}\,dt.
\end{align*}
For \(0<t\leq R\), Bishop--Gromov comparison gives
\[
\operatorname{Vol}(B_t(x_0))
\geq
\left(\frac{t}{R}\right)^n\operatorname{Vol}(B_R(x_0)).
\]
Consequently,
\begin{align*}
\int_{d(x,x_0)}^R
\frac{t^{2s-1}}{\operatorname{Vol}(B_t(x_0))}\,dt
&\leq
\frac{R^n}{\operatorname{Vol}(B_R(x_0))}
\int_{d(x,x_0)}^Rt^{-n-1+2s}\,dt
\\
&\leq
\frac{R^n}{(n-2s)\operatorname{Vol}(B_R(x_0))}
\frac{1}{d(x,x_0)^{n-2s}}.
\end{align*}
Condition \ref{B}  gives
\begin{align*}
\int_R^{+\infty}
\frac{t^{2s-1}}{\operatorname{Vol}(B_t(x_0))}\,dt
&\leq
\gamma
\frac{R^{2s-1}f(R)}{\operatorname{Vol}(B_R(x_0))}
\int_R^{+\infty}\frac{dt}{f(t)}
\leq
\frac{\beta\gamma f(R)R^{n-1}}
{\operatorname{Vol}(B_R(x_0))}
\frac{1}{d(x,x_0)^{n-2s}},
\end{align*}
where the last inequality follows from \(d(x,x_0)\leq R\). Combining
these estimates proves part~(iv).

Part~(v) follows from part~(iv), with \(R=R_0(x_0)\). Indeed, if
\(0<\rho\leq R_0(x_0)\), then
\begin{align*}
\int_{B_\rho(x_0)}G_M^s(x,x_0)\,d\mu(x)
&\leq
C_2\mathcal A_{x_0}(R_0(x_0))
\int_{B_\rho(x_0)}
\frac{d\mu(x)}{d(x,x_0)^{n-2s}}.
\end{align*}
By the coarea formula, integration by parts, and the Bishop--Gromov upper
bound,
\begin{align*}
\int_{B_\rho(x_0)}
\frac{d\mu(x)}{d(x,x_0)^{n-2s}}
&=
\frac{\operatorname{Vol}(B_\rho(x_0))}{\rho^{n-2s}}
+
(n-2s)\int_0^\rho
\frac{\operatorname{Vol}(B_r(x_0))}{r^{n+1-2s}}\,dr
\\
&\leq
\omega_n\rho^{2s}
+
(n-2s)\omega_n\int_0^\rho r^{2s-1}\,dr=
\omega_n\frac{n}{2s}\rho^{2s}.
\end{align*}
This proves part~(v).

Finally, let \(R\geq R_0(x_0)\). By symmetry, the upper estimate in
Theorem~\ref{thm:bounGreen's}, Tonelli's theorem, the coarea formula, and
integration by parts,
\begin{align*}
\int_{B_R(x_0)}G_M^s(x,x_0)\,d\mu(x)
&\leq
C_2\int_{B_R(x_0)}
\int_{d(x,x_0)}^{+\infty}
\frac{t^{2s-1}}{\operatorname{Vol}(B_t(x_0))}
\,dt\,d\mu(x)
\\
&=
C_2\operatorname{Vol}(B_R(x_0))
\int_R^{+\infty}
\frac{t^{2s-1}}{\operatorname{Vol}(B_t(x_0))}\,dt
+
C_2\int_0^Rr^{2s-1}\,dr
\\
&\leq
C_2\gamma R^{2s-1}f(R)
\int_R^{+\infty}\frac{dt}{f(t)}
+
C_2\frac{R^{2s}}{2s}
\\
&\leq
C_2\max\{\gamma,1\}h(R).
\end{align*}
This proves part~(vi) and completes the proof.
\qed

\par\bigskip\noindent
\textbf{Acknowledgments.}
The authors are members of the Gruppo Nazionale per l'Analisi Matematica, la Probabilit\`a e le loro Applicazioni (GNAMPA, Italy) of the Istituto Nazionale di Alta Matematica (INdAM, Italy). The first author is partially supported by the PRIN project 2022 ''Partial differential equations and related geometric-functional inequalities", ref. 20229M52AS. The second author is partially supported by the PRIN projects 2022 Geometric-analytic methods for PDEs
and applications, ref. 2022SLTHCE. Both PRIN projects above are
financially supported by the EU, in the framework of the "Next Generation EU initiative". The authors acknowledge the use of AI tools during the exploratory stage of this project. All mathematical arguments and proofs in the final manuscript were checked and written by the authors.


\end{document}